\RequirePackage{fix-cm}
\documentclass[envcountsame,smallextended]{svjour3}       % onecolumn (second format)
\smartqed  % flush right qed marks, e.g. at end of proof
\usepackage{amsthm}
\usepackage{amssymb}
\usepackage{graphicx}
\usepackage[top=35mm, bottom=25mm, left=20mm, right=20mm]{geometry}
\usepackage[numbers,sort&compress]{natbib}

\usepackage{amsmath}
\usepackage{amsfonts}
\usepackage{amssymb}
\usepackage{enumitem}
\usepackage{graphicx}
\usepackage{subcaption}
\usepackage{tikz}
\usetikzlibrary{arrows.meta, positioning}
\usepackage{pgfplots}
\usetikzlibrary{patterns}
\usepackage{multicol}
\usepackage{hyperref}
\usepackage{xcolor}
\usepackage{cancel}
\usepackage{ulem}
\usepackage[active]{srcltx}

\numberwithin{equation}{section}

\usepackage{algorithm}%
\usepackage{algorithmicx}%
\usepackage{algpseudocode}%

\usepackage{booktabs}
\usepackage{siunitx}

\newcommand\R{\mathbb{R}}
\newcommand\Rinf{\overline{\mathbb{R}}}
\newcommand\inter[1]{ {\rm \textbf{int}}(#1)} %interior
\newcommand\dom[1]{ \bs{{\rm dom}}(#1)} %domain
\newcommand\Dom[1]{ \bs{{\rm Dom}}(#1)} %domain set-valued
\newcommand\dist{ \bs{{\rm dist}}} %domain set-valued
\newcommand\gf{\varphi} %general function in results
\newcommand\gh{\psi} %general function in preliminaries
\newcommand\fgam[3]{#1_{#3}^{#2}}

\newcommand\prox[3]{ \bs{{\rm prox}}_{#2#1}^{#3}}

\newcommand\ov[1]{\overline{#1}}
\newcommand\mb{\mathbf{B}}
\newcommand\bs[1]{\boldsymbol{#1}}
\newcommand\argmin[1]{\bs{\arg\min}_{#1}}
\newcommand\argmint[1]{\mathop{\bs{\arg\min}}\limits_{#1}}
\newcommand\Nz{\mathbb{N}_0}

\journalname{Set-Valued and Variational Analysis}

\begin{document}

\title{
Moreau envelope and proximal-point methods under the lens of high-order regularization 
}


\author{Alireza Kabgani         \and
        Masoud Ahookhosh %etc.
}


\institute{Corresponding author: Alireza Kabgani \\
A. Kabgani, M. Ahookhosh \at
              Department of Mathematics, University of Antwerp, Antwerp, Belgium. \\
              \email{alireza.kabgani@uantwerp.be, masoud.ahookhosh@uantwerp.be}             \\
              The Research Foundation Flanders (FWO) research project G081222N and UA BOF DocPRO4 projects with ID 46929 and 48996 partially supported the paper's authors.
}

\date{Received: date / Accepted: date}

\maketitle

\begin{abstract}
This paper is devoted to investigating the fundamental properties of the high-order proximal operator (HOPE) and the high-order Moreau envelope (HOME) in the nonconvex setting, where the quadratic regularization ($p=2$) is
replaced by a $p$-order regularizer with $p > 1$.
After establishing several basic properties of HOPE and HOME, we study the differentiability and weak smoothness of HOME under $q$-prox-regularity with $q \geq 2$ and $p$-calmness for $p \in (1,2]$ and $2 \leq p \leq q$. Furthermore, we propose a high-order proximal-point algorithm (HiPPA) and analyze the convergence of the generated sequence to proximal fixed points. Our results pave the way for the development of a high-order smoothing theory with $p>1$ that can lead to new algorithmic advances in the nonconvex setting. To illustrate this potential for nonsmooth and nonconvex optimization, we apply HiPPA to the Nesterov–Chebyshev–Rosenbrock functions.

\end{abstract}

\vspace{-4mm}
%%%%%%%%%%%%%%%%%%%%%%%%%%%%%%%%%%%%%%%%%%%%%%%%%%%%%%%%%%%%%%%%%%%%%%%%%%%
%%%%%%%%%%%%%%%%%%%%%%%%%%%%%%%%%%%%%%%%%%%%%%%%%%%%%%%%%%%%%%%%%%%%%%%%%%%
\section{Introduction}\label{intro}
In the context of nonsmooth and nonconvex optimization, the {\it smoothing paradigm} leads to strong theoretical and computational tools that can be efficiently used to tailor fast iterative schemes that commonly outperform classical {\it subgradient-based} methods.
Due to the superior numerical efficiency of {\it gradient-based} schemes over subgradient methods, {\it smoothing techniques} have attracted considerable attention in recent decades; see, e.g. \cite{Ahookhosh21,Beck12,ben2006smoothing,bertsekas2009nondifferentiable,Bot15,Bot2020,Moreau65,nesterov2005smooth,patrinos2013proximal,Stella17,Themelis18, themelis2019acceleration,themelis2020douglas}.
Among these, approaches based on {\it convolution techniques}, such as {\it infimal-convolution} \cite{Bauschke17,Beck12,Burke2017} and {\it integral-convolution} \cite{Burke2013Gradient,Chen2012}, have been particularly influential.

Notably, the {\it Moreau envelope} (also referred to as the {\it Moreau-Yosida regularization}) is arguably the most widely studied technique, as it provides a smooth approximation to a given function while preserving the same set of minimizers. Originally introduced by Jean-Jacques Moreau \cite{Moreau65}, this concept has become a cornerstone of nonsmooth optimization theory and algorithms due to its favorable properties in both convex (e.g., \cite{Beck12, Bot15, Drusvyatskiy19, Ghaderi24, Parikh14, Shefi16}) and nonconvex (e.g., \cite{Burke13,Kabganitechadaptive,Kabgani24itsopt,KecisThibault15,Poliquin96,Rockafellar09}) settings.

For a proper and lower semicontinuous (lsc) function $\gf: \R^n\to\Rinf:=\R\cup \{+\infty\}$ and a parameter $\gamma>0$, the {\it Moreau envelope} \cite{Moreau65} is defined as
\begin{equation}\label{eq:mor}
    \fgam{\gf}{}{\gamma}(x):=\mathop{\bs{\inf}}\limits_{y\in \R^n} \left(\gf(y)+\frac{1}{2\gamma}\Vert x- y\Vert^2\right).
\end{equation}
This construction yields a smooth approximation of the original cost function $\varphi$ under some suitable assumptions, such as {\it prox-regularity} and {\it calmness}; see, e.g., \cite{Poliquin96,Rockafellar09}. It thereby provides a more tractable optimization framework.
In particular, if $\gf$ is convex, then $\gf_\gamma$ not only preserves convexity but is also $\gamma^{-1}$-smooth; that is, it is Fr\'{e}chet differentiable with a Lipschitz continuous gradient of constant $\gamma^{-1}$. 
In addition, the set of minimizers for the original cost $\gf$ and its envelope $\gf_\gamma$ coincides, i.e.,
$ \argmin{x\in \R^n} \gf(x)=\argmin{x\in \R^n} \gf_\gamma (x)$, as established in \cite[Chapter~12]{Bauschke17} (see, e.g., Propositions 12.15, 12.29, and 12.30 therein).
 
A closely related concept is the {\it proximal-point operator} \cite{Moreau65}, which plays a fundamental role in the definition of the Moreau envelope and in the development of proximal-based iterative schemes for minimizing the cost function $\gf$. It is defined as
\begin{equation}\label{eq:prox}
\prox{\gf}{\gamma}{} (x):=\argmint{y\in \R^n} \left(\gf(y)+\frac{1}{2\gamma}\Vert x- y\Vert^2\right).
\end{equation}
The proximal operator enjoys several favorable properties in the convex setting that contribute to its wide use in algorithm design:
(i) $\prox{\gf}{\gamma}{}$ is single-valued at each $x\in \R^n$; (ii) it is $1$-Lipschitz continuous; (iii) it satisfies the identity
\[\argmint{x\in \R^n} \gf(x)=\bs{\rm Fix}(\prox{\gf}{\gamma}{}),\] 
where the latter set denotes the set of fixed points of $\prox{\gf}{\gamma}{}$; see, e.g.,  \cite[Chapter 12]{Bauschke17}. 

The generic {\it proximal-point method} is given by $x^{k+1}=\prox{\gf}{\gamma}{} (x^k)$, satisfying $x^k\to x^*\in \argmin{x\in \R^n} \gf(x)$ under convexity assumptions. Moreover, for convex $\gf$, it holds that
\begin{equation}\label{eq:relmorprox}
\prox{\gf}{\gamma}{} (x^k) = x^k - \gamma \nabla \gf_\gamma (x^k).
\end{equation}
implying that the proximal method can be interpreted as a gradient descent scheme with constant step-size $\gamma$; cf. \cite{beck2017first,Hoheisel2020regularization,Parikh14}.
Thanks to their simple structure and low memory requirements, proximal-point methods have received significant attention following the seminal works by Martinet \cite{martinet1970breve,martinet1972determination}; see, e.g., \cite{Ahookhosh24,Guler92,Kim21,Nesterov2023a,Parikh14,Salzo12} and references therein. 
In the nonconvex setting, however, the fundamental properties of the proximal operator and the Moreau envelope require more delicate assumptions, such as prox-regularity and calmness of the cost function in local or global regimes; we refer the reader to \cite{Poliquin96,Rockafellar09} for a detailed account.

Recent studies of proximal-point methods and the Moreau envelope have uncovered the further potential of these approaches in the presence of high-order regularization terms; see, e.g., \cite{Ahookhosh24,Ahookhosh23,Ahookhosh2025,Kabgani24itsopt,Kabganitechadaptive,KecisThibault15,Nesterov2022,Nesterov2023a,zhu2024global}. In these developments, as opposed to the classical case, the quadratic regularization in \eqref{eq:mor} and \eqref{eq:prox} is replaced by the term $\tfrac{1}{p} \|x-y\|^p$ for $p>1$, yielding a significantly more flexible framework.
In particular, the theoretical groundwork has been established in \cite{KecisThibault15}, while \cite{Kabgani24itsopt} introduces an {\it inexact two-level smoothing optimization framework} (ItsOPT) for general nonsmooth and nonconvex optimization problems.
This framework consists of two levels: (i) at the lower level, the high-order proximal auxiliary problems are solved inexactly to produce an inexact oracle for HOME; (ii) at the upper level, an inexact zero-, first- or second-order method is developed to minimize HOME. Furthermore, the framework has been adapted for solving nonsmooth weakly convex optimization problems in \cite{Kabganitechadaptive}.
The central role of the basic properties of HOME, including its differentiability and weak smoothness, in these methodologies motivates the present study of the fundamental and differential properties of HOME in the nonconvex setting.

%%%%%%%%%%%%%%%%%%%%%%%%%%%%%%%%%%%%%%%%%%%%%%%%%%%%%%%%%%%%%%%%%%%%%%%%%%%%%%%%%%%
\subsection{{\bf Contribution}} \label{sec:contributions}
Our contributions are threefold:
\begin{description}[wide, labelwidth=!, labelindent=0pt]
\item[{\bf (i)}] {\bf Fundamental properties of HOPE and HOME.}
We derive several fundamental properties of the high-order Moreau envelope (HOME) and the corresponding proximal operator (HOPE), including {\it coercivity} and {\it sublevel set relationships between $\varphi$ and $\varphi_\gamma$} (cf., Propositions~\ref{lem:hiordermor:coer}~and~\ref{pro:rel:sublevel}, Corollary~\ref{cor:rel:sublevel}), facilitating the design of algorithms on $\varphi_\gamma$. Furthermore, we introduce the notion of {\it $p$-calmness}, a key condition for establishing the differentiability of HOME, and characterize its relationships with classical reference points (cf., Theorem~\ref{lem:progpcalm}).

\item[{\bf (ii)}] {\bf Differentiability and weak smoothness of HOME.} We conduct a comprehensive analysis of the differentiability and weak smoothness properties of HOME for nonsmooth and nonconvex functions under $p$-calmness and $q$-prox-regularity assumptions (see Definitions~\ref{def:critic}~and~\ref{def:pprox-regular}), with $p>1$.
Our results demonstrate that HOME is continuously differentiable when either $q \geq 2$ and $p \in (1,2]$ or $2 \leq p \leq q$ (cf. Theorems~\ref{th:dif:proxreg12}~and~\ref{th:dif:proxreg}), and weakly smooth with H\"{o}lder-continuous gradients under broader conditions (cf. Theorems~\ref{th:dif:proxreg12:weaksm}~and~\ref{th:dif:proxreg:weak}). However, for $q < p$ the differential properties of HOME remain open.
Although the case $q \in (1, 2)$ is particularly interesting, its study lies beyond the scope of this paper and will be addressed in future work (see Remark~\ref{rem:remonqprox}~\ref{rem:remonqprox:d}).
The interplay between $p$ (in $p$-calmness) and $q$ (in $q$-prox-regularity) is summarized in Subfigure~(a) of Figure~\ref{fig:qpdiff}. 
Furthermore, the relationships among reference points of $\varphi$ and $\varphi_\gamma$ are clarified (cf. Fact~\ref{prop:relcrit}, Corollary~\ref{prop:relcrit2}, and Remark~\ref{rem:revofrel}.), as emphasized in Subfigure~(b) of Figure~\ref{fig:qpdiff}.

\item[{\bf (iii)}] {\bf The high-order proximal-point algorithm (HiPPA).}
The HiPPA algorithm is introduced and a simple convergence analysis toward proximal fixed points, using properties of HOME, is studied (cf. Theorem~\ref{th:HiPPA}). Preliminary numerical tests of HiPPA on Nesterov-Chebyshev-Rosenbrock functions demonstrate its promising potential for solving nonsmooth and nonconvex optimization problems.
\end{description}

%%%%%%%%%%%%%%%%%%%%%%%%%%%%%%%%%%%%%%%%%%%%%%%%%%%%%%%%%%%%%%%%%%%%%%%%%%%%

%%%%%%%%%%%%%%%%%%%%%%%%%%%%%%%%%%%%%%%%%%%%%%%%%%%%%%%%%%%%%%%%%%%%%%%%%%%%
\begin{figure}[H]
    \centering
    \subfloat[Differentiability of HOME]{
        \begin{tikzpicture}
            \begin{axis}[
                xmin=0.98, xmax=4,
                ymin=1.98, ymax=3,
                axis lines=middle,
                xlabel={$q$},
                ylabel={$p$},
                xtick={1, 2,3,4},
                ytick={1,2,3, 4},
                enlargelimits=false,
                width=5cm,
                height=5cm,
                axis equal,
                scale=1.5,
                xlabel style={at={(ticklabel cs:1)}, anchor=west, yshift=4ex},
                ylabel style={at={(ticklabel cs:1)}, anchor=south,xshift=4ex}
            ]

            \addplot[domain=2:4, ultra thick, color=blue] {x};

            \node[rotate=90] at (axis cs:1.5,2.5) {Remark~\ref{rem:remonqprox}~\ref{rem:remonqprox:d}};

            \fill[green!70, opacity=0.7] (axis cs:2,1) -- (axis cs:2,2) -- (axis cs:4,2) -- (axis cs:4,1) -- cycle;
            \node at (axis cs:3,1.6) [anchor=north, black] {Theorem~\ref{th:dif:proxreg12}};
  
            \fill[blue!50, opacity=0.7] (axis cs:2,2) -- (axis cs:4,4) -- (axis cs:4,2) -- cycle;
            \node at (axis cs:3.1,2.6) [anchor=north, black, rotate=45] {Theorem~\ref{th:dif:proxreg}};
    
            \fill[red!70, opacity=0.7] (axis cs:2,2) -- (axis cs:2,4) -- (axis cs:4,4) -- cycle;
            \node at (axis cs:2.6,3.5) [anchor=north, black, rotate=45] {Open};
            \end{axis}
        \end{tikzpicture}
         
     \hspace{15mm}

    }~
     \subfloat[Relationships among reference points]{
      \begin{tikzpicture}[
    node distance=0.5cm,
    box/.style={draw, rounded corners, minimum width=2.5cm, minimum height=1cm, align=center},
    arrow/.style={-{Stealth}}
]

\node[box] (Mcpsi) {$\bs{\rm Mcrit}(\gf)$};
\node[box, right=1.5cm of Mcpsi] (Fcrit) {$\bs{\rm Fcrit}(\gf)$};
\node[box, below=1.25cm of Mcpsi] (argmin) {$\argmint{y\in \R^n} \gf(y)$\\$=$\\$\argmint{y\in \R^n} \fgam{\gf}{p}{\gamma}(y)$};
\node[box, below=1.30cm of Fcrit] (Fin) {$\bs{\rm Fix}(\prox{\gf}{\gamma}{p})$\\$=$\\$\bs{\rm{Zero}}(\nabla\fgam{\gf}{p}{\gamma})$};
\node[box, below=1.25cm of argmin] (Fcrit2) {$\bs{\rm Fcrit}(\fgam{\gf}{p}{\gamma})$};
\node[box, right=1.5cm of Fcrit2] (Mcpsi2) {$\bs{\rm Mcrit}(\fgam{\gf}{p}{\gamma})$};

\draw[arrow] (Fcrit) --(Mcpsi);
%\draw[arrow] (Fin)-- node[midway,left,xshift=0.3cm,yshift=.65cm,rotate=90] {\tiny Th. \ref{prop:relcrit} $\ref{prop:relcrit:opfix}$}(Fcrit) ;
\draw[arrow]  (4.5,-1.81)--node[midway,left,xshift=0.3cm,yshift=.4cm,rotate=90] {\tiny C. \ref{prop:relcrit2}}(4.5,-0.5);
%\draw[arrow]  (argmin)-- (Fin);
\draw[arrow] (1.37,-2.05) --node[midway,above,xshift=0cm,yshift=0cm] {\tiny C. \ref{prop:relcrit2} } (2.76,-2.05);
\draw[arrow] (2.76,-2.65) --node[midway] {\Large $\times$} (1.37,-2.65);
%\draw[arrow] (2.76,-2.65) ---node[midway,below] {\tiny Re. \ref{rem:revofrel} $\ref{rem:revofrel:a}$} (1.37,-2.65);
\draw[arrow] (2.76,-2.65) --node[midway,below,yshift=-0.05cm] {\tiny R. \ref{rem:revofrel} $\ref{rem:revofrel:a}$} (1.37,-2.65);
\draw[arrow] (argmin) -- node[midway,left,xshift=-0.3cm,yshift=.4cm,rotate=90] {\tiny{C. \ref{prop:relcrit2}}}  (Fcrit2);
\draw[arrow] (argmin) --  node[midway,left,xshift=-0.3cm,yshift=.4cm,rotate=90] {\tiny C. \ref{prop:relcrit2}} (Mcpsi);
\draw[arrow] (Fcrit2) -- (Mcpsi2);
\draw[arrow] (Mcpsi2) -- node[midway,left,xshift=0.3cm,yshift=.4cm,rotate=90] {\tiny C. \ref{prop:relcrit2} } (Fin);
\draw[arrow] (3.5,-0.5) --node[midway,left,xshift=-0.3cm,yshift=.65cm,rotate=90] {\tiny R. \ref{rem:revofrel} $\ref{rem:revofrel:b}$} (3.5,-1.81);
\draw[arrow] (3.5,-0.5) --node[midway] {\Large $\times$} (3.5,-1.81);

\end{tikzpicture}
    }
    \caption{(a) Differentiability of HOME under $q$-prox-regularity and $p$-calmness ; (b)  Relationships among reference points of $\gf$ and $\fgam{\gf}{p}{\gamma}$.}
    \label{fig:qpdiff}
\end{figure}

%%%%%%%%%%%%%%%%%%%%%%%%%%%%%%%%%%%%%%%%%%%%%%%%%%%%%%%%%%%%%%%%%%%%%%%%%%%%%%%%%%%%%%%%%%%%%%%%%%%%%%%%%%
\subsection{{\bf Organization}} \label{sec:contributions}
The paper is organized as follows. Section~\ref{sec:preliminaries} introduces the necessary notation and preliminaries. Section~\ref{sec:home} discusses the structural properties of HOPE and HOME. Section~\ref{sec:on diff} investigates the differentiability and weak smoothness of HOME. Section~\ref{sec:applications} presents HiPPA and evaluates its performance on challenging nonsmooth and nonconvex optimization problems.  Finally,
we discuss the implications and limitations of our results in Section~\ref{sec:disc}.

%%%%%%%%%%%%%%%%%%%%%%%%%%%%%%%%%%%%%%%%%%%%%%%%%%%%%%%%%%%%%%%%%%%%%%%%%%%%%%%%%%%%%%%%%%%%%%%%%%%%%%%%%%
%%%%%%%%%%%%%%%%%%%%%%%%%%%%%%%%%%%%%%%%%%%%%%%%%%%%%%%%%%%%%%%%%%%%%%%%%%%%%%%%%%%%%%%%%%%%%%%%%%%%%%%%%%
\section{Preliminaries and notations} \label{sec:preliminaries}
 This section establishes the foundational notation and concepts used throughout the paper.
%%%%%%%%%%%%%%%%%%    n-dimensional Euclidean space
Let $\R^n$ denote the $n$-dimensional \textit{Euclidean space} endowed with the \textit{Euclidean norm} $\Vert\cdot\Vert$, and the standard \textit{inner product}
 $\langle\cdot, \cdot\rangle$.
%%%%%%%%%%%%%%%%%%    Balls in R^n:
We denote the open ball with center $\ov{x}\in \R^n$ and radius $r>0$ as $\mb(\ov{x}; r)$.
%%%%%%%%%%%%%%%%%%    interior:
The \textit{interior} of a set $C\subseteq \R^n$ is denoted by $\inter{C}$.
%%%%%%%%%%%%%%%%%%   distance:
The distance from $x\in \R^n$ to a nonempty set $C\subseteq\R^n$ is defined as 
$\dist(x,C):=\bs\inf_{y\in C}\Vert y - x\Vert$.
%%%%%%%%%%%%%%%%%%    infinity conventions:
We adopt the convention that $\infty - \infty = \infty$.
% $0\cdot\infty=0$, $\frac{1}{0}=\infty$, and $\infty \cdot \infty =\infty$.}

%%%%%%%%%%%%%%%%%%    effective domain:
For $\gh: \R^n \to \Rinf:=\R\cup\{+\infty\}$, the \textit{effective domain} is 
 $\dom{\gh}:= \{x \in   \R^{n}\mid~\gh(x)< + \infty \}$,
%%%%%%%%%%%%%%%%%%    Proper:
and $\gh$ is \textit{proper} if $\dom{\gh}\neq \emptyset$.
%%%%%%%%%%%%%%%%%%    sublevel set:
The \textit{sublevel set} of $\gh$ at height $\lambda \in \R$ is $\mathcal{L}(\gh, \lambda):=\{x\in \R^n\mid \gh(x)\leq \lambda\}$.
%%%%%%%%%%%%%%%%%%    Argmin
The set of minimizers of $\gh$ over $C\subseteq\R^n$ is denoted by $\argmin{x\in C}\gh(x)$.
%%%%%%%%%%%%%%%%%%    lower semicontinuous
The function $\gh$ is \textit{lower semicontinuous} (lsc) at $\ov{x} \in \R^{n}$ if, for any sequence $\{x^k\}_{k\in \mathbb{N}} \subseteq \R^{n}$ with $x^k \rightarrow \ov{x}$, we have $\bs\liminf_{k\rightarrow + \infty} \gh(x^k)\geq \gh(\ov{x})$. 
The function $\gh$ is lsc on $\R^n$ if it is lsc at every $x \in \R^n$.
%%%%%%%%%%%%%%%%%%    coercive
We say $\gh$ is \textit{coercive} if $\bs\lim_{\Vert x\Vert\to +\infty}  \gh(x)=+\infty$.
%%%%%%%%%%%%%%%%%%    set valued
For a set-valued mapping $T: \R^n \rightrightarrows \R^n$  its domain is 
\[\Dom{T}:=\{x\in\R^n\mid T(x)\neq\emptyset\}.\]
%%%%%%%%%%%%%%%%%% 
For $p>1$, the gradient of the function $\frac{1}{p}\Vert x\Vert^p$ is
\begin{align*}\label{eq:formulaofdiff}
\nabla\left(\frac{1}{p}\Vert x\Vert^p\right)
=\left\{
\begin{array}{ll}
\Vert x\Vert^{p-2} x~~& x\neq 0,\\
 0 & x=0.
\end{array}\right.
\end{align*}
Hence, we write $\nabla\left(\frac{1}{p}\Vert x\Vert^p\right)=\Vert x\Vert^{p-2} x$ by adopting the convention $\frac{0}{0}=0$ for $p\in (1, 2)$ and $x=0$.

%%%%%%%%%%%%%%%%%% 

We now present key inequalities essential for later sections.
%%%%%%%%%%%%%%%%%%%%%%%%%%%%%%%%%%%%%%%%%%%%
%%%%%%%%%%%%%%%%%%%%%%%%%%%%%%%%%%%%%%%%%%%%
\begin{fact}[Basic inequalities I]\label{lemma:ineq:inequality p} 
Let $a, b\in \R^n$. The following hold:
\begin{enumerate}[label=(\textbf{\alph*}), font=\normalfont\bfseries, leftmargin=0.7cm]
\item \label{lemma:ineq:inequality p:ineq1} For each $\lambda\in (0,1)$ and $p\geq 1$, 
$\Vert a+b\Vert^p\geq \lambda^{p-1}\Vert a\Vert^p-\left(\frac{\lambda}{1-\lambda}\right)^{p-1}\Vert b\Vert^p$.
\item \label{lemma:ineq:inequality p:ineq4} For each $p\geq 1$, $\Vert a-b\Vert^p\leq 2^{p-1}\left(\Vert a\Vert^p+\Vert b\Vert^p\right)$.
\item  \label{lemma:ineq:inequality p:ineq3} For each $p\geq 2$, $\langle \Vert a\Vert^{p-2}a - \Vert b\Vert^{p-2}b, a - b\rangle\geq \left(\frac{1}{2}\right)^{p-2} \Vert a - b\Vert^p$.
\end{enumerate}
\end{fact}
\begin{proof}
For \ref{lemma:ineq:inequality p:ineq1}, see \cite[Lemma 2.1]{Kabgani24itsopt}. 
Assertion \ref{lemma:ineq:inequality p:ineq4} follows by setting $\lambda=\frac{1}{2}$, and substituting $a-b$ for $a$ in Assertion~$\ref{lemma:ineq:inequality p:ineq1}$. 
For \ref{lemma:ineq:inequality p:ineq3}, see \cite[Lemma~4.2.3]{Nesterov2018}.
\end{proof}
%%%%%%%%%%%%%%%%%%%%%%%%%%%%%%%%%%%%%%%%%%%%
%%%%%%%%%%%%%%%%%%%%%%%%%%%%%%%%%%%%%%%%%%%%
Let us consider the function $\kappa: (1, 2]\to (0, +\infty)$ given by
\begin{equation}\label{eq:formofKs:def}
\kappa(t):=\left\{
   \begin{array}{ll}
     \frac{(2+\sqrt{3})(t-1)}{16} & t\in (1, \widehat{t}], \\[0.2cm]
      \frac{2+\sqrt{3}}{16}\left(1-\left(3-\sqrt{3}\right)^{1-t}\right) ~~& t\in [\widehat{t},2),
      \\[0.2cm]
      1 & t=2,
   \end{array}\right.
\end{equation}
where $\widehat{t}$ is the solution of the equation
$\frac{t(t-1)}{2}=1-\left[1+\frac{(2-\sqrt{3})t}{t-1}\right]^{1-t}$, on $(1, 2]$,
and is determined numerically as $\widehat{t} \approx 1.3214$. For the sake of simplicity, we set $\kappa_t:=\kappa(t)$.
%%%%%%%%%%%%%%%%%%%%%%%%%%%%%%%%%%%%%%%%%%%%
%%%%%%%%%%%%%%%%%%%%%%%%%%%%%%%%%%%%%%%%%%%%
\begin{lemma}[Basic inequalities II]\label{lem:findlowbounknu:lemma}
Let $a, b\in \R^n$.
The following hold:
\begin{enumerate}[label=(\textbf{\alph*}), font=\normalfont\bfseries, leftmargin=0.7cm]
\item \label{lem:findlowbounknu:lemma:e2}
Let $r>0$ and $p\in (1,2]$. Then, for any $a, b\in \mb(0; r)$,
\begin{equation}\label{findlowbounknu:eq1}
\langle \Vert a\Vert^{p-2}a - \Vert b\Vert^{p-2}b, a-b\rangle\geq \kappa_pr^{p-2}\Vert a - b\Vert^2.
\end{equation}
\item \label{lem:findlowbounknu:lemma:e3}
Let $r>0$, $p\geq 2$, and $s=\frac{p}{p-1}$. Then, for any $a, b\in \mb(0; r)$,
\begin{equation}\label{findlowbounknu:eq1b}
\left\Vert\Vert a\Vert^{p-2}a - \Vert b\Vert^{p-2}b\right\Vert\leq  \frac{2r^{p-2}}{\kappa_{s}}\Vert a - b\Vert.
\end{equation}
\end{enumerate}
\end{lemma}
\begin{proof}
For $p=2$, both inequalities are straightforward. We now address the other claims.\\
$\ref{lem:findlowbounknu:lemma:e2}$
From \cite[Remark~1 and eq.~(1.1)]{XuRoach}, we have
\begin{equation*}\label{eq:eq1:findlowbounknu:eq2}
\langle \Vert a\Vert^{p-2}a - \Vert b\Vert^{p-2}b, a-b\rangle\geq \frac{K_{p}}{p} \left(\bs\max\{\Vert a\Vert, \Vert b\Vert\}\right)^p \left(1-\sqrt{1-\frac{\Vert a - b\Vert^2}{16(\bs\max\{\Vert a\Vert, \Vert b\Vert\})^2}}\right),
\end{equation*}
where
\begin{align*}\label{formofKs:def}
   K_p:=4(2+\sqrt{3})\bs\min \Biggl\{&\bs\min\left\{\frac{p(p-1)}{2},1\right\}, \bs\min\left\{\frac{p}{2},1\right\}(p-1),\\&
    (p-1)\left[1-(\sqrt{3}-1)^{\frac{p}{p-1}}\right], 1-\left[1+\frac{(2-\sqrt{3})p}{p-1}\right]^{1-p}\Biggr\}.
 \end{align*}
Since $p\in (1, 2)$, this simplifies to
 \begin{align*}
   K_{p}=4(2+\sqrt{3})\bs\min \left\{\frac{p(p-1)}{2}, (p-1)\left[1-(\sqrt{3}-1)^{\frac{p}{p-1}}\right], 1-\left[1+\frac{(2-\sqrt{3})p}{p-1}\right]^{1-p}\right\}.
 \end{align*}
 Define
\begin{align*}
h_1(p):=\frac{p(p-1)}{2},~~
h_2(p):=(p-1)\left[1-(\sqrt{3}-1)^{\frac{p}{p-1}}\right],~~
h_3(p):=1-\left[1+\frac{(2-\sqrt{3})p}{p-1}\right]^{1-p}.
\end{align*}
From Fig.~\ref{fig:thetas}, it follows that $ K_{p}$ can be computed as 
\begin{equation}\label{eq:formofKs:def3}
K_p=\left\{
\begin{array}{ll}
2(2+\sqrt{3})p(p-1)& p\in (1, \widehat{t}],  \\[0.2cm]
      4(2+\sqrt{3})\left(1-\left[1+\frac{(2-\sqrt{3})p}{p-1}\right]^{1-p}\right)~~ &p\in [\widehat{t}, 2),
\end{array}\right.
\end{equation}
where $\widehat{t}$ is the solution of the nonlinear system $h_1(p)=h_3(p)$.

 \begin{figure}[H]
\centering
\begin{tikzpicture}
\begin{axis}[
    xlabel={$p$},
    ylabel={$h_i, i=1,2,3$},
    xmin=1, xmax=2,
    ymin=0, ymax=0.6,
    axis lines=middle,
    width=0.55\linewidth, 
    height=0.3\linewidth, 
    samples=400,
]
\addplot[color=red,domain=1:2,dashed,thick]{0.5*x*(x-1)} node[pos=0.55, right] {$h_1$};
\addplot[color=blue,domain=1:2,dotted,thick]{(x-1)-(x-1)*((sqrt(3)-1)^((x)/(x-1)))} node[pos=0.90,   below right] {$h_2$};
\addplot[color=green,domain=1:2,smooth,thick]{1-(1+((2-sqrt(3))*(x)/(x-1)))^(1-x)} node[pos=0.9, below right] {$h_3$};
\addplot[color=black, mark=*, only marks] coordinates {(1.32141416, 0.2125)};
\addplot[color=black, mark=*, only marks] coordinates {(1.32141416, 0)};
\draw[dashed] (axis cs:1.32141416, 0.2125) -- (axis cs:1.32141416, 0) node[pos=0.6, below left] {$\widehat{t}$};
\end{axis}
\end{tikzpicture}
\caption{Plot of $h_1$, $h_2$, and $h_3$ in the proof of Lemma~\ref{lem:findlowbounknu:lemma}.}\label{fig:thetas}
\end{figure}

On the other hand, since
$1-\sqrt{1-\frac{\Vert a - b\Vert^2}{16(\bs\max\{\Vert a\Vert, \Vert b\Vert\})^2}}\geq \frac{\Vert a - b\Vert^2}{32(\bs\max\{\Vert a\Vert, \Vert b\Vert\})^2}$,
it follows that
\begin{equation}\label{findlowbounknu:eq2}
\langle \Vert a\Vert^{p-2}a - \Vert b\Vert^{p-2}b, a-b\rangle\geq \frac{K_{p}}{p} \left(\bs\max\{\Vert a\Vert, \Vert b\Vert\}\right)^{p-2}  \frac{\Vert a - b\Vert^2}{32}.
\end{equation}
 Since $\bs\max\{\Vert a\Vert, \Vert b\Vert\}\leq r$ and $p\in (1,2)$, we have $\left(\bs\max\{\Vert a\Vert, \Vert b\Vert\}\right)^{p-2}\geq r^{p-2}$.
Therefore, combining this with \eqref{findlowbounknu:eq2} yields
\[
\langle \Vert a\Vert^{p-2}a - \Vert b\Vert^{p-2}b, a-b\rangle\geq \frac{K_p r^{p-2}}{32 p}\Vert a - b\Vert^2.
\]
If $p\in (1, \widehat{t}]$, then from \eqref{eq:formofKs:def3}, we obtain
\begin{equation}\label{eq:findlowbounknu:eq2:a}
 \frac{K_p}{32 p}=\frac{2(2+\sqrt{3})p(p-1)}{32 p}=\frac{(2+\sqrt{3})(p-1)}{16}.
\end{equation}
Additionally, for $p\in [\widehat{t}, 2)$, we have $1+\frac{(2-\sqrt{3})p}{p-1}\geq 3-\sqrt{3}$ and hence
$\left[1+\frac{(2-\sqrt{3})p}{p-1}\right]^{1-p}\leq \left[3-\sqrt{3}\right]^{1-p}$. Thus,
\begin{align}\label{eq:findlowbounknu:eq2:b}
\frac{K_p}{32 p}=\frac{ 4(2+\sqrt{3})\left(1-\left[1+\frac{(2-\sqrt{3})p}{p-1}\right]^{1-p}\right)}{32 p}
&\geq \frac{4(2+\sqrt{3})\left(1-\left[3-\sqrt{3}\right]^{1-p}\right)}{64}
\nonumber\\&= \frac{(2+\sqrt{3})\left(1-\left[3-\sqrt{3}\right]^{1-p}\right)}{16}.
\end{align}
From \eqref{eq:findlowbounknu:eq2:a} and \eqref{eq:findlowbounknu:eq2:b}, we obtain \eqref{findlowbounknu:eq1}.
\\
$\ref{lem:findlowbounknu:lemma:e3}$ 
Let $p>2$ and set $s=\frac{p}{p-1}$. Thus, $s\in (1, 2)$.
From \cite[eqs.~(1.3)~and~(3.5)]{XuRoach}, we obtain
\begin{equation}\label{eq1:lem:findlowbounknu:lemma:e3}
\left\Vert\Vert a\Vert^{p-2}a - \Vert b\Vert^{p-2}b\right\Vert\leq \frac{K_s}{s}\frac{(\bs\max\{\Vert a\Vert, \Vert b\Vert\})^{p}}{\Vert a - b\Vert}\Theta(z),
\end{equation}
where $K_s$ is given by \eqref{eq:formofKs:def3}, 
$z=\frac{8\Vert a - b\Vert}{\frac{K_s}{s}\bs\max\{\Vert a\Vert, \Vert b\Vert\}}$, and
\[\Theta(z)=\left\{\begin{array}{ll}
                       \bs\max\{0, z-1\}~~ & n=1, \\\\
                       \sqrt{1+z^2}-1 & n\geq 2.
                     \end{array}\right.
                     \]
Since $\Theta(z)\leq z^2$ for all $n\in\mathbb{N}$, 
inequality \eqref{eq1:lem:findlowbounknu:lemma:e3} together with $\bs\max{\Vert a\Vert, \Vert b\Vert}\leq r$ yields
\begin{align*}
\left\Vert\Vert a\Vert^{p-2}a - \Vert b\Vert^{p-2}b\right\Vert&\leq \frac{64(\bs\max\{\Vert a\Vert, \Vert b\Vert\})^{p-2}}{\frac{K_s}{s}}\Vert a - b\Vert\leq  \frac{64r^{p-2}}{\frac{K_s}{s}}\Vert a - b\Vert.
\end{align*}
Similar to the proof of Assertion~$\ref{lem:findlowbounknu:lemma:e2}$,
if $s\in (1, \widehat{t}]$, then
\begin{equation}\label{eq:findlowbounknu:eq2:c}
\frac{K_s}{64s}=
\frac{(2+\sqrt{3})(s-1)}{32}=\frac{\kappa_s}{2}.
\end{equation}
For $s\in [\widehat{t}, 2)$, we have 
\begin{align}\label{eq:findlowbounknu:eq2:d}
\frac{K_s}{64 s}=\frac{ 4(2+\sqrt{3})\left(1-\left[1+\frac{(2-\sqrt{3})s}{s-1}\right]^{1-s}\right)}{64 s}
&\geq \frac{(2+\sqrt{3})\left(1-\left[3-\sqrt{3}\right]^{1-s}\right)}{32}=\frac{\kappa_s}{2}.
\end{align}
Therefore, combining \eqref{eq:findlowbounknu:eq2:c} and \eqref{eq:findlowbounknu:eq2:d} establishes \eqref{findlowbounknu:eq1b}.
\end{proof}
Note that in Lemma~\ref{lem:findlowbounknu:lemma}~$\ref{lem:findlowbounknu:lemma:e3}$, when $p=2$, one could replace $\frac{2}{\kappa_{s}}$ with $\frac{1}{\kappa_{s}}=1$. However, to maintain consistency, we do not treat this case separately. If needed, it will be addressed explicitly in the text.

%%%%%%%%%%%%%%%%%%%%%%%%%%%%%%%%%%%%%%%%%%%%%%%%%%%%%%%%%%%%%%%%%%%%%%%%%%%%%%%%%%%
%%%%%%%%%%%%%%%%%%%%%%%%%%%%%%%%%%%%%%%%%%%%%%%%%%%%%%%%%%%%%%%%%%%%%%%%%%%%%%%%%%%
A proper function $\gh: \R^n \to \Rinf$  is said to be \textit{Fr\'{e}chet differentiable} at $\ov{x}\in \inter{\dom{\gh}}$ with \textit{Fr\'{e}chet derivative}  
$\nabla \gh(\ov{x})$
 if 
\[
\mathop{\bs\lim}\limits_{x\to \ov{x}}\frac{\gh(x) -\gh(\ov{x}) - \langle \nabla \gh(\ov{x}) , x - \ov{x}\rangle}{\Vert x - \ov{x}\Vert}=0.
\]
%%%%%%%%%%%%%%%%%%%%%%%%%%%%%%%%%%%%%%%%%%%%%%%%%%%%%%%%%%%%%%%%%%%%%%%%%%%%%%%%%%%
%%%%%%%%%%%%%%%%%%%%%%%%%%%%%%%%%%%%%%%%%%%%%%%%%%%%%%%%%%%%%%%%%%%%%%%%%%%%%%%%%%%
Here, we discuss some important tools from generalized differentiability. 

\begin{definition}[Fr\'{e}chet/regular and Mordukhovich/limiting subdifferentials]\cite{Mordukhovich2018,Rockafellar09}
Let $\gh: \R^n \to \Rinf$ be a proper function and $\ov{x}\in \dom{\gh}$.
\begin{enumerate}[label=(\textbf{\alph*}), font=\normalfont\bfseries, leftmargin=0.7cm]
\item The \textit{Fr\'{e}chet/regular subdifferential} of $\gh$ at $\ov{x}$  is defined as
\[
\widehat{\partial}\gh(\ov{x}):=\left\{\zeta\in \R^n\mid~\mathop{\bs\liminf}\limits_{x\to \ov{x}}\frac{\gh(x)- \gh(\ov{x}) - \langle \zeta, x - \ov{x}\rangle}{\Vert x - \ov{x}\Vert}\geq 0\right\}.
\]
\item The \textit{Mordukhovich/limiting  subdifferential} of $\gh$ at $\ov{x}$ is defined as
\[
\partial \gh(\ov{x}):=\left\{\zeta\in \R^n\mid~\exists x^k\to \ov{x}, \zeta^k\in \widehat{\partial}\gh(x^k),~~\text{with}~~\gh(x^k)\to \gh(\ov{x})~\text{and}~ \zeta^k\to \zeta\right\}.
\]
\end{enumerate}
\end{definition}

%%%%%%%%%%%%%%%%
In general, we have $\widehat{\partial}\gh(\ov{x})\subseteq\partial \gh(\ov{x})$
see \cite[Theorem 8.6]{Rockafellar09}. Moreover, equality holds for several important classes of functions, such as smooth, convex, and amenable functions; see \cite{Mordukhovich2018,Rockafellar09}. In this work, we examine the relationships between minimizers of $\gh$, its high-order Moreau envelope, zeros of these subdifferentials, and other related notions. It is also worth noting that the Mordukhovich (limiting) subdifferential plays a central role in the definitions of prox-regularity and $q$-prox-regularity, which are the key variational notions used in this work.

%%%%%%%%%%%%%%%%
The class of \textit{prox-regular functions}, described next, encompasses a wide array of significant functions encountered in optimization, including, but not limited to, proper lsc convex functions, $\mathcal{C}^{2}$ functions, lower-$\mathcal{C}^{2}$ functions, strongly amenable functions, primal-lower-nice functions \cite{Poliquin96,Rockafellar09}, weakly convex functions \cite{nurminskii1973quasigradient}, variationally convex functions \cite{Khanh23,Rockafellar2019Variational}, and more; see, e.g., \cite{Ahookhosh21, Bareilles2023, Lewis2016, Mordukhovich2016, Mordukhovich2021} and references therein for some applications.

 \begin{definition}[Prox-regularity]\label{def:prox-regular}\cite{Poliquin96}
Let $\gh: \R^n\to \Rinf$ be a proper lsc function and let $\ov{x}\in\dom{\gh}$. Then, $\gh$ is said to be \textit{prox-regular} at $\ov{x}$ for $\ov{\zeta}\in\partial \gh(\ov{x})$ if there exist $\varepsilon>0$ and $\rho\geq 0$ such that
\[
\gh(x')\geq \gh(x)+\langle \zeta, x'-x\rangle-\frac{\rho}{2}\Vert x'-x\Vert^2,~\qquad \forall x'\in \mb(\ov{x}; \varepsilon),
\]
whenever $x\in \mb(\ov{x}; \varepsilon)$, $\zeta\in \partial \gh(x)\cap \mb(\ov{\zeta}; \varepsilon)$, and $\gh(x)<\gh(\ov{x})+\varepsilon$.
\end{definition}

%%%%%%%%%%%%%%%%%%
We conclude this section by introducing the class of functions with $\nu$-H\"{o}lder continuous gradient. Later, we show that, under certain conditions, the high-order Moreau envelope of a function possesses this property.
\begin{definition}[H\"{o}lder continuous gradient]
A proper function $\gh: \R^n \to\Rinf$  is said to have a \textit{$\nu$-H\"{o}lder continuous gradient} on $C\subseteq \dom{\gh}$ with $\nu\in (0, 1]$ if it is Fr\'{e}chet differentiable and  there exists a constant $L_\nu\geq 0$ such that
\begin{equation}\label{eq:nu-Holder continuous gradient}
\Vert \nabla \gh(y)- \nabla \gh(x)\Vert \leq L_\nu \Vert y-x\Vert^\nu, \qquad \forall x, y\in C.
\end{equation}
\end{definition}
The class of such functions is denoted by $\mathcal{C}^{1, \nu}_{L_\nu}(C)$, and are called weakly smooth. 
We use $\mathcal{C}^{k}(C)$ to denote the class of functions that are $k$-times continuously differentiable on $C$, for $k\in \mathbb{N}$.

\section{Moreau envelope with high-order regularization}
\label{sec:home}
Here, we first introduce the high-order proximal operator (HOPE) and the high-order Moreau envelope (HOME) and recall their basic properties (Fact~\ref{th:level-bound+locally uniform}), including the non-emptiness of HOPE as well as the finiteness and continuity of HOME under the concept of high-order prox-boundedness 
 (Definition~\ref{def:s-prox-bounded}). Next, several additional properties of HOPE and HOME will be discussed. 

 %%%%%%%%%%%%
Let us begin with the definitions of HOPE and HOME by considering a $p$th-order regularization term, for $p>1$.
%%%%%%%%%%%%%%%%%%%%%%%%%%%%%%%%%%%%%%%%%%%%%%%%%%%%%%%%%%%%%%%%%%%%%%%%%%%%%%%%%%%
%%%%%%%%%%%%%%%%%%%%%%%%%%%%%%%%%%%%%%%%%%%%%%%%%%%%%%%%%%%%%%%%%%%%%%%%%%%%%%%%%%%
\begin{definition}[High-order proximal operator and Moreau envelope]\label{def:Hiorder-Moreau env}
Let $p>1$  and $\gamma>0$, and let $\gf: \R^n \to \Rinf$ be a proper function. 
    The \textit{high-order proximal operator} (\textit{HOPE}) of $\gf$ of parameter $\gamma$,
    $\prox{\gf}{\gamma}{p}: \R^n \rightrightarrows \R^n$, is defined as
   \begin{equation}\label{eq:Hiorder-Moreau prox}\tag{HOPE}
       \prox{\gf}{\gamma}{p} (x):=\argmint{y\in \R^n} \left(\gf(y)+\frac{1}{p\gamma}\Vert x- y\Vert^p\right),
    \end{equation}     
    and the \textit{high-order Moreau envelope} (\textit{HOME}) of $\gf$ of parameter $\gamma$, 
    $\fgam{\gf}{p}{\gamma}:\R^n\to \R\cup\{\pm \infty\}$, 
    is given by
    \begin{equation}\label{eq:Hiorder-Moreau env}\tag{HOME}
    \fgam{\gf}{p}{\gamma}(x):=\mathop{\bs{\inf}}\limits_{y\in \R^n} \left(\gf(y)+\frac{1}{p\gamma}\Vert x- y\Vert^p\right).
    \end{equation}
\end{definition}
%%%%%%%%%%%%%%%%%%%%%%%%%%%%%%%%%%%%%%%%%%%%%%%%%%%%%%%%%%%%%%%%%%%%%%%%%%%%%%%%%%%
%%%%%%%%%%%%%%%%%%%%%%%%%%%%%%%%%%%%%%%%%%%%%%%%%%%%%%%%%%%%%%%%%%%%%%%%%%%%%%%%%%%
Note that HOPE is a set-valued operator, while HOME is a function.
HOME is also known as the epigraphical regularization with parameter $\gamma$ of the function $\gf$ \cite[page~701]{Attouch91} and the Moreau $p$-envelope \cite{KecisThibault15}.

%%%%%%%%%%%%%%%%%%%%%%%%%%%%%%%%%%%%%%%%%%%%%%%%%%%%%%%%%%%%%%%%%%%%%%%%%%%%%%%%%%%
%%%%%%%%%%%%%%%%%%%%%%%%%%%%%%%%%%%%%%%%%%%%%%%%%%%%%%%%%%%%%%%%%%%%%%%%%%%%%%%%%%%
\subsection{{\bf Fundamental properties of HOPE and HOME}}
\label{subsec:well-defi}
The finite-valuedness of HOME and the non-emptiness of HOPE are two crucial properties that require careful investigation. In this subsection, we establish certain assumptions that ensure these properties; 
see Fact~\ref{th:level-bound+locally uniform}. Before diving into the details, let us first introduce some preliminary concepts.

Fact~\ref{fact:horder:Bauschke17:p12.9} collects some well-established properties of HOME.
\begin{fact}[Domain and majorizer for HOME]\cite[Proposition~12.9]{Bauschke17}\label{fact:horder:Bauschke17:p12.9}
Let $p> 1$ and $\gf:\R^n \to \Rinf$ be a proper function. Then, for every $\gamma > 0$,
$\dom{\fgam{\gf}{p}{\gamma}}=\R^n$. Additionally, for every $\gamma_2>\gamma_1>0$ and $x\in \R^n$, $\fgam{\gf}{p}{\gamma_2}(x)\leq \fgam{\gf}{p}{\gamma_1}(x)\leq \gf(x)$.
\end{fact}
As indicated by Fact~\ref{fact:horder:Bauschke17:p12.9}, for a proper function $\gf$, one has $\fgam{\gf}{p}{\gamma}(x) < +\infty$ for every $\gamma > 0$. However, this does not ensure finiteness of $\fgam{\gf}{p}{\gamma}$ in nonconvex settings. In particular, there may be instances with $\fgam{\gf}{p}{\gamma}(x) = -\infty$ for each $\gamma>0$ and $x \in \R^n$ (see Example~\ref{ex:nondif:prox:prox}).
To prevent this situation, Poliquin and Rockafellar \cite{Poliquin96} assumed (for $p = 2$) that $\gf$ majorizes a quadratic function; see Assumption~4.1 in \cite{Poliquin96}. This condition, which we refer to as $2$-calmness, implies the existence of some $\gamma > 0$ and $x \in \R^n$ such that $\fgam{\gf}{2}{\gamma}(x) > -\infty$, a property now widely recognized as prox-boundedness (see \cite[Definition~1.23]{Rockafellar09}). In Subsection~\ref{subsec:critical}, we thoroughly examine the connections between these notions. 
Let us generalize the prox-boundedness for $p > 1$.

%%%%%%%%%%%%%%%%%%%%%%%%%%%%%%%%%%%%%%%%%%%%%%%%%%%%%%%%%%%%%%%%%%%%%%%%%%%%%%%%%%%
%%%%%%%%%%%%%%%%%%%%%%%%%%%%%%%%%%%%%%%%%%%%%%%%%%%%%%%%%%%%%%%%%%%%%%%%%%%%%%%%%%%
\begin{definition}[High-order prox-boundedness]\label{def:s-prox-bounded}
A function $\gf:\R^n\to \Rinf$ is said to be \textit{high-order prox-bounded} with order $p> 1$, if there exist a $\gamma>0$ and $x\in \R^n$ such that
$\fgam{\gf}{p}{\gamma}(x)>-\infty$. 
The supremum of all such $\gamma$ is denoted by $\gamma^{\gf, p}$ and is referred to as the threshold of high-order prox-boundedness of $\gf$.
\end{definition}
%%%%%%%%%%%%%%%%%%%%%%%%%%%%%%%%%%%%%%%%%%%%%%%%%%%%%%%%%%%%%%%%%%%%%%%%%%%%%%%%%%%
%%%%%%%%%%%%%%%%%%%%%%%%%%%%%%%%%%%%%%%%%%%%%%%%%%%%%%%%%%%%%%%%%%%%%%%%%%%%%%%%%%%
The following characterizations of high-order prox-boundedness will be useful in the remainder of this paper.
\begin{proposition}[Characterizations of high-order prox-boundedness]\label{lemma:charac:sprox}
Let $p> 1$ and $\gf:\R^n\to \Rinf$ be a proper lsc function. The following statements are equivalent:
\begin{enumerate}[label=(\textbf{\alph*}), font=\normalfont\bfseries, leftmargin=0.7cm]
  \item \label{lemma:charac:sprox:a} $\gf$ is high-order prox-bounded;
  \item \label{lemma:charac:sprox:b} there exists an $\ell>0$ such that $\gf(\cdot)+\ell\Vert \cdot\Vert^p$ is bounded from below on $\R^n$;
  \item \label{lemma:charac:sprox:c} $\mathop{\bs\liminf}\limits_{\Vert x\Vert\to \infty} \frac{\gf(x)}{\Vert x\Vert^p}>-\infty$.
\end{enumerate}
\end{proposition}
 \begin{proof}
 $\ref{lemma:charac:sprox:a}\Rightarrow \ref{lemma:charac:sprox:b}$ Since $\gf$ is high-order prox-bounded, there exist $\gamma>0$, $\ov{x}\in \R^n$, and $\ell_0\in \R$ such that
  \begin{align*}
  \gf(y)+\frac{2^{p-1}}{p\gamma}\left(\Vert \ov{x}\Vert^p+\Vert y\Vert^p\right)\geq  \gf(y)+\frac{1}{p\gamma}\Vert \ov{x}- y\Vert^p\geq \ell_0, \quad \forall y\in \R^n,
  \end{align*}
where the first inequality follows from Fact~\ref{lemma:ineq:inequality p}~\ref{lemma:ineq:inequality p:ineq4}. This implies that
  \[\gf(y)+\frac{2^{p-1}}{p\gamma}\Vert y\Vert^p\geq \ell_0 -\frac{2^{p-1}}{p\gamma}\Vert \ov{x}\Vert^p,\]
  for all $y\in \R^n$.
  Letting $\ell:=\frac{2^{p-1}}{p\gamma}$, we obtain the desired result.
  \\
  $\ref{lemma:charac:sprox:b}\Rightarrow \ref{lemma:charac:sprox:c}$ Since there exist $\ell, \ell_0\in\R$ such that $\gf(x)+\ell\Vert x\Vert^p\geq \ell_0$ for every $x\in\R^n$, dividing both sides by $\Vert x\Vert^p$ and taking the limit as $\Vert x\Vert\to \infty$ confirms the validity of the claim.
\\
$\ref{lemma:charac:sprox:c}\Rightarrow \ref{lemma:charac:sprox:a}$ From \cite[Exercise~1.14]{Rockafellar09}, there exist $\ell, \ell_0\in \R$ such that for each $x\in \R^n$,
$\gf(x)\geq \ell \Vert x\Vert^p+\ell_0$.
    Thus, for $\ov{x}=0$,
    \[
    \gf(x)+\frac{1}{p\gamma}\Vert x-\ov{x}\Vert^p= \gf(x)+\frac{1}{p\gamma}\Vert x\Vert^p\geq \left(\frac{1}{p\gamma}+\ell\right) \Vert x\Vert^p+\ell_0,\quad \forall x\in \R^n.
    \]
    If $\ell\geq 0$, then for any $\gamma>0$, $\fgam{\gf}{p}{\gamma}(\ov{x})>-\infty$. If $\ell<0$, then for any $\gamma\in \left(0, -\frac{1}{p\ell}\right)$, $\fgam{\gf}{p}{\gamma}(\ov{x})>-\infty$. In both cases, there exists some $\gamma>0$ such that $\fgam{\gf}{p}{\gamma}(\ov{x})>-\infty$  for $\ov{x}$. Therefore, $\gf$ is high-order prox-bounded.
\end{proof}
%%%%%%%%%%%%%%%%%%%%%%%%%%%%%%%%%%%%%%%%%%%%%%%%%%%%%%%%%%%%%%%%%%%%%%%%%%%%%%%%%%%
%%%%%%%%%%%%%%%%%%%%%%%%%%%%%%%%%%%%%%%%%%%%%%%%%%%%%%%%%%%%%%%%%%%%%%%%%%%%%%%%%%%
Note that convex functions and also lower-bounded functions are evidently high-order prox-bounded with $\gamma^{\gf, p} = +\infty$.
%%%%%%%%%%%%%%%%%%%%%%%%%%%%%%%%%%%%%%%%%%%%%%%%%%%%%%%%%%%%%%%%%%%%%%%%%%%%%%%%%%%
%%%%%%%%%%%%%%%%%%%%%%%%%%%%%%%%%%%%%%%%%%%%%%%%%%%%%%%%%%%%%%%%%%%%%%%%%%%%%%%%%%%

The following fact introduces the conditions ensuring the non-emptiness and outer semicontinuity of $\prox{\gf}{\gamma}{p}$, as well as the continuity of $\fgam{\gf}{p}{\gamma}$. Similar results for $p=2$ can be found in \cite[Theorem~1.25]{Rockafellar09}.

 \begin{fact}[Basic properties of HOME and HOPE]\cite{Kabgani24itsopt}\label{th:level-bound+locally uniform}
Let $p>1$ and $\gf: \R^n\to \Rinf$ be a proper lsc function that is high-order prox-bounded with a threshold $\gamma^{\gf, p}>0$. Then, for each $\gamma\in (0, \gamma^{\gf, p})$,
\begin{enumerate}[label=(\textbf{\alph*}), font=\normalfont\bfseries, leftmargin=0.7cm]
\item \label{level-bound+locally uniform:proxnonemp} $\prox{\gf}{\gamma}{p}(x)$ is nonempty and compact and $\fgam{\gf}{p}{\gamma}(x)$ is finite for every $x\in \R^n$;
\item \label{level-bound+locally uniform:cononx} $\fgam{\gf}{p}{\gamma}$ is continuous on $\R^n$;
 \item \label{level-bound+locally uniform2:con} $\fgam{\gf}{p}{\gamma}$ depends continuously on $(x,\gamma)$ in $\R^n\times (0, \gamma^{\gf, p})$;
\item \label{level-bound+locally uniform2:conv} if $y^k\in \prox{\gf}{\gamma_k}{p}(x^k)$, with $x^k\to \ov{x}$ and $\gamma_k\to \gamma\in (0, \gamma^{\gf, p})$, then the sequence $\{y^k\}_{k\in \mathbb{N}}$ is bounded. Furthermore, all cluster points of the sequence lie in $\prox{\gf}{\gamma}{p}(\ov{x})$.
     \end{enumerate}
 \end{fact}

%%%%%%%%%%%%%%%%%%%%%%%%%%%%%%%%%%%%%%%%%%%%%%%%%%%%%%%

%%%%%%%%%%%%%%%%%%%%%%%%%%%%%%%%%%%%%%%%%%%%%%%%%%%%%%%
Coercivity, which we study next for $\fgam{\gf}{p}{\gamma}$, is an essential condition in many numerical algorithms, as it provides a sufficient condition for the existence of a global minimum point under certain conditions.
\begin{proposition}[Coercivity]\label{lem:hiordermor:coer} 
Let $p> 1$ and $\gf: \R^n \to \Rinf$ be a proper and coercive function.
Then, for any $\gamma>0$, the following statements hold:
\begin{enumerate}[label=(\textbf{\alph*}), font=\normalfont\bfseries, leftmargin=0.7cm]
  \item\label{hiordermor:coer:coer}  $\fgam{\gf}{p}{\gamma}$ is coercive;
  \item \label{hiordermor:coer:bound} for each $\lambda\in \R$,  the sublevel set $\mathcal{L}(\fgam{\gf}{p}{\gamma}, \lambda)$ is bounded.
\end{enumerate}
\end{proposition}
\begin{proof}
$\ref{hiordermor:coer:coer}$
Fix $\varepsilon>0$. By the definition of $\fgam{\gf}{p}{\gamma}$, for any $x\in\R^n$,  there exists $y^\varepsilon_x\in\R^n$ such that
\begin{equation}\label{eq:hiordermor:coer}
\fgam{\gf}{p}{\gamma}(x)\leq \gf(y^\varepsilon_x)+\frac{1}{p\gamma}\Vert x-y^\varepsilon_x\Vert^p\leq \fgam{\gf}{p}{\gamma}(x)+\varepsilon.
\end{equation}
By contradiction, suppose that $\fgam{\gf}{p}{\gamma}$ is not coercive. 
Then, there exists a sequence $\{x^k\}_{k\in \mathbb{N}} \subseteq \R^{n}$ with $\Vert x^k\Vert \rightarrow \infty$ as $k\to \infty$ and $\bs\lim_{k\to \infty}\fgam{\gf}{p}{\gamma}(x^k)<\infty$. Let $\{y^\varepsilon_{x^k}\}_{k\in \mathbb{N}}$ be a corresponding sequence to $\{x^k\}_{k\in \mathbb{N}}$ such that each $y^\varepsilon_{x^k}$ satisfies \eqref{eq:hiordermor:coer}. Since $\varepsilon$ is fixed and, by convention $\infty - \infty = \infty$, relation \eqref{eq:hiordermor:coer} implies that 
$\bs\lim_{k\to \infty}\gf(y^\varepsilon_{x^k})<\infty$ and $\bs\lim_{k\to  \infty}\Vert x^k-y^\varepsilon_{x^k}\Vert^p<\infty$.
From the coercivity of $\gf$ and $\bs\lim_{k\to \infty}\gf(y^\varepsilon_{x^k})<\infty$, we deduce that 
$\bs\lim_{k\to \infty}\Vert y^\varepsilon_{x^k}\Vert< \infty$. 
By Fact~\ref{lemma:ineq:inequality p}~\ref{lemma:ineq:inequality p:ineq1}, we have
\[\Vert x^k\Vert^p\leq 2^{p-1}\left(\Vert x^k -y^\varepsilon_{x^k}\Vert^p+\Vert y^\varepsilon_{x^k}\Vert^p\right).\]
Since $\Vert y^\varepsilon_{x^k}\Vert$ remains bounded while $\Vert x^k\Vert\to \infty$, the above inequality implies
$\bs\lim_{k \to\infty}\Vert x^k -y^\varepsilon_{x^k}\Vert=\infty$, which is a contradiction.
\\
$\ref{hiordermor:coer:bound}$ This follows directly from Assertion~$\ref{hiordermor:coer:coer}$ and \cite[Proposition~11.12]{Bauschke17}.
\end{proof}
  %%%%%%%%%%%%%%%%%%%%%%%%%%%%%%%%%%%%%%%%%%%%%%%%%%%%%%%
%%%%%%%%%%%%%%%%%%%%%%%%%%%%%%%%%%%%%%%%%%%%%%%%%%%%%%%
The coercivity of the original function $\gf$ ensures the boundedness of its sublevel sets, which contain the global minimizers if any exist. To compute a minimizer, many algorithms generate a decreasing sequence $\{\gf(x^k)\}_{k\in \Nz} \subseteq \mathcal{L}(\gf, \gf(x^0))$ starting from an initial point $x^0$.
On the other hand, properties such as differentiability, Lipschitz continuity, and convexity are only guaranteed locally. 
Thus, for a given $r>0$, we consider an initial point $x^0$ such that
$\mathcal{L}(\gf, \gf(x^0))\subseteq \mb(0; r)$, and assume that the desired properties hold within $\mb(0; r)$.
For simplicity, we reference the ball centered at the origin, but this can be generalized by replacing $\mb(0; r)$ with $\mb(\ov{x}; r)$ for a reference point $\ov{x}\in \R^n$.
Since our analysis focuses on the function $\fgam{\gf}{p}{\gamma}$ instead of $\gf$, and because an explicit form of 
$\fgam{\gf}{p}{\gamma}$ is 
often unavailable for general functions, even though closed forms exist for some important cases when $p=2$, see, e.g., \cite{beck2017first}, it is crucial to clarify the relationship between the sublevel sets of the original function and those of HOME. Understanding this relationship allows us to select appropriate values for $r$ and $\gamma$ such that $\mathcal{L}(\gf, \gf(x^0))\subseteq \mb(0; r)$ implies $\mathcal{L}(\fgam{\gf}{p}{\gamma}, \fgam{\gf}{p}{\gamma}(x^0))\subseteq \mb(0; r)$.
The subsequent result establishes this relationship, which, surprisingly, appears to have been overlooked in the existing literature, even for $p=2$.

   %%%%%%%%%%%%%%%%%%%%%%%%%%%%%%%%%%%%%%%%%%%%%%%%%%%%%%%
%%%%%%%%%%%%%%%%%%%%%%%%%%%%%%%%%%%%%%%%%%%%%%%%%%%%%%%
 \begin{proposition}[Sublevel sets of $\gf$ and $\fgam{\gf}{p}{\gamma}$]\label{pro:rel:sublevel}
 Let $p> 1$ and $\gf: \R^n\to \Rinf$ be a proper lsc function that is coercive. If for some $r>0$ and $\lambda\in \R$, we have $\mathcal{L}(\gf, \lambda)\subseteq \mb(0; r)$, then there exists a $\widehat{\gamma}>0$ such that for each $\gamma\in (0, \widehat{\gamma}]$, we have $\mathcal{L}(\fgam{\gf}{p}{\gamma}, \lambda)\subseteq \mb(0; r)$.
 \end{proposition}
\begin{proof}
First, we establish the existence of a $\widehat{\gamma}>0$ guaranteeing $\mathcal{L}(\fgam{\gf}{p}{\widehat{\gamma}}, \lambda)\subseteq \mb(0; r)$.
By contradiction, let us assume that for any $\gamma>0$, there exists $x^\gamma$ such that $\Vert x^\gamma\Vert\geq r$ and 
$\fgam{\gf}{p}{\gamma}(x^\gamma)\leq \lambda$. Consider a decreasing sequence $\gamma_{k}\downarrow 0$ and the corresponding sequence $\{x^{\gamma_{k}}\}_{k\in \mathbb{N}}$. 
From Fact~\ref{fact:horder:Bauschke17:p12.9}, we know that for each $k\in \mathbb{N}$,
$\fgam{\gf}{p}{\gamma_1}(x^{\gamma_{k}})\leq \fgam{\gf}{p}{\gamma_{k}}(x^{\gamma_{k}})\leq \lambda$. 
The coercivity of $\fgam{\gf}{p}{\gamma_1}$ (see Proposition~\ref{lem:hiordermor:coer}) implies that 
$\Vert x^{\gamma_{k}}\Vert\nrightarrow \infty$ as $k\to \infty$. 
Hence, there exists an infinite subset $J\subseteq \mathbb{N}$ and convergent subsequence $\{x^{\gamma_j}\}_{j\in J}$ with a limiting point $\widehat{x}$ such that $x^{\gamma_j}\to \widehat{x}$ as $j\to\infty$, which implies $\Vert \widehat{x}\Vert\geq r$. 
On the contrary, by Theorem~\ref{th:level-bound+locally uniform}~$\ref{level-bound+locally uniform2:con}$ and \cite[Theorem~3.1~(c)]{KecisThibault15}, we have $\gf(\widehat{x})=\bs\lim_{j\to \infty}\fgam{\gf}{p}{\gamma_j}(x^{\gamma_j})\leq \lambda$, which in turn implies $\Vert \widehat{x}\Vert<r$, which is a clear contradiction. 
Finally, since for any $\gamma\in (0, \widehat{\gamma}]$, $\fgam{\gf}{p}{\widehat{\gamma}}\leq \fgam{\gf}{p}{\gamma}$ (see Fact~\ref{fact:horder:Bauschke17:p12.9}), we have
$\mathcal{L}(\fgam{\gf}{p}{\gamma}, \lambda)\subseteq\mathcal{L}(\fgam{\gf}{p}{\widehat{\gamma}}, \lambda)$, and the proof is completed.
\end{proof}
  %%%%%%%%%%%%%%%%%%%%%%%%%%%%%%%%%%%%%%%%%%%%%%%%%%%%%%%
%%%%%%%%%%%%%%%%%%%%%%%%%%%%%%%%%%%%%%%%%%%%%%%%%%%%%%%
The next result is useful for handling the sublevel sets of HOME in the presence of errors, which arise when computing inexact elements of HOPE; see \cite[Remark~29]{Kabganitechadaptive}.

\begin{corollary} \label{cor:rel:sublevel}
Let $p> 1$, $x^0\in\R^n$, and $\gf: \R^n\to \Rinf$ be a proper lsc function that is coercive. 
If for some $r>0$ and $\lambda\in \R$, we have
\[
\mathcal{S}_1:=\left\{x\in \R^n \mid \gf(x)\leq \gf(x^0)+\lambda\right\}\subseteq \mb(0; r),
\]
then there exists some $\widehat{\gamma}>0$ such that for each $\gamma\in (0, \widehat{\gamma}]$,
\[
\mathcal{S}_2:=\left\{x\in \R^n \mid \fgam{\gf}{p}{\gamma}(x)\leq \fgam{\gf}{p}{\gamma}(x^0)+\lambda\right\}
\subseteq \mb(0; r).
\]
\end{corollary}
\begin{proof}
From Proposition~\ref{pro:rel:sublevel}, there exists some $\widehat{\gamma}>0$ such that for each $\gamma\in (0, \widehat{\gamma}]$, 
\begin{equation}\label{eq:cor:rel:sublevel:a}
\mathcal{S}_3:=\left\{x\in \R^n \mid \fgam{\gf}{p}{\gamma}(x)\leq \gf(x^0)+\lambda\right\}\subseteq \mb(0; r).
\end{equation}
Moreover, invoking Fact~\ref{fact:horder:Bauschke17:p12.9}, for each $\gamma>0$  it holds that $\fgam{\gf}{p}{\gamma}(x^0)\leq \gf(x^0)$, implying $\mathcal{S}_2\subseteq \mathcal{S}_3\subseteq \mb(0; r)$.
\end{proof}

\begin{remark}
 Considering the notation in Corollary~\ref{cor:rel:sublevel}, for certain important classes of functions, such as weakly convex functions, it is possible to find some $\ov{\gamma}>0$ such that, for any  $r> \gf(x_0)+\lambda $ and  
 $\gamma\in (0, \bs\min\{\ov{\gamma}, \widehat{\gamma}\}]$, we have $\mathcal{S}_2\subseteq \mb(0; r)$ and $\fgam{\gf}{p}{\gamma}$ is differentiable on $\mathcal{S}_2$. This result, discussed in more detail in \cite{Kabganitechadaptive}, provides the groundwork for developing gradient methods based on the gradient of $\fgam{\gf}{p}{\gamma}$. 
\end{remark}

%%%%%%%%%%%%%%%%%%%%%%%%%%%%%

By Fact~\ref{fact:horder:Bauschke17:p12.9}, for any $\gamma>0$, $\mathcal{L}(\gf, \lambda)\subseteq \mathcal{L}(\fgam{\gf}{p}{\gamma}, \lambda)$.
However, the gap between the graphs of $\gf$ and $\fgam{\gf}{p}{\gamma}$ explains why the reverse inclusion cannot be expected. The following example shows that even under convexity, 
the conclusion of Proposition~\ref{pro:rel:sublevel} does not necessarily hold for all $\gamma>0$.
\begin{example}\label{ex:sublvelinc}
Let $p=2$ and let $\gf: \R\to\R$  be defined by $\gf(x)=\vert x-2\vert$. Setting $r=1.35$ and $\ov{x}=2$, subfigure~\ref{fig:ex:sublvelinc}~(a) illustrates  that with $\gamma=1$ and $\lambda=1$, although $\mathcal{L}(\gf, \lambda)\subseteq \mb(\ov{x}; r)$,  it holds that $\mathcal{L}(\fgam{\gf}{p}{\gamma}, \lambda)\nsubseteq \mb(\ov{x}; r)$. However, by decreasing $\gamma$ as shown in subfigure~\ref{fig:ex:sublvelinc}~(b), the result of Proposition~\ref{pro:rel:sublevel} is satisfied.
\end{example}
\begin{figure}[ht]
        \subfloat[$\gamma=1$ and $p=2$]
        {\includegraphics[width=7.3cm]{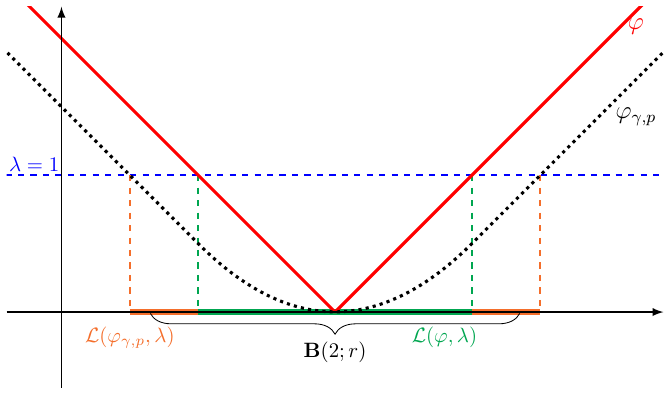}}\qquad\qquad
        \subfloat[$\gamma=0.4$ and $p=2$]
        {\includegraphics[width= 7.3cm]{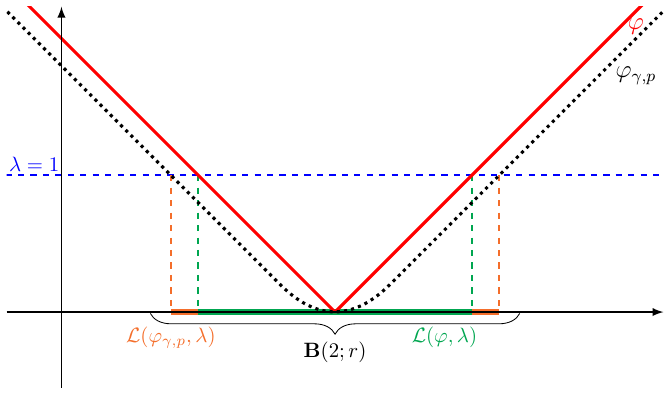}}
        \vspace{2mm}
        \caption{
        Graphs of $\gf$ and $\fgam{{\gf}}{p}{\gamma}$ with different $\gamma$ in Example~\ref{ex:sublvelinc} and the relationship between sublevel sets.
         \label{fig:ex:sublvelinc}}
    \end{figure}

\subsection{{\bf First-order criticality}} 
\label{subsec:critical}
In \cite{Kabgani24itsopt}, the relationships among the reference points of  $\gf$ and $ \fgam{\gf}{p}{\gamma}$ have been discussed. In the following, we recall these notions and discuss their relationships with 
\textit{$p$-calm points}.

\begin{definition}[Reference points]\label{def:critic}
Let $p>1$ and $\gh: \R^n \to \Rinf$ be a proper lsc function, and let $\ov{x}\in \dom{\gh}$. Then, $\ov{x}$ is called
\begin{enumerate}[label=(\textbf{\alph*}), font=\normalfont\bfseries, leftmargin=0.7cm]
  \item \label{def:critic:f} a \textit{Fr\'{e}chet critical point} if $0\in \widehat{\partial}\gh(\ov{x})$, denoted by $\ov{x}\in \bs{\rm Fcrit}(\gh)$;
 \item \label{def:critic:m} a \textit{Mordukhovich critical point} if $0\in \partial \gh(\ov{x})$, denoted by $\ov{x}\in \bs{\rm Mcrit}(\gh)$;
  \item \label{def:critic:min} a \textit{minimizer} of $\gh$ on $C\subseteq \mathbb{R}^n$ if $\gh(\ov{x})\leq \gh(x)$ for every $x\in C$, denoted by $\ov{x}\in \argmin{x\in C}\gh(x)$;
   \item \label{def:critic:pfix} a \textit{proximal fixed point} if $\ov{x}\in \prox{\gh}{\gamma}{p}(\ov{x})$, denoted by $\ov{x}\in \bs{\rm Fix}(\prox{\gh}{\gamma}{p})$;
    \item \label{def:critic:calm} a \textit{$p$-calm point} of $\gh$ with constant $M>0$ if
$\gh(x)+M\Vert x - \ov{x}\Vert^p>\gh(\ov{x})$, for each $x\in \R^n$ with $x\neq \ov{x}$.
\end{enumerate}
\end{definition}
While definitions~$\ref{def:critic:f}-\ref{def:critic:pfix}$ are known in the literature, the concept of $p$-calmness is inspired by \cite[Assumptions~4.1]{Poliquin96}. This notion is particularly helpful for establishing the differentiability of HOME in the vicinity of a $p$-calm point $\ov{x}$ (see Section~\ref{sec:on diff}).
%%%%%%%%%%%%%%%%%%%%%%%%%%%%%%%%%%%%%%%%%%%%%%%%%%%%%%%
%%%%%%%%%%%%%%%%%%%%%%%%%%%%%%%%%%%%%%%%%%%%%%%%%%%%%%%
In the subsequent results, we explore the relationships among the concepts introduced in 
Definition~\ref{def:critic}.
\begin{fact}[Relationships among reference points]\cite[Theorem~3.6]{Kabgani24itsopt}\label{prop:relcrit}
Let $p>1$ and $\gf: \R^n \to \Rinf$ be a proper lsc function and let $\ov{x}\in \dom{\gf}$. 
 If $\gf$ is high-order prox-bounded with a threshold $\gamma^{\gf, p}>0$  and $\bs\inf_{y\in\R^n} \gf(y)\neq -\infty$, then, for each $\gamma\in (0, \gamma^{\gf, p})$, we have
\begin{align*}
    \argmin{x\in \R^n}\gf(x)=\argmin{y\in \R^n} \fgam{\gf}{p}{\gamma}(y)&\subseteq\bs{\rm Fcrit}(\fgam{\gf}{p}{\gamma})\\&\subseteq \bs{\rm Mcrit}(\fgam{\gf}{p}{\gamma})\subseteq \bs{\rm Fix}(\prox{\gf}{\gamma}{p})\subseteq
\bs{\rm Fcrit}(\gf)\subseteq \bs{\rm Mcrit}(\gf).
\end{align*}
\end{fact}
See Remark~\ref{rem:revofrel} for examples showing that the reverse inclusions in Fact~\ref{prop:relcrit} may fail.
%%%%%%%%%%%%%%%%%%%%%%%%%%%%%%%%%%%%%%%%%%%%%%%%%%%%%%%
The following theorem elucidates the relationships among calmness, high-order prox-boundedness, minimizers, and proximal fixed points.
\begin{theorem}[Relationships with $p$-calmness]\label{lem:progpcalm}
Let $p>1$ and $\gf: \R^n \to \Rinf$ be a proper lsc function and let $\ov{x}\in \dom{\gf}$.
Then, the following statements hold:
\begin{enumerate}[label=(\textbf{\alph*}), font=\normalfont\bfseries, leftmargin=0.7cm]
\item \label{lem:progpcalm:mincalm} If $\ov{x}\in \argmin{x\in \R^n}\gf(x)$, then $\ov{x}$ is a $p$-calm point of
$\gf$ with any constant $M>0$;
\item \label{lem:progpcalm:critic} If $\prox{\gf}{\gamma}{p}(\ov{x})=\{\ov{x}\}$, then $\ov{x}$ is a $p$-calm point of
$\gf$ with constant $M=\frac{1}{p\gamma}$;
\item \label{lem:progpcalm:critic2} If $\ov{x}$ is a $p$-calm point of $\gf$ with constant $M$, then for each $\gamma\in \left(0, \frac{1}{pM}\right]$, we have  $\prox{\gf}{\gamma}{p}(\ov{x})=\{\ov{x}\}$;
\item \label{lem:progpcalm:critic3} If for some $\widehat{\gamma}>0$, we have $\ov{x}\in \bs{\rm Fix}(\prox{\gf}{\widehat{\gamma}}{p})$, then for each $\gamma\in (0, \widehat{\gamma})$, we have $\prox{\gf}{\gamma}{p}(\ov{x})=\{\ov{x}\}$, which in turn implies that $\ov{x}$ is a $p$-calm point of
$\gf$ with constant $M=\frac{1}{p\gamma}$;
\item \label{lem:progpcalm:calm} If $\ov{x}$ is a $p$-calm point of $\gf$,  then $\gf$ is high-order prox-bounded.
\end{enumerate}
\end{theorem}
\begin{proof}
$\ref{lem:progpcalm:mincalm}-\ref{lem:progpcalm:critic2}$ These claims directly follow from the definitions.\\
$\ref{lem:progpcalm:critic3}$ Given that $\ov{x}\in \bs{\rm Fix}(\prox{\gf}{\widehat{\gamma}}{p})$, for any $x\in\R^n$ and for any $\gamma\in (0, \widehat{\gamma})$, we have
\[\gf(\ov{x})\leq \gf(x)+\frac{1}{p\widehat{\gamma}}\Vert x - \ov{x}\Vert^p<\gf(x)+\frac{1}{p\gamma}\Vert x - \ov{x}\Vert^p.\]
Thus,  $\prox{\gf}{\gamma}{p}(\ov{x})=\{\ov{x}\}$. The second part comes from 
Assertion~$\ref{lem:progpcalm:critic}$.
\\
$\ref{lem:progpcalm:calm}$ By definition and using the inequality $\Vert x - \ov{x}\Vert^p\leq 2^{p-1}\left(\Vert x\Vert^p+\Vert \ov{x}\Vert^p\right)$, there exists some $M>0$ such that
\[
\gf(x)+M2^{p-1}\left(\Vert x\Vert^p+\Vert \ov{x}\Vert^p\right)\geq \gf(x)+M\Vert x - \ov{x}\Vert^p>\gf(\ov{x}),\qquad \forall x\in\R^n,
\]
which implies high-order prox-boundedness by Proposition~\ref{lemma:charac:sprox}.
\end{proof}

%%%%%%%%%%%%%%%%%%%%%%%%%%%%%%%%%%%%%%%%%%%%%%%%%%%%%%%
%%%%%%%%%%%%%%%%%%%%%%%%%%%%%%%%%%%%%%%%%%%%%%%%%%%%%%%

The relationships among calm points and other concepts, such as minimizers and proximal fixed points, highlighted in Theorem~\ref{lem:progpcalm}, underscore the significance of analyzing these points. Consequently, in the next section, we establish the properties of HOME and HOPE around these points. It is important to note that, while $p$-calmness implies high-order prox-boundedness, it is a stronger condition in the sense that high-order prox-boundedness does not imply $p$-calmness at each point, as illustrated by the following example.

\begin{example}\label{exam:hopbnotcalm}
Let $\gf: \R\to\R$ be given as 
\begin{equation*}\label{eq:hopbnotcalm}
\gf(x):=\left\{\begin{array}{ll}
                 -\vert x\vert & x\in [-1, 1],\\
                 |x|-2~~ & \text{otherwise}.
               \end{array}\right.
\end{equation*}
This function is lower-bounded and, consequently, high-order prox-bounded. However, at $\ov{x}=0$, for every $M>0$ and $p>1$, there exists some $x\in \R$ such that $\gf(x)+M\Vert x - \ov{x}\Vert^p<\gf(\ov{x})$. Thus, $\ov{x}$  is not a $p$-calm point of $\gf$.
\end{example}

%%%%%%%%%%%%%%%%%%%%%%%%%%%%%%%%%%%%%%%%%%%%%%%%%%%%%%%%%%%%%%%%%%%%%%%%%%%%%%%%%%
%%%%%%%%%%%%%%%%%%%%%%%%%%%%%%%%%%%%%%%%%%%%%%%%%%%%%%%%%%%%%%%%%%%%%%%%%%%%%%%%%%
\section{On differentiability and weak smoothness of HOME}\label{sec:on diff}
In this section, we examine the differentiability and weak smoothness of HOME around $p$-calm points. Using $q$-prox-regularity (Definition~\ref{def:pprox-regular}), we show differentiability when $q \geq 2$ and $p \in (1,2]$ or $2 \leq p \leq q$ (Theorems~\ref{th:dif:proxreg12}~and~\ref{th:dif:proxreg}). We also discuss weak smoothness under these conditions (Theorems~\ref{th:dif:proxreg12:weaksm}~and~\ref{th:dif:proxreg:weak}).

\subsection{{\bf Differential properties of HOME}}
\label{subsec:on diff}
Our first result illustrates the relationship between the continuous differentiability of HOME and the single-valuedness of HOPE, which is a direct consequence of \cite[Proposition~3.1]{KecisThibault15}.

\begin{proposition}[Characterization of differentiability of HOME]\label{th:diffcharact}
Let $p > 1$ and $\gf: \R^n \to \Rinf$ be a proper lsc function that is high-order prox-bounded with a threshold $\gamma^{\gf, p} > 0$. Then, for each $\gamma \in (0, \gamma^{\gf, p})$ and open subset $U \subseteq \R^n$, the following statements are equivalent:
\begin{enumerate}[label=(\textbf{\alph*}), font=\normalfont\bfseries, leftmargin=0.7cm]
\item \label{th:diffcharact:diff} $\fgam{\gf}{p}{\gamma} \in \mathcal{C}^{1}(U)$;
\item \label{th:diffcharact:prox} $\prox{\gf}{\gamma}{p}$ is nonempty, single-valued, and continuous on $U$.
\end{enumerate}
Under these conditions, for any $x \in U$ and $y = \prox{\gf}{\gamma}{p}(x)$, we have $\nabla\fgam{\gf}{p}{\gamma}(x) = \frac{1}{\gamma} \Vert x - y \Vert^{p-2} (x - y)$.
\end{proposition}
%%%%%%%%%%%%%%%%%%%%%%%%%
From Proposition~\ref{th:diffcharact}~$\ref{th:diffcharact:prox}$, it is clear that the differentiability of $\fgam{\gf}{p}{\gamma}$ depends on single-valuedness and continuity of $\prox{\gf}{\gamma}{p}$. Next, we explore assumptions that guarantee these desired conditions.

%%%%%%%%%%%%%%%%%%%%%%%
The following definition specifies the class of prox-regular functions and some of its specific subclasses.
 
 %%%%%%%%%%%%%%%%%%%%%%%%%%%%%%%%%%%%%%%%%%%%%%%%%%%%%%%

%%%%%%%%%%%%%%%%%%%%%%%%%%%%%%%%%%%%%%%%%%%%%%%%%%%%%%%
\begin{definition}[$q$-prox-regularity]\label{def:pprox-regular}
Let $\gf: \R^n \to \Rinf$ be a proper lsc function and $\ov{x} \in \dom{\gf}$. Then, $\gf$ is called \textit{$q$-prox-regular} at $\ov{x}$ for $\ov{\zeta} \in \partial \gf(\ov{x})$ and $q \geq 2$ if there exist $\varepsilon > 0$ and $\rho \geq 0$ such that
\begin{equation}\label{eq:def:qprox}
\gf(x') \geq \gf(x) + \langle \zeta, x' - x \rangle - \frac{\rho}{2} \Vert x' - x \Vert^q, \qquad \forall x' \in \mb(\ov{x}; \varepsilon),
\end{equation}
whenever $x \in \mb(\ov{x}; \varepsilon)$, $\zeta \in \partial \gf(x) \cap \mb(\ov{\zeta}; \varepsilon)$, and $\gf(x) < \gf(\ov{x}) + \varepsilon$.
\end{definition}
%%%%%%%%%%%%%%%%%%%%%%%%%%%%%%%%%%%%%%%%%%%%%%%%%%%%%%%%%%%%%%%%
When $q = 2$, $q$-prox-regularity is simply referred to as prox-regularity.
The following remark highlights several points concerning the definition of $q$-prox-regularity. 

%%%%%%%%%%%%%%%%%%%%%%%%%%%%%%%%%%%%%%%%%%%%%%%%%%%%%%%%%%%%%%%%
\begin{remark}\label{rem:remonqprox}
\begin{enumerate}[label=(\textbf{\alph*}), font=\normalfont\bfseries, leftmargin=0.7cm]
\item\label{rem:remonqprox:d} We note that $q$-prox-regularity with $q \in (1, 2)$ is particularly interesting, as it encompasses paraconvex functions \cite{Rolewicz00} and functions with H\"{o}lder continuous gradients. However, the study of this class of functions lies beyond the scope of this paper and will be addressed in a separate work.

\item \label{rem:remonqprox:b} The main motivation behind introducing $q$-prox-regularity lies in their behavior around $p$-calm points to achieve differentiability of $\fgam{\gf}{p}{\gamma}$, as will be discussed in
Theorems~\ref{th:dif:proxreg12}~and~\ref{th:dif:proxreg}. In fact, if $\bar{x}$ is a $p$-calm point, then by Theorem~\ref{lem:progpcalm}~\ref{lem:progpcalm:critic2},  with an appropriate choice of $\gamma$, we have $\ov{x}\in \bs{\rm Fix}(\prox{\gf}{\gamma}{p})$. Consequently, from 
Fact~\ref{prop:relcrit},  it follows that $\ov{x}\in \bs{\rm Mcrit}(\gf)$. Now, substituting $x=\ov{x}$ and $\zeta=0$ into \eqref{eq:def:qprox}, we obtain
 \[
\gf(x') \geq \gf(\ov{x}) - \frac{\rho}{2} \Vert x' - \ov{x} \Vert^q, \qquad \forall x' \in \mb(\ov{x}; \varepsilon).
\]
Thus, by choosing some $M> \frac{\rho}{2}$, we establish a local $q$-calmness  for $\ov{x}$. Later, we utilize the power $p$ of the $p$-calmness property in the HOPE operator while allowing $q\geq p$ for the case where $p\geq 2$; see Theorem~\ref{th:dif:proxreg}.
\end{enumerate}
\end{remark}
%%%%%%%%%%%%%%%%%%%%%%%%%%%%%%%%%%%%%%%%%%%%%%%%%%%%%%%%%%%%%%%%%
The following proposition, together with the subsequent example, clarifies the relationship between prox-regularity and $q$-prox-regularity.
\begin{proposition}[$q$-prox-regularity implies prox-regularity]\label{prop:qproxprox}
Let $\gf: \R^n \to \Rinf$ be a proper lsc function and $\ov{x} \in \dom{\gf}$. If $\gf$ is $q$-prox-regular at $\ov{x}$ for $\ov{\zeta} \in \partial \gf(\ov{x})$ and $q \geq 2$, then it is  prox-regular at $\ov{x}$ for $\ov{\zeta} \in \partial \gf(\ov{x})$.
\end{proposition}
\begin{proof}
Let $\gf$ be $q$-prox-regular at $\ov{x}$ for $\ov{\zeta} \in \partial \gf(\ov{x})$ with $q \geq 2$, $\varepsilon > 0$, and $\rho \geq 0$. Then, $q$-prox-regularity is preserved when $\varepsilon$ is shrunk. Hence, without loss of generality, we may assume $\varepsilon<\tfrac{1}{2}$ and that \eqref{eq:def:qprox} holds for all $x' \in \mb(\ov{x}; \varepsilon)$
whenever $x \in \mb(\ov{x}; \varepsilon)$, $\zeta \in \partial \gf(x) \cap \mb(\ov{\zeta}; \varepsilon)$, and $\gf(x) < \gf(\ov{x}) + \varepsilon$. For such $x, x' \in \mb(\ov{x}; \varepsilon)$, we have 
\[
\Vert x' - x \Vert\leq \Vert x' - \ov{x} \Vert + \Vert x - \ov{x} \Vert<2\varepsilon<1.
\]
Thus, since $q\geq 2$, $\Vert x' - x \Vert^q\leq \Vert x' - x \Vert^2$.
Consequently, for all $x' \in \mb(\ov{x}; \varepsilon)$,
\begin{align*}
\gf(x') &\geq \gf(x) + \langle \zeta, x' - x \rangle - \frac{\rho}{2} \Vert x' - x \Vert^q
\\  &\geq \gf(x) + \langle \zeta, x' - x \rangle - \frac{\rho}{2} \Vert x' - x \Vert^2,
\end{align*}
whenever $x \in \mb(\ov{x}; \varepsilon)$, $\zeta \in \partial \gf(x) \cap \mb(\ov{\zeta}; \varepsilon)$, and $\gf(x) < \gf(\ov{x}) + \varepsilon$.
Hence, by choosing $\varepsilon$ sufficiently small in Definition~\ref{def:pprox-regular}, every $q$-prox-regular function (with $q\geq 2$) is prox-regular.
\end{proof}
%%%%%%%%%
While every $q$-prox-regular function (with $q\ge 2$) is prox-regular, the converse need not hold. The following example, deals with this statement.
\begin{example}[Prox-regular but not $q$-prox-regular for any $q>2$]\label{exam:prnotq}
Consider $\gf:\R\to\R$ given by
$\gf(x) = x^4 - x^2$. While this function is prox-regular at $\ov{x} = 0$ for $\ov{\zeta}=0\in \partial \gf(\ov{x})=\{\nabla \gf(\ov{x})$\}, it is not $q$-prox-regular for any $q > 2$. 
To see this, it suffices to substitute $x=\ov{x}$ in \eqref{eq:def:qprox}. Clearly, for any $\rho>0$ and $q>2$, there always exists an $ x' \in \mb(\ov{x}; \varepsilon)$ such that $\gf(x')=x'^4 - x'^2< - \frac{\rho}{2} \Vert x' \Vert^q$.
\end{example}

As noted in Example~\ref{exam:prnotq}, there exist prox-regular functions which are not $q$-prox-regular for any $q > 2$.
On the other hand, convex and locally convex functions as important subclasses of prox-regular functions are $q$-prox-regular for all $q>2$. 
Moreover, as the following example demonstrates, even certain nonconvex functions, such as lsc piecewise convex functions with jumps at their endpoints, are also $q$-prox-regular for every $q>2$.

\begin{example}[Piecewise convex functions with jumps at the endpoints]
 If $\gh: \R\to \R$ is a lsc piecewise convex function with jumps at the endpoints, we can represent its domain, potentially by excluding some boundary points, as the union of intervals $\bigcup_{i=1}^m (a_i, b_i]$, where $m\in \mathbb{N}$ and $a_i<b_i= a_{i+1}$ for each $1\leq i\leq m$.  On each interval $(a_i, b_i]$,  $\gh$ is convex, and there exist constants  $\theta_i, \lambda_i>0$ such that $\gh(b_i)+\lambda_i<\gh(b_i+\mu_i)$ for each $\mu_i\in (0, \theta_i)$. For an endpoint $b_i$, there exists an $\varepsilon_i>0$ such that
   \[C_i:=\{x\in \R\mid \gh(x) < \gh(b_i) + \varepsilon_i\}\cap \mb(b_i; \varepsilon_i)\subseteq (a_i, b_i].\]Thus, $\gh$ is convex on $C_i$, and the conditions outlined in Definition~\ref{def:pprox-regular} are satisfied. The generalization of this result to an lsc piecewise convex function $\gh: \R^n \to \R$ with a jump at boundary points is straightforward.
     This class of functions frequently arises in practical applications, such as in constrained convex optimization problems, where it is reformulated as an unconstrained optimization problem by incorporating an indicator function - an lsc piecewise convex function with jumps at the boundary points of the constraint set. Further examples include discontinuous piecewise linear optimization \cite{Conn98}, the $\ell_0$-norm regularization \cite{yang2023projective}, piecewise regression models \cite{lu2023advanced}, and other similar applications.
\end{example}

%%%%%%%%%%%%%%%%%%%%%%%%%%%%%%%%%%%%%%%%%%%%%%%%%%%%%%%
The differentiability of $\fgam{\gf}{p}{\gamma}$ in the vicinity of $\ov{x} \in \R^n$ was explored in the seminal work \cite{Poliquin96}, for the case $p = 2$ and when $\gf$ is prox-regular at $\ov{x}$ for $\ov{\zeta} \in \partial \gf(\ov{x})$. For further insights, we refer the readers to \cite[Proposition~13.37]{Rockafellar09}.
A natural inquiry is whether an analogous result holds for $p \neq 2$ under $q$-prox-regularity. In the following, we focus on this question for both cases: $p > 2$ and $p \in (1, 2]$.

Note that, as opposed to convex functions, $q$-prox-regular functions are not necessarily high-order prox-bounded, even when $q = p = 2$. For instance, consider the following example.

\begin{example}\label{ex:nondif:prox:prox}
Let $\gf: \R \to \R$ be defined as $\gf(x) = -\exp(x^2)$. Since $\gf \in \mathcal{C}^2(\R)$, it is a prox-regular function. For any $\gamma > 0$ and $x \in \R$, we have 
\[
\mathop{\bs\lim}\limits_{|y| \to \infty} \frac{-\exp(y^2)}{\frac{1}{2\gamma} |y - x|^2} = -\infty.
\]
As $|y| \to \infty$, the term $-\exp(y^2)$ dominates $\frac{1}{2\gamma} |y - x|^2$, causing the function 
\[f(y) = -\exp(y^2) + \frac{1}{2\gamma} |y - x|^2,\]
to tend towards $-\infty$, regardless of how small $\gamma$ is. Therefore, for any $\gamma > 0$ and $x \in \R$, we have $\fgam{\gf}{2}{\gamma}(x) = -\infty$.
\end{example}

If $\gf$ is both prox-regular and $2$-calm at $\ov{x}$, then with an appropriately selected $\gamma$, the function $\fgam{\gf}{2}{\gamma}$ is continuously differentiable in a neighborhood of $\ov{x}$ \cite[Theorem~4.4]{Poliquin96}. However, the following example demonstrates that prox-regularity alone does not ensure the differentiability of $\fgam{\gf}{p}{\gamma}$ when $p>2$, even under $2$-calmness.

\begin{example}[Nondifferentiability of HOME under prox-regularity]\label{ex:nondif:prox:p3}
Let $\gf: \R \to \R$ be defined by $\gf(x) = x^4 - x^2$, which is prox-regular at any $x \in \R$ and $2$-calm at $\ov{x} = 0$. Let $p > 2$ and $\gamma > 0$ be arbitrary. Considering the optimality condition for $\bs\inf_{y \in \R}\left\{y^4 - y^2 + \frac{1}{p\gamma}|y|^p\right\}$, we obtain
$4y^3 - 2y + \frac{1}{\gamma}\vert y\vert^{p-2}y = 0$.
The critical point $y = 0$ is a local maximizer. Defining $f(y) = 4y^2 - 2 + \frac{1}{\gamma}|y|^{p-2}$, we observe that $f(0) < 0$ and $f(1) = f(-1) > 0$, indicating the existence of at least two critical points corresponding to global minimizers. Consequently, for any $p > 2$ and $\gamma > 0$, $\prox{\gf}{\gamma}{p}(\ov{x})$ is not single-valued, indicating the nondifferentiability of $\fgam{\gf}{p}{\gamma}$ at $\ov{x}$, as indicated in Proposition~\ref{th:diffcharact}.
\end{example}
%%%%%%%%%%%%%%%%%%%%%%%%%%%%%%%%%%%%%%%%%%%%%%%%%
%%%%%%%%%%%%%%%%%%%%%%%%%%%%%%%%%%%%%%%%%%%%%%%%%

%%%%%%%%%%%%%%%%%%%%%%%%%%%%%%%%%%%%%%%%%%%%%%%%%%%%%%%
%%%%%%%%%%%%%%%%%%%%%%%%%%%%%%%%%%%%%%%%%%%%%%%%%%%%%%%
Now, we address the differentiability of $\fgam{\gf}{p}{\gamma}$ and the single-valuedness of $\prox{\gf}{\gamma}{p}$ under $q$-prox-regularity and $p$-calmness. 
Before going into details, we provide the following lemma which is helpful in the rest of this section.
In Fact~\ref{th:level-bound+locally uniform}, we showed that under the high-order prox-boundedness of the function $\gf$, by selecting an appropriate $\gamma$, the set $\prox{\gf}{\gamma}{p}(x)$ is nonempty and bounded for each $x \in \R^n$. In the following lemma, we demonstrate the \textit{uniform boundedness} of HOPE around a $p$-calm point; that is, there exists a neighborhood $U$ of the $p$-calm point $\ov{x}$ and a constant $\varepsilon > 0$ such that for each $x \in U$, we have
$\prox{\gf}{\gamma}{p}(x) \subseteq \mb(\ov{x}; \varepsilon)$.
For the sake of simplicity, we assume that $\ov{x} = 0$ is a $p$-calm point of $\gf$ and $\gf(\ov{x}) = 0$.

%%%%%%%%%%%%%%%%%%%%%%%%%%%%%%%%%%%%%%%%%%%%%%%%%
%%%%%%%%%%%%%%%%%%%%%%%%%%%%%%%%%%%%%%%%%%%%%%%%%
\begin{lemma}[Uniform boundedness of HOPE]\label{lem:prox:pcalm}
Let $p > 1$ and $\gf: \R^n \to \Rinf$ be a proper lsc function. Suppose $\ov{x} = 0$ is a $p$-calm point of $\gf$ with constant $M > 0$ and $\gf(\ov{x}) = 0$. Then, for any $\gamma \in \left(0, \frac{2^{1-p}}{Mp}\right)$ and any $\varepsilon > 0$, there exists a neighborhood $U$ of $\ov{x}$ such that for any $x \in U$, we have $\prox{\gf}{\gamma}{p}(x) \neq \emptyset$. Moreover, if $y \in \prox{\gf}{\gamma}{p}(x)$, then
\[
\Vert y \Vert < \varepsilon, \quad \gf(y) < \varepsilon, \quad \text{and} \quad \frac{1}{\gamma} \Vert x - y \Vert^{p-1} < \varepsilon.
\]
\end{lemma}

\begin{proof}
For $x\in\R^n$ and an arbitrary $\delta>0$, it holds that
\begin{equation}\label{eq1:lem:prox:pcalm}
\fgam{\gf}{p}{\gamma}(x)\leq \gf(0)+\frac{1}{p\gamma}\Vert x\Vert^p=\frac{1}{p\gamma}\Vert x\Vert^p,
\end{equation}
and there exists $y\in\R^n$ depending on $\delta$ such that
\begin{equation}\label{eq1b:lem:prox:pcalm}
\gf(y)+\frac{1}{p\gamma}\Vert x-y\Vert^p\leq \fgam{\gf}{p}{\gamma}(x)+\delta.
\end{equation}
For such $y$, from \eqref{eq1:lem:prox:pcalm} and the $p$-calmness at $\ov{x}=0$, we have
\begin{equation}\label{eq2:lem:prox:pcalm}
-M\Vert y\Vert^p+\frac{1}{p\gamma}\Vert x-y\Vert^p\leq \gf(y)+\frac{1}{p\gamma}\Vert y - x\Vert^p\leq \fgam{\gf}{p}{\gamma}(x)+\delta\leq \frac{1}{p\gamma}\Vert x\Vert^p+\delta.
\end{equation}
%Since
%$2^{1-p}\Vert y\Vert^p-\Vert x\Vert^p\leq \Vert y-x\Vert^p$.
Together with Fact~\ref{lemma:ineq:inequality p}~\ref{lemma:ineq:inequality p:ineq1}, \eqref{eq2:lem:prox:pcalm} yields
\begin{align*}
 -M\Vert y\Vert^p+\frac{1}{p\gamma}\left(2^{1-p}\Vert y\Vert^p-\Vert x\Vert^p\right)\leq \frac{1}{p\gamma}\Vert x\Vert^p+\delta,
\end{align*}
i.e.,
$
(2^{1-p}-Mp\gamma)\Vert y\Vert^p\leq 2\Vert x\Vert^p+p\gamma\delta
$.
Setting $\mu:=(2^{1-p}-Mp\gamma)^{-1}$, which is positive, we have
\begin{equation}\label{eq4:lem:prox:pcalm}
\Vert y\Vert\leq \left(2\mu\Vert x\Vert^p+p\gamma\delta\mu\right)^{\frac{1}{p}}.
\end{equation}
In addition, from \eqref{eq1:lem:prox:pcalm} and \eqref{eq1b:lem:prox:pcalm}, we obtain
\begin{equation}\label{eq5:lem:prox:pcalm}
\gf(y)\leq \frac{1}{p\gamma}\Vert x\Vert^p+\delta.
\end{equation}
Now, we define $C:=\{y\in\R^n: \Vert y\Vert\leq \varepsilon, \gf(y)\leq \varepsilon\}$,
which is bounded due to $\Vert y\Vert\leq \varepsilon$ and closed due to the lower semicontinuity of $\gf$, making it a compact set. Let us choose small $\delta>0$ and $\varrho>0$ such that
\begin{align}\label{eq6a:lem:prox:pcalm}
\left(2\mu\varrho^p+p\gamma\delta\mu\right)^{\frac{1}{p}}< \varepsilon;
\quad
\frac{1}{p\gamma}\varrho^p+\delta< \varepsilon;\quad
\frac{\left(1+\left(2\mu\right)^{\frac{1}{p}}\right)^{p-1}}{\gamma}\varrho^{p-1}< \varepsilon.
\end{align}
 Define the neighborhood $U:=\{x\in\R^n: \Vert x\Vert< \varrho\}$.  
 Let $x\in U$ and $y$ satisfy \eqref{eq1b:lem:prox:pcalm}. Then, by \eqref{eq4:lem:prox:pcalm} and the first inequality in \eqref{eq6a:lem:prox:pcalm}, we have $\Vert y\Vert<\varepsilon$, and by \eqref{eq5:lem:prox:pcalm} and the second inequality in \eqref{eq6a:lem:prox:pcalm}, we also have
$\gf(y)<\varepsilon$. Consequently, if $x\in U$ and $y$ satisfies \eqref{eq1b:lem:prox:pcalm}, then $y\in C$. Specifically,
if $x\in U$, then
\[\prox{\gf}{\gamma}{p}(x)=\argmin{y\in C} \left(\gf(y)+\frac{1}{p\gamma}\Vert x- y\Vert^p\right).\]
Since $\gf$ is lsc and $C$ is compact,  $\prox{\gf}{\gamma}{p}(x)\neq \emptyset$  for any  $x\in U$.
Assume that  $x\in U$ and $y\in \prox{\gf}{\gamma}{p}(x)$. Then, $y$ satisfies \eqref{eq1b:lem:prox:pcalm} for any $\delta>0$. Thus, by letting $\delta\downarrow 0$ and using \eqref{eq4:lem:prox:pcalm}, we get
$\Vert y\Vert\leq \left(2\mu\right)^{\frac{1}{p}}\Vert x\Vert$, i.e., $\Vert x - y\Vert\leq \left(1+\left(2\mu\right)^{\frac{1}{p}}\right)\Vert x\Vert$. Moreover, by the third inequality in \eqref{eq6a:lem:prox:pcalm}, we obtain
\[
 \frac{1}{\gamma}\Vert x-y\Vert^{p-1}\leq \frac{\left(1+\left(2\mu\right)^{\frac{1}{p}}\right)^{p-1}}{\gamma}\Vert x\Vert^{p-1}< \varepsilon,
\]
giving our desired result.
\end{proof}
%%%%%%%%%%%%%%%%%%%%%%%%%%%%%%%%%%%%%%%%%%%%%%%%%%%%%%%
The following corollary gives conditions ensuring the uniform boundedness of HOPE under prox-boundedness, without requiring $p$-calmness.

\begin{corollary}[Uniform boundedness of HOPE under prox-boundedness]\label{cor:prox:pcalm}
Let $p>1$ and let $\gf:\R^n \to \Rinf$ be a proper lsc function that is high-order prox-bounded with threshold $\gamma^{\gf,p}$ and $\gf(\ov{x})=0$ for $\ov{x}=0$. Then, for each $\widehat{\gamma}\in (0,\gamma^{\gf,p})$ and sufficiently small 
$\gamma \in \left(0, \frac{2^{1-p}}{l p}\right)$, where $\ell:=\frac{2^{p-1}}{p\widehat{\gamma}}$, and any $\varepsilon > 0$, there exists a neighborhood $U$ of $\ov{x}$ such that for any $x \in U$, we have $\prox{\gf}{\gamma}{p}(x) \neq \emptyset$. Moreover, if $y \in \prox{\gf}{\gamma}{p}(x)$, then
\[
\Vert y \Vert < \varepsilon, \quad \gf(y) < \varepsilon, \quad \text{and} \quad \frac{1}{\gamma} \Vert x - y \Vert^{p-1} < \varepsilon.
\]
\end{corollary}
\begin{proof}
By Proposition~\ref{lemma:charac:sprox}, for each $\widehat{\gamma} \in (0, \gamma^{\gf, p})$, with setting $\ell:=\frac{2^{p-1}}{p\widehat{\gamma}}$, there exists $\ell_0 \in \mathbb{R}$ such that $\gf(\cdot) + \ell\| \cdot \|^p$ is bounded from below by $\ell_0$ on $\mathbb{R}^n$. Given that $\gf(\ov{x}) = 0$, we have $\ell_0 \leq 0$.
To follow the proof of Lemma~\ref{lem:prox:pcalm}, we substitute $M = \ell$, choose $\gamma \in \left(0, \frac{2^{1-p}}{Mp}\right)$, and add $\ell_0$ to the left-hand side of \eqref{eq2:lem:prox:pcalm} to derive an inequality similar to \eqref{eq4:lem:prox:pcalm} as
\[\|y\| \leq \left(2 \mu \|x\|^p + p \gamma \mu (\delta - \ell_0)\right)^{\frac{1}{p}}.\]
Note that $\delta - \ell_0 \geq 0$ for any $\delta>0$. 
We must then choose small $\delta > 0$, $\varrho > 0$, and $\gamma > 0$ such that, analogous to the first inequality in \eqref{eq6a:lem:prox:pcalm},
\begin{equation}\label{eq:cor:prox:pcalm}
    \left(2 \mu \varrho^p + p \gamma \mu (\delta - \ell_0)\right)^{\frac{1}{p}} < \varepsilon.
\end{equation}
Note that, in addition to choosing $\gamma \in \left(0, \frac{2^{1-p}}{M p}\right)$, $\gamma$ must be sufficiently small such that \eqref{eq:cor:prox:pcalm} holds.
    
\end{proof}
%%%%%%%%%%%%%%%%%%%%%%%%%%%%%%%%%%%%%%%%%%%%%%%%%%%%%%%

The following remark outlines assumptions used throughout this section and Subsection~\ref{subsec:on weak}.

\begin{remark}\label{rem:regardingthsdiff}
In Theorems~\ref{th:dif:proxreg12}~and~\ref{th:dif:proxreg}, as well as 
Theorems~\ref{th:dif:proxreg12:weaksm}~and~\ref{th:dif:proxreg:weak}, we formulate our results around $\ov{x} = 0$, assuming that $\ov{x}$ is a $p$-calm point of $\gf$ with constant $M > 0$ and $\gf(\ov{x}) = 0$. We also select $\gamma < \frac{2^{1-p}}{Mp} \leq \frac{1}{Mp}$. 
Thus, by Theorem~\ref{lem:progpcalm}~\ref{lem:progpcalm:critic2}, we have $\ov{x} \in \bs{\rm Fix}(\prox{\gf}{\gamma}{p})$, which, by Fact~\ref{prop:relcrit}, implies that $\ov{x} \in \bs{\rm Mcrit}(\gf)$, i.e., $0\in \partial \gf(\ov{x})$. This allows us to assume $q$-prox-regularity with $q \geq 2$ at $\ov{x} = 0$ for $\ov{\zeta} = 0 \in \partial \gf(\ov{x})$.
\end{remark}

First, we investigate the differentiability of HOME under $q$-prox-regularity and $p$-calmness for $p \in (1,2]$. 
\begin{theorem}[Differentiability of HOME under $q$-prox-regularity for $q\geq 2$ and $p\in(1,2\text{]}$]\label{th:dif:proxreg12}
Let $p\in (1,2]$ and $\gf: \R^n\to \Rinf$ be a proper lsc function. Suppose $\ov{x} = 0$ is a $p$-calm point of $\gf$ with constant $M > 0$ and $\gf(\ov{x}) = 0$.
Then, for each $\gamma\in \left(0, \frac{2^{1-p}}{Mp}\right)$ when $p\in (1,2)$ and $\gamma\in \left(0, \bs\min\left\{\frac{1}{4M}, \frac{1}{\rho}\right\}\right)$ when $p=2$, under the assumption that $\gf$ is $q$-prox-regular with $q\geq 2$ at $\ov{x}$ for $\ov{\zeta}=0\in\partial \gf(\ov{x})$ and $\rho>0$, there exists a neighbourhood $U$ of $\ov{x}$ such that $\prox{\gf}{\gamma}{p}$ is single-valued and continuous on $U$, and $\fgam{\gf}{p}{\gamma}\in \mathcal{C}^{1}(U)$.
\end{theorem}
\begin{proof}
Suppose $\gf: \R^n\to \Rinf$ is $q$-prox-regular at $\ov{x}=0$ for $\ov{\zeta}=0$ with some $\varepsilon>0$ and $\rho>0$.
Since the $q$-prox-regularity at $\ov{x}$ for $\ov{\zeta}$ is preserved under shrinking the neighborhood, we may, after possibly replacing $\varepsilon$ by a smaller value and relabeling it as $\varepsilon$, assume that 
\[2\varepsilon<\bs\min\left\{\left(\frac{\kappa_p}{\gamma\rho}\right)^{\frac{1}{2-p}},1\right\},\]
when $p\in (1, 2)$ and $\varepsilon<\frac{1}{2}$ when $p=2$, where $\kappa_p$ is introduced in \eqref{eq:formofKs:def}. As such, for each $x'\in \mb(\ov{x}; \varepsilon)$,
\begin{align}\label{eq1ex1:th:dif:proxreg}
\gf(x')\geq \gf(x)+\langle \zeta, x'-x\rangle-\frac{\rho}{2}\Vert x'-x\Vert^q
\overset{(i)}{\geq} \gf(x)+\langle \zeta, x'-x\rangle-\frac{\rho}{2}\Vert x'-x\Vert^2,
\end{align}
when $\Vert x\Vert <\varepsilon$, $\zeta\in \partial \gf(x)$, $\Vert \zeta\Vert< \varepsilon$, and $\gf(x)<\varepsilon$.
Note that the inequality $(i)$ follows from the fact that $\Vert x'-x\Vert\leq\Vert x'-\ov{x}\Vert+\Vert x- \ov{x}\Vert<2\varepsilon<1$ as $\varepsilon<\frac{1}{2}$. Since $q\geq 2$ and
$\Vert x'-x\Vert<1$, we have $\Vert x'-x\Vert^q \leq \Vert x'-x\Vert^2$.
By Lemma~\ref{lem:prox:pcalm}, there exists a neighbourhood $U\subseteq \mb(0; \varepsilon)$ such that for each $x\in U$, $\prox{\gf}{\gamma}{p}(x)\neq \emptyset$. Additionally, if $y\in \prox{\gf}{\gamma}{p}(x)$, then
$\frac{1}{\gamma}\Vert x-y\Vert^{p-2}(x-y)\in \partial \gf(y)$ and 
\[
\Vert y\Vert< \varepsilon,\quad \gf(y)< \varepsilon,\quad \left\Vert\frac{1}{\gamma}\Vert x-y\Vert^{p-2}(x-y)\right\Vert=\frac{1}{\gamma}\Vert x-y\Vert^{p-1}< \varepsilon.
\]
Considering $x_i\in U$ and $y_i\in \prox{\gf}{\gamma}{p}(x_i)$, $i=1,2$, and \eqref{eq1ex1:th:dif:proxreg}, we come to
\begin{equation}\label{eq1:th:dif:proxreg}
\gf(y_2)\geq \gf(y_1)+\frac{1}{\gamma}\Vert x_1-y_1\Vert^{p-2}\langle x_1-y_1 , y_2-y_1\rangle-\frac{\rho}{2}\Vert y_2-y_1\Vert^2,
\end{equation}
and
\begin{equation}\label{eq2:th:dif:proxreg}
\gf(y_1)\geq \gf(y_2)+\frac{1}{\gamma}\Vert x_2-y_2\Vert^{p-2}\langle x_2-y_2 , y_1-y_2\rangle-\frac{\rho}{2}\Vert y_1-y_2\Vert^2.
\end{equation}
Adding \eqref{eq1:th:dif:proxreg} and \eqref{eq2:th:dif:proxreg}, results in
\begin{align*}
\frac{1}{\gamma} \langle \Vert x_2-y_2\Vert^{p-2} (x_2-y_2)- \Vert x_1-y_1\Vert^{p-2} (x_1-y_1) , y_2-y_1\rangle&\geq  -\rho\Vert y_1-y_2\Vert^2.
\end{align*}
Since $\Vert x_i -y_i\Vert\leq r_1:=2\varepsilon$ ($i=1,2$),  
Lemma~\ref{lem:findlowbounknu:lemma}~$\ref{lem:findlowbounknu:lemma:e2}$ yields
 with $r_2:=\kappa_p r_1^{p-2}$, 
\begin{align*}
\langle \Vert x_2-y_2\Vert^{p-2}(x_2-y_2) - \Vert x_1-y_1\Vert^{p-2}(x_1-y_1),& (x_2-y_2)-(x_1-y_1)\rangle
\\&\geq  r_2\Vert (x_2-y_2) - (x_1-y_1)\Vert^{2}\\
&\geq r_2\big\vert\Vert y_2-y_1\Vert -\Vert x_2-x_1\Vert\big\vert^{2}.
\end{align*}
Note that
$\Vert y_2-y_1\Vert\leq r_1$ and if $a, b\in \R$, then $\vert a -b\vert^{2} -a^{2}\geq -2\vert a\vert \vert b\vert$.
Setting $a=\Vert y_2-y_1\Vert$ and $b=\Vert x_2-x_1\Vert$, we have
\[\big\vert\Vert y_2-y_1\Vert -\Vert x_2-x_1\Vert\big\vert^{2}-\Vert y_2-y_1\Vert^{2}\geq -2 r_1 \Vert x_2-x_1\Vert,\]
i.e.,
\begin{align*}
  &r_2\left(2r_1 \Vert x_2-x_1\Vert-\Vert y_2-y_1\Vert^{2}\right) 
  \\&\hspace{2cm}\geq
\langle \Vert x_2-y_2\Vert^{p-2}(x_2-y_2) - \Vert x_1-y_1\Vert^{p-2}(x_1-y_1), (x_1-y_1)-(x_2-y_2)\rangle
\\
&\hspace{2cm}=\langle \Vert x_2-y_2\Vert^{p-2}(x_2-y_2) - \Vert x_1-y_1\Vert^{p-2}(x_1-y_1), x_1-x_2\rangle\\
&\hspace{2cm}~~~+\langle \Vert x_2-y_2\Vert^{p-2}(x_2-y_2) - \Vert x_1-y_1\Vert^{p-2}(x_1-y_1), y_2-y_1\rangle\\
&\hspace{2cm}\geq \langle \Vert x_2-y_2\Vert^{p-2}(x_2-y_2) - \Vert x_1-y_1\Vert^{p-2}(x_1-y_1), x_1-x_2\rangle-\rho\gamma\Vert y_2 - y_1\Vert^{2}
\\
&\hspace{2cm}\geq -\left(\Vert x_2-y_2\Vert^{p-1}+\Vert x_1-y_1\Vert^{p-1}\right)\Vert x_ 2-x_1\Vert-\rho\gamma\Vert y_2 - y_1\Vert^{2}\\
&\hspace{2cm}\geq-2r_1^{p-1}\Vert x_2-x_1\Vert-\rho\gamma\Vert y_2 - y_1\Vert^{2},
\end{align*}
leading to
$(r_2-\rho\gamma)\Vert y_2 - y_1\Vert^{2}\leq  (2r_1r_2+2r_1^{p-1}) \Vert x_2-x_1\Vert$.
Since $r_2-\rho\gamma>0$, this ensures
\begin{equation}\label{lochol:eq11}
\Vert y_2 - y_1\Vert\leq \left(\frac{2r_1r_2+2r_1^{p-1}}{r_2-\rho\gamma}\right)^{\frac{1}{2}} \Vert x_2-x_1\Vert^{\frac{1}{2}}.
\end{equation}
From \eqref{lochol:eq11}, the single-valuedness and continuity of $\prox{\gf}{\gamma}{p}(x)$ for any $x\in U$ are obtained. Moreover,
by invoking Proposition~\ref{th:diffcharact}, $\fgam{\gf}{p}{\gamma}\in \mathcal{C}^{1}(U)$.
\end{proof}

%%%%%%%%%%%%%%%%%%%%%%%%%%%%%%%%%%%%%%%%%%%%%%%%%%%%%%%%%%%%%%%%%
%%%%%%%%%%%%%%%%%%%%%%%%%%%%%%%%%%%%%%%%%%%%%%%%%%%%%%%%%%%%%%%%%

\begin{theorem}[Differentiability of HOME under $q$-prox-regularity for $2\leq p\leq q$]\label{th:dif:proxreg}
Let  $p\geq 2$ and $\gf: \R^n\to \Rinf$ be a proper lsc function. Suppose $\ov{x} = 0$ is a $p$-calm point of $\gf$ with constant $M > 0$ and $\gf(\ov{x}) = 0$.
Then, for each $\gamma\in \left(0, \bs\min\left\{\frac{2^{1-p}}{Mp},\frac{1}{\rho 2^{2p-3}}\right\}\right)$,  under the assumption that $\gf$ is $q$-prox-regular with $q\geq p$ at $\ov{x}$ for $\ov{\zeta}=0\in\partial \gf(\ov{x})$ and $\rho>0$,
there exists a neighborhood $U$ of $\ov{x}$ such that 
$\prox{\gf}{\gamma}{p}$ is single-valued and continuous on $U$, and $\fgam{\gf}{p}{\gamma}\in \mathcal{C}^{1}(U)$. 
\end{theorem}
\begin{proof}
Suppose that $\gf: \R^n \to \Rinf$ is $q$-prox-regular at $\ov{x} = 0$ for $\ov{\zeta} = 0$ with constants $\varepsilon > 0$ and $\rho > 0$. By shrinking $\varepsilon$ if necessary, such that $\varepsilon < \frac{1}{2}$, and following a similar approach to the proof of Theorem~\ref{th:dif:proxreg12}, there exists a neighborhood $U \subseteq \mb(0; \varepsilon)$ such that if $x_i \in U$ and $y_i \in \prox{\gf}{\gamma}{p}(x_i)$ for $i = 1,2$, then
 \begin{align}\label{eq1-3:th:dif:proxreg}
\langle \Vert x_1-y_1\Vert^{p-2} (x_1-y_1) - \Vert x_2-y_2\Vert^{p-2} (x_2-y_2) , y_1-y_2\rangle\geq -\gamma\rho \Vert y_1-y_2\Vert^q
\geq -\gamma\rho \Vert y_1-y_2\Vert^p.
\end{align}
Moreover, applying Lemmas~\ref{lemma:ineq:inequality p}~\ref{lemma:ineq:inequality p:ineq3}~and~\ref{lem:findlowbounknu:lemma}~\ref{lem:findlowbounknu:lemma:e3}, and setting $s = \frac{p}{p-1}$ yield
\begin{align*}
&\langle \Vert x_1-y_1\Vert^{p-2} (x_1-y_1) - \Vert x_2-y_2\Vert^{p-2} (x_2-y_2) , y_1-y_2\rangle
\\&~~~=\langle \Vert x_1-y_1\Vert^{p-2} (x_1-y_1) - \Vert x_2-y_2\Vert^{p-2} (x_2-y_2) , (x_2-y_2)-(x_1-y_1)-(x_2-x_1)\rangle
\\&~~~\leq -\left(\frac{1}{2}\right)^{p-2}\Vert (x_2-y_2)-(x_1-y_1)\Vert^p
\\&~~~~~~~~+\left\Vert \Vert x_1-y_1\Vert^{p-2} (x_1-y_1) - \Vert x_2-y_2\Vert^{p-2} (x_2-y_2)\right\Vert \Vert x_1-x_2\Vert
\\&~~~\leq -\left(\frac{1}{2}\right)^{p-2}\Vert (x_2-y_2)-(x_1-y_1)\Vert^p
\\&~~~~~~~~+ \frac{2(2\varepsilon)^{p-2}}{\kappa_s}\Vert (x_1-y_1)-(x_2- y_2)\Vert \Vert  x_1-x_2\Vert
\\&~~~\leq -\left(\frac{1}{2}\right)^{p-2}\Vert (x_2-y_2)-(x_1-y_1)\Vert^p+ \frac{2(2\varepsilon)^{p-2}}{\kappa_s}4\varepsilon \Vert  x_1-x_2\Vert
\\&~~~= -\left(\frac{1}{2}\right)^{p-2}\Vert (x_2-y_2)-(x_1-y_1)\Vert^p+ \frac{2^{p+1}\varepsilon^{p-1}}{\kappa_s} \Vert  x_1-x_2\Vert.
\end{align*}
Thus, from \eqref{eq1-3:th:dif:proxreg},
\begin{align*}
& -\gamma\rho \Vert y_1-y_2\Vert^p
\leq -\left(\frac{1}{2}\right)^{p-2}\Vert (x_2-y_2)-(x_1-y_1)\Vert^p+\frac{2^{p+1}\varepsilon^{p-1}}{\kappa_s}\Vert  x_2-x_1\Vert,
\end{align*}
i.e.,
\[
-\gamma\rho \Vert y_1-y_2\Vert^p+\left(\frac{1}{2}\right)^{p-2}\left(\left(\frac{1}{2}\right)^{p-1}\Vert y_2-y_1\Vert^p-\Vert  x_2-x_1\Vert^p\right)\leq \frac{2^{p+1}\varepsilon^{p-1}}{\kappa_s} \Vert  x_2-x_1\Vert,
\]
leading to
\begin{align*}
\left(\frac{1}{2}\right)^{p-1}\left(\left(\frac{1}{2}\right)^{p-2}-\gamma\rho 2^{p-1}\right)\Vert y_2-y_1\Vert^p&\leq \left(\frac{1}{2}\right)^{p-2}\Vert  x_2-x_1\Vert^p+\frac{2^{p+1}\varepsilon^{p-1}}{\kappa_s} \Vert  x_2-x_1\Vert\\
&\leq \left(\left(\frac{1}{2}\right)^{p-2}(2\varepsilon)^{p-1}+\frac{2^{p+1}\varepsilon^{p-1}}{\kappa_s}\right) \Vert  x_2-x_1\Vert.
\end{align*}
Since $\gamma<\frac{1}{\rho 2^{2p-3}}$, we have
\begin{align}\label{eq1g2:lochol:eq11}
\Vert y_2-y_1\Vert^p\leq \widehat{L}_p \Vert  x_2-x_1\Vert,
\end{align}
where
\begin{equation}\label{eq1g2:lochol:eq11:b}
    \widehat{L}_p:=\left(\frac{\left(\frac{1}{2}\right)^{p-2}(2\varepsilon)^{p-1}+\frac{2^{p+1}\varepsilon^{p-1}}{\kappa_s}}{\left(\frac{1}{2}\right)^{p-1}\left(\left(\frac{1}{2}\right)^{p-2}-
\gamma\rho 2^{p-1}\right)}\right).
\end{equation}
From \eqref{eq1g2:lochol:eq11}, the single-valuedness of $\prox{\gf}{\gamma}{p}(x)$ for any $x \in U$ follows. By invoking Proposition~\ref{th:diffcharact}, it follows that $\fgam{\gf}{p}{\gamma} \in \mathcal{C}^{1}(U)$.
\end{proof}

 Now, we can add another piece to the chain given in Fact~\ref{prop:relcrit}.
 
 \begin{corollary}[Extended relationships among reference points]\label{prop:relcrit2}
 Under the assumptions of Theorem~\ref{th:dif:proxreg12} or Theorem~\ref{th:dif:proxreg}, with an appropriate choice of $\gamma>0$, we have
\begin{align*}
    \argmin{x\in \R^n}\gf(x)&=\argmin{y\in \R^n} \fgam{\gf}{p}{\gamma}(y)\subseteq\bs{\rm Fcrit}(\fgam{\gf}{p}{\gamma})
    \subseteq \bs{\rm Mcrit}(\fgam{\gf}{p}{\gamma})
    \\
&\subseteq \bs{\rm Fix}(\prox{\gf}{\gamma}{p})=\bs{\rm{Zero}}(\nabla\fgam{\gf}{p}{\gamma}):= \{x\in \R^n\mid \nabla\fgam{\gf}{p}{\gamma}(x) =0\}
\subseteq\bs{\rm Fcrit}(\gf)\subseteq \bs{\rm Mcrit}(\gf).
\end{align*}
\end{corollary}

The following remark discusses reverse relations in Corollary~\ref{prop:relcrit2}
 \begin{remark}\label{rem:revofrel}
Let us consider the assumptions of Corollary~\ref{prop:relcrit2}.
\begin{enumerate}[label=(\textbf{\alph*}), font=\normalfont\bfseries, leftmargin=0.7cm]
\item \label{rem:revofrel:a}
 It is possible that $\bs{\rm Fix}(\prox{\gf}{\gamma}{p}) \nsubseteq  \argmin{x \in \R^n} \gf(x)$: For $p=2$, $\gamma=0.5$, and $\gf(x) = \cos(x)$, we have $0 \in \bs{\rm Fix}(\prox{\gf}{\gamma}{p})$, but $0 \notin \argmin{x \in \R^n} \gf(x)$.
  
  \item \label{rem:revofrel:b} It is also possible that $\bs{\rm Fcrit}(\gf) \nsubseteq \bs{\rm Fix}(\prox{\gf}{\gamma}{p})$: See \cite[Example~3.6]{Themelis18}.
 
  \end{enumerate}
  \end{remark}

%%%%%%%%%%%%%%%%%%%%%%%%%%%%%%%%%%%%%%%%%%%%%%%%%%%%%%%%%%%%%%%%%%%%%%%%%%%%%%%%%%%%
\subsection{{\bf Weak smoothness of HOME}}\label{subsec:on weak}
A function is called \textit{weakly smooth} if it is differentiable with H\"older continuous gradient; see, e.g., \cite{nesterov2015universal,ahookhosh2019accelerated}.
Here, we address the weak smoothness of $\fgam{\gf}{p}{\gamma}$ and the H\"{o}lder continuity of $\prox{\gf}{\gamma}{p}$ under the conditions of $q$-prox-regularity and $p$-calmness. 
In both Theorems~\ref{th:dif:proxreg12:weaksm}~and~\ref{th:dif:proxreg:weak} (see below), we assume that
$\gf$ is $q$-prox-regular with $q\geq 2$ in $\ov{x}=0$ for $\ov{\zeta}=0\in\partial \gf(\ov{x})$ with $\varepsilon<\frac{1}{2}$ and $\rho>0$. Regarding assumption $\varepsilon<\frac{1}{2}$, note that if a function is  $q$-prox-regular with some $\varepsilon>0$, it remains $q$-prox-regular  with any $0<\varepsilon'<\varepsilon$. Hence, without loss of generality, we impose the condition $\varepsilon<\frac{1}{2}$.

We first consider the case of $p$-calmness for $p \in (1,2]$.
\begin{theorem}[Weak smoothness of HOME under $q$-prox-regularity for $q\geq 2$ and $p\in(1,2\text{]}$]\label{th:dif:proxreg12:weaksm}
Let $p\in (1,2]$ and $\gf: \R^n\to \Rinf$ be a proper lsc function. Suppose $\ov{x} = 0$ is a $p$-calm point of $\gf$ with constant $M > 0$ and $\gf(\ov{x}) = 0$.
Then, for each $\gamma\in \left(0, \bs\min\left\{\frac{2^{1-p}}{Mp}, \frac{\kappa_p(2\varepsilon)^{p-2}}{\rho}\right\}\right)$, 
under the assumption that $\gf$ is $q$-prox-regular with $q\geq 2$ at $\ov{x}=0$ for $\ov{\zeta}=0\in\partial \gf(\ov{x})$ with $\varepsilon<\frac{1}{2}$ and $\rho>0$, 
there exists a neighborhood $U\subseteq \mb(\ov{x}; \varepsilon)$ such that
\begin{equation}\label{lochol:maineq:weaksm}
\Vert \prox{\gf}{\gamma}{p}(x_2) -  \prox{\gf}{\gamma}{p}(x_1) \Vert\leq L_p \Vert  x_2-x_1\Vert^{\frac{1}{q}}, \qquad \forall x_1, x_2\in U,
\end{equation}
and $\fgam{\gf}{p}{\gamma}\in\mathcal{C}^{1, \frac{p-1}{q}}_{\mathcal{L}_p}(U)$, i.e.,
 \begin{align}\label{locholof gra:maineq:weaksm}
 \left\Vert \nabla \fgam{\gf}{p}{\gamma}(x_2)-\nabla \fgam{\gf}{p}{\gamma}(x_1)\right\Vert \leq \mathcal{L}_p\Vert x_2-x_1\Vert^{\frac{p-1}{q}},\qquad \forall x_1, x_2\in U,
  \end{align}
  where $L_p:=\left(\frac{4\varepsilon(\kappa_p+1)}{\kappa_p-\rho\gamma(2\varepsilon)^{2-p}}\right)^{\frac{1}{q}}$ and 
  $\mathcal{L}_p:= \frac{2^{2-p}}{\gamma}\left((2\varepsilon)^{\frac{q-1}{q}}+L_p \right)^{p-1}$.
 \end{theorem}
\begin{proof}
By the assumptions, we get
\[2\varepsilon<\bs\min\left\{\left(\frac{\kappa_p}{\gamma\rho}\right)^{\frac{1}{2-p}},1\right\},\]
for $p\in (1, 2)$ and $\gamma\in \left(0, \bs\min\left\{\frac{1}{4M}, \frac{1}{\rho}\right\}\right)$ for $p=2$. Hence, from the proof of Theorem~\ref{th:dif:proxreg12}, there exists a neighbourhood $U\subseteq \mb(0; \varepsilon)$ such that for each $x_i\in U$, $i=1,2$, we have
\begin{equation*}
\Vert \prox{\gf}{\gamma}{p}(x_2) - \prox{\gf}{\gamma}{p}(x_1)\Vert^q\leq \Vert \prox{\gf}{\gamma}{p}(x_2) - \prox{\gf}{\gamma}{p}(x_1)\Vert^2\leq 
\left(\frac{2r_1r_2+2r_1^{p-1}}{r_2-\rho\gamma}\right) \Vert x_2-x_1\Vert, 
\end{equation*}
where $r_1:=2\varepsilon$ and $r_2:=\kappa_p r_1^{p-2}$. Thus, we obtain \eqref{lochol:maineq:weaksm}
with $L_p:=\left(\frac{4\varepsilon(\kappa_p+1)}{\kappa_p-\rho\gamma(2\varepsilon)^{2-p}}\right)^{\frac{1}{q}}$.

For $x_1, x_2 \in  U$, we get
  \begin{align*}
\left\Vert \nabla \fgam{\gf}{p}{\gamma}(x_2)-\nabla \fgam{\gf}{p}{\gamma}(x_1)\right\Vert  &=\left\Vert \nabla\left(\frac{1}{p\gamma}\Vert\cdot\Vert^{p}\right)\left(x_2 - \prox{\gf}{\gamma}{p}(x_2)\right)-\nabla\left(\frac{1}{p\gamma}\Vert\cdot\Vert^{p}\right)\left(x_1 - \prox{\gf}{\gamma}{p}(x_1)\right)\right\Vert\\
 &\leq \frac{2^{2-p}}{\gamma} \left\Vert (x_2-x_1) - (\prox{\gf}{\gamma}{p}(x_2)-\prox{\gf}{\gamma}{p}(x_1))\right\Vert^{p-1},
  \end{align*}
  where the last inequality comes from \cite[Theorem~6.3]{Rodomanov2020}. Together with \eqref{lochol:maineq:weaksm}, this ensures
  \begin{align*}
\left\Vert \nabla \fgam{\gf}{p}{\gamma}(x_2)-\nabla \fgam{\gf}{p}{\gamma}(x_1)\right\Vert  &\leq \frac{2^{2-p}}{\gamma}\left\Vert (x_2-x_1) - (\prox{\gf}{\gamma}{p}(x_2)-\prox{\gf}{\gamma}{p}(x_1))\right\Vert^{p-1}
 \\
 &\leq \frac{2^{2-p}}{\gamma}\left(\Vert x_2-x_1\Vert + \Vert \prox{\gf}{\gamma}{p}(x_2)-\prox{\gf}{\gamma}{p}(x_1)\Vert\right)^{p-1}\\
 &\leq  \frac{2^{2-p}}{\gamma}\left(\Vert x_2-x_1\Vert^{\frac{q-1}{q}}\Vert x_2-x_1\Vert^{\frac{1}{q}} +L_p  \Vert x_2-x_1\Vert^{\frac{1}{q}}\right)^{p-1}\\
&\leq \frac{2^{2-p}}{\gamma}\left((2\varepsilon)^{\frac{q-1}{q}}+L_p \right)^{p-1} \Vert x_2-x_1\Vert^{\frac{p-1}{q}}.
  \end{align*}
By setting $\mathcal{L}_p:= \frac{2^{2-p}}{\gamma}\left((2\varepsilon)^{\frac{q-1}{q}}+L_p \right)^{p-1}$, we have established  \eqref{locholof gra:maineq:weaksm}.
\end{proof}

%%%%%%%%%%%%%%%%%%%%%%%%%%%%%%%%%%%%%%%%%%%%%%%%%%%%%%%%%%%%%%%%%%
\begin{theorem}[Weak smoothness of HOME under $q$-prox-regularity for $2\leq p\leq q$]\label{th:dif:proxreg:weak}
Let $p\geq 2$ and $\gf: \R^n\to \Rinf$ be a proper lsc function. Suppose $\ov{x} = 0$ is a $p$-calm point of $\gf$ with constant $M > 0$ and $\gf(\ov{x}) = 0$.
Then, for each $\gamma\in  \left(0, \bs\min\left\{\frac{2^{1-p}}{Mp},\frac{1}{\rho 2^{2p-3}}\right\}\right)$, 
under the assumption that $\gf$ is $q$-prox-regular with $q\geq p$ at $\ov{x}=0$ for $\ov{\zeta}=0\in\partial \gf(\ov{x})$ with $\varepsilon<\frac{1}{2}$ and $\rho>0$,
there exists a neighborhood $U\subseteq \mb(\ov{x}; \varepsilon)$ such that,
for each $x_1, x_2\in U$,
\begin{align}\label{eq2g2:lochol:eq11:weaksmo}
\Vert \prox{\gf}{\gamma}{p}(x_2) -  \prox{\gf}{\gamma}{p}(x_1) \Vert\leq L_p \Vert  x_2-x_1\Vert^{\frac{1}{q}},\qquad \forall x_1, x_2\in U,
\end{align}
and $\fgam{\gf}{p}{\gamma}\in\mathcal{C}^{1, \frac{1}{q}}_{\mathcal{L}_p}(U)$, i.e.,
 \begin{align}\label{eq:pg2:locholof gra:maineq:weaksmo}
 \left\Vert \nabla \fgam{\gf}{p}{\gamma}(x_2)-\nabla \fgam{\gf}{p}{\gamma}(x_1)\right\Vert \leq \mathcal{L}_p\Vert x_2-x_1\Vert^{\frac{1}{q}},\qquad \forall x_1, x_2\in U,
  \end{align}
  where $L_p:=\left(
\frac{2\left(1+\frac{2^p}{\kappa_s}\right)\varepsilon^{p-1}}{\frac{1}{2^{2p-3}}-\gamma\rho}
\right)^{\frac{1}{q}}$, $\mathcal{L}_p:= \frac{2(2\varepsilon)^{p-2}}{\gamma \kappa_{s}}\left((2\varepsilon)^{\frac{q-1}{q}}+L_p \right)$, and $s=\frac{p}{p-1}$.
\end{theorem}
\begin{proof}
From the proof of Theorem~\ref{th:dif:proxreg},
there exists a neighborhood $U\subseteq \mb(0; \varepsilon)$ such that if $x_i\in U$, $i=1,2$, then
\[
\Vert\prox{\gf}{\gamma}{p}(x_2) -  \prox{\gf}{\gamma}{p}(x_1) \Vert^q\leq \Vert\prox{\gf}{\gamma}{p}(x_2) -  \prox{\gf}{\gamma}{p}(x_1) \Vert^p\leq \widehat{L}_p \Vert  x_2-x_1\Vert,
\]
where $\widehat{L}_p$ is given in \eqref{eq1g2:lochol:eq11:b}.
Hence, \eqref{eq2g2:lochol:eq11:weaksmo} satisfies with
$
L_p:=\left(
\frac{2\left(1+\frac{2^p}{\kappa_s}\right)\varepsilon^{p-1}}{\frac{1}{2^{2p-3}}-\gamma\rho}
\right)^{\frac{1}{q}}.
$

Assume that $x_1, x_2 \in  U$. From Lemma~\ref{lem:findlowbounknu:lemma}~$\ref{lem:findlowbounknu:lemma:e3}$, by setting $s=\frac{p}{p-1}$, we obtain
  \begin{align*}
 \left\Vert \nabla \fgam{\gf}{p}{\gamma}(x_2)-\nabla \fgam{\gf}{p}{\gamma}(x_1)\right\Vert  &=\left\Vert \nabla\left(\frac{1}{p\gamma}\Vert\cdot\Vert^{p}\right)\left(x_2 - \prox{\gf}{\gamma}{p}(x_2)\right)-\nabla\left(\frac{1}{p\gamma}\Vert\cdot\Vert^{p}\right)\left(x_1 - \prox{\gf}{\gamma}{p}(x_1)\right)\right\Vert\\
 &\leq \frac{2(2\varepsilon)^{p-2}}{\gamma \kappa_{s}} \left\Vert (x_2-x_1) - (\prox{\gf}{\gamma}{p}(x_2)-\prox{\gf}{\gamma}{p}(x_1))\right\Vert.
  \end{align*}
  Together with \eqref{eq2g2:lochol:eq11:weaksmo}, for each $x_1, x_2 \in  U$, this ensures
  \begin{align*}
\left\Vert \nabla \fgam{\gf}{p}{\gamma}(x_2)-\nabla \fgam{\gf}{p}{\gamma}(x_1)\right\Vert  &\leq \frac{2(2\varepsilon)^{p-2}}{\gamma \kappa_s}\left\Vert (x_2-x_1) - (\prox{\gf}{\gamma}{p}(x_2)-\prox{\gf}{\gamma}{p}(x_1))\right\Vert
 \\
 &\leq \frac{2(2\varepsilon)^{p-2}}{\gamma \kappa_{s}}\left(\Vert x_2-x_1\Vert + \Vert \prox{\gf}{\gamma}{p}(x_2)-\prox{\gf}{\gamma}{p}(x_1)\Vert\right)\\
 &\leq  \frac{2(2\varepsilon)^{p-2}}{\gamma \kappa_{s}}\left(\Vert x_2-x_1\Vert^{\frac{q-1}{q}}\Vert x_2-x_1\Vert^{\frac{1}{q}} +L_p  \Vert x_2-x_1\Vert^{\frac{1}{q}}\right)\\
&\leq \frac{2(2\varepsilon)^{p-2}}{\gamma \kappa_{s}}\left((2\varepsilon)^{\frac{q-1}{q}}+L_p \right) \Vert x_2-x_1\Vert^{\frac{1}{q}}.
  \end{align*}
  With $\mathcal{L}_p:= \frac{2(2\varepsilon)^{p-2}}{\gamma \kappa_{s}}\left((2\varepsilon)^{\frac{q-1}{q}}+L_p \right)$, \eqref{eq:pg2:locholof gra:maineq:weaksmo} holds.
\end{proof}
%%%%%%%%%%%%%%%%%%%%%%%%%%%%%%%%%%%%%%%%%%%%%%%%%%%%%%%%%%%%%%%%%%%%%%%%%%%%

%%%%%%%%%%%%%%%%%%%%%%%%%%%%%%%%%%%%%%%%%%%%%%%%%%%%%%%%%%%%%%%%%%%%%%%%%%%%
%%%%%%%%%%%%%%%%%%%%%%%%%%%%%%%%%%%%%%%%%%%%%%%%%%%%%%%%%%%%%%%%%%%%%%%%%%%%%%%%%%%
\section{{\bf High-order proximal-point method for nonsmooth optimization}} \label{sec:applications}
This section deals with optimization problems of the form
\begin{equation}\label{eq:mainproblem}
\mathop{\bs\min}\limits_{x\in \mathbb {R}^n}\  \gf(x),
\end{equation}
where $\gf: \R^n\to \Rinf$ is a proper lsc function that is neither necessarily smooth nor convex.
We assume that the set of minimizers of $\gf$ is nonempty and for $x^*\in \argmin{x\in \R^n} \gf(x)$, denote the corresponding minimal value by $\gf^*$. As described in Section~\ref{intro}, the classical methods to solve general problems of the form \eqref{eq:mainproblem} are subgradient-based methods or proximal-point methods; see, e.g., \cite{bagirov2014introduction,beck2017first, Nesterov2018}; however, our method of interest in this study is proximal point methods. In the classical setting, the power in the regularization term is $p=2$, which may not be the best choice considering the geometry of the underlying function in \eqref{eq:mainproblem}. As such, we aim at developing and analyzing a high-order proximal-point algorithm (HiPPA) for solving \eqref{eq:mainproblem}, demonstrating its convergence properties and practical efficacy for nonsmooth and nonconvex optimization problems.
 
 To tackle \eqref{eq:mainproblem}, we propose the \textit{High-order Proximal-Point Algorithm} (HiPPA), defined by the iterative scheme
\begin{equation}\label{eq:HiPPA}
x^{k+1}\in \prox{\gf}{\gamma}{p}(x^k),
\end{equation}
with parameter $\gamma>0$ and order $p>1$. 
This algorithm generalizes the classical proximal-point method by incorporating a high-order regularization term, offering better adaptability to the geometry of the underlying nonsmooth and nonconvex objective functions. We establish that HiPPA converges subsequentially to a proximal fixed point, which, under appropriate conditions on $\gamma$, is also a $p$-calm point and a critical point of $\gf$. 
In the following, we present the convergence analysis, followed by practical insights derived from our preliminary numerical experiments on the Nesterov-Chebyshev–Rosenbrock functions; see Figure~\ref{fig:3d_functions}.

%%%%%%%%%%%%%%%%%%%%%%%%%%%%%%%%%%%%%%%%%%%%%%%%%%%%%%%%%%%%%%%%%%%%%%%%%%%%%
%\subsection{{\bf Convergence analysis}}
Here, we verify the convergence of the sequence $\{x^k\}_{k\in \Nz}$ generated by HiPPA. Let us begin by showing some key properties of this sequence and analyzing its subsequential convergence.
%%%%%%%%%%%%%%%%%%%%%%%%%%%%%%% 
\begin{theorem}[Subsequential convergence  of HiPPA]\label{th:HiPPA}
Let $\gf: \R^n\to \Rinf$ be a proper lsc function such that the set of minimizers of $\gf$ is nonempty.
Let $\gamma>0$, and let $\{x^k\}_{k\in \Nz}$ be a sequence generated by \eqref{eq:HiPPA}. Then the following hold:
\begin{enumerate}[label=(\textbf{\alph*}), font=\normalfont\bfseries, leftmargin=0.7cm]
  \item \label{prop:propertiesppa:noninc} the sequences $\{\gf(x^k)\}_{k \in \Nz}$ and $\{\fgam{\gf}{p}{\gamma}(x^k)\}_{k \in \Nz}$ are non-increasing;
  \item \label{prop:propertiesppa:Cauchy} $\sum_{k=0}^{\infty}\Vert x^{k+1} - x^k\Vert^p<\infty$ and $\bs\lim_{k\to \infty}(x^k - {x}^{k+1})= 0$;
\item \label{th:comcon:cluster} each cluster point  of $\{x^k\}_{k\in \Nz}$  is a proximal fixed point.
\end{enumerate}
\end{theorem}
\begin{proof}
\ref{prop:propertiesppa:noninc} Since $x^{k+1} \in \prox{\gf}{\gamma}{p}(x^k)$, by Fact \ref{fact:horder:Bauschke17:p12.9}, we have:
\begin{equation}\label{eq1:prop:propertiesppa}
 \gf(x^{k+1})\leq  \gf(x^{k+1})+\frac{1}{p\gamma}\Vert x^{k+1}-x^k\Vert^p=\fgam{\gf}{p}{\gamma}(x^k)\leq \gf(x^k),
\end{equation}
and
\begin{equation}\label{eq2:prop:propertiesppa}
 \fgam{\gf}{p}{\gamma}(x^{k+1})\leq \gf(x^{k+1})=\fgam{\gf}{p}{\gamma}(x^k)-\frac{1}{p\gamma}\Vert x^{k+1}-x^k\Vert^p\leq \fgam{\gf}{p}{\gamma}(x^k),
\end{equation}
These show that both $\{\gf(x^k)\}_{k \in \Nz}$ and $\{\fgam{\gf}{p}{\gamma}(x^k)\}_{k \in \Nz}$ are nonincreasing.
\\
\ref{prop:propertiesppa:Cauchy}
 With reusing \eqref{eq1:prop:propertiesppa}, we have
\[
\sum_{k=0}^{\infty}\Vert x^{k+1} - x^k\Vert^p\leq p\gamma(\gf(x^0)-\gf^*)<\infty.
\]
Consequently, $\bs\lim_{k\to \infty}(x^k - {x}^{k+1})= 0$.
\\
$\ref{th:comcon:cluster}$ 
Let $\widehat{x}$ be a cluster point of $\{x^k\}_{k\in \Nz}$, and let $\{x^j\}_{j\in J\subseteq \Nz}$ be a subsequence converging to $\widehat{x}$.
Since $x^{j+1}\in \prox{\gf}{\gamma}{p}(x^{j})$,
Assertion~$\ref{prop:propertiesppa:Cauchy}$ yields $\Vert x^{j} - x^{j+1}\Vert\rightarrow 0$ as $j\to \infty$, $j\in J$.
Because $\Vert x^j - \widehat{x}\Vert\to 0$
and $\Vert x^{j+1}-\widehat{x}\Vert\leq \Vert x^{j} - x^{j+1}\Vert +\Vert x^j - \widehat{x}\Vert$, it follows that
$x^{j+1}\to \widehat{x}$ as $j\to \infty$, $j\in J$.
By Fact~\ref{th:level-bound+locally uniform}~$\ref{level-bound+locally uniform2:conv}$, we conclude that $\widehat{x}\in \prox{\gf}{\gamma}{p}(\widehat{x})$, which confirms that $\widehat{x}$ is a proximal fixed point.
\end{proof}
%%%%%%%%%%%%%%%%%%%%%%%%%%%%%%%%%

\begin{remark}
In Theorem~\ref{th:HiPPA}, the assumptions ensure that $\gh$ is bounded from below and thus high-order prox-bounded with threshold $\gamma^{\gf,p}=+\infty$. Consequently, by Fact~\ref{th:level-bound+locally uniform}, the proximal operator $\prox{\gf}{\gamma}{p}(x^k)$ is nonempty at every iteration. Moreover, the convergence results in Theorem~\ref{th:HiPPA} remain valid regardless of which element of the proximal mapping is chosen when multiple solutions exist.
\end{remark}

%%%%%%%%%%
In Theorem~\ref{th:HiPPA}, no restriction is imposed on $\gamma$. Thus, for any $\gamma>0$, the sequence $\{x^k\}_{k\in \Nz}$ generated by \eqref{eq:HiPPA} subsequentially converges to a proximal fixed point. By Theorem~\ref{lem:progpcalm}~\ref{lem:progpcalm:critic3}, these fixed points are $p$-calm points, and according to Fact~\ref{prop:relcrit}, they are Mordukhovich critical points.
 Theorems~\ref{th:dif:proxreg12}~and~\ref{th:dif:proxreg} further guarantee local differentiability of HOME around such points, enhancing the theoretical foundation of the algorithm explained in the following remark.

 %%%%%%%%%%%%%%%%%%%%%%%%%%%%%%%%%%%%%%%%%%%%%
\begin{remark}
 Theorems~\ref{th:dif:proxreg12}~and~\ref{th:dif:proxreg} show that, selecting an appropriate $\gamma>0$ and $p>1$, for each $p$-calm point $\ov{x}$ of $\gf$, there exists a neighborhood $U$ of $\ov{x}$ where $\fgam{\gf}{p}{\gamma}\in \mathcal{C}^{1}(U)$. This enables gradient-based iterative algorithms to find critical points of $\fgam{\gf}{p}{\gamma}$, which are also critical points of $\gf$ (see Fact~\ref{prop:relcrit}).
In \cite{Kabganitechadaptive}, a similar strategy is applied to weakly convex functions, extending the region of differentiability with suitable $\gamma$ and $p\in (1, 2]$.
For HiPPA, we allow $p>1$. Under the assumptions of Proposition~\ref{th:diffcharact}, if $\fgam{\gf}{p}{\gamma}$ is continuously differentiable at a point $x\in \R^n$,
then $\prox{\gf}{\gamma}{p}(x)$ is single-valued, and for $y= \prox{\gf}{\gamma}{p}(x)$, we have
\[\nabla\fgam{\gf}{p}{\gamma}(x) = \frac{1}{\gamma} \Vert x - y \Vert^{p-2} (x - y).\]
Simple calculation yields
\begin{equation}\label{eq:proxgrad}
\prox{\gf}{\gamma}{p}(x) = x  - \gamma^{\frac{1}{p-1}}\left\Vert \nabla\fgam{\gf}{p}{\gamma}(x)\right\Vert^{\frac{2-p}{p-1}}\nabla\fgam{\gf}{p}{\gamma}(x),
\end{equation}
generalizing \eqref{eq:relmorprox}, meaning that HiPPA acts as a gradient method near a $p$-calm point.
 \end{remark}
 
Next, we analyze the finite termination of HiPPA with a practical stopping criterion and quantify the quality of the returned solution.

\begin{corollary}\label{th:com}
Let $\{x^k\}_{k\in \Nz}$ be generated by \eqref{eq:HiPPA}. If  the algorithm terminates when
 $\Vert x^k - {x}^{k+1}\Vert\leq \epsilon$ for a given tolerance $\epsilon>0$, then
\begin{enumerate}[label=(\textbf{\alph*}), font=\normalfont\bfseries, leftmargin=0.7cm]
\item \label{rem:comcon:compx}
the algorithm terminates within 
$k\leq \frac{p\gamma(\gf(x^0)-\gf^*)}{\epsilon^p}$ iterations;

\item \label{th:comcon:dist} 
For the returned point ${x}^{k+1}$,  we have 
\[
\dist(0, \partial \gf({x}^{k+1}))\leq \frac{\epsilon^{p-1}}{\gamma}.
\]
\end{enumerate}
\end{corollary}
\begin{proof}
\ref{rem:comcon:compx}
From \eqref{eq2:prop:propertiesppa}, we obtain
\[
\frac{1}{p\gamma} K\bs\min_{0\leq k\leq K-1}\Vert x^k -{x}^{k+1}\Vert^{p}\leq \fgam{\gf}{p}{\gamma}(x^0)-\gf^*\leq \gf(x^0)-\gf^*.
\]
Thus, we require, 
$\frac{p\gamma(\gf(x^0)-\gf^*)}{K}\leq \epsilon^p$.
Hence, the algorithm terminate within some
$k\leq \frac{p\gamma(\gf(x^0)-\gf^*)}{\epsilon^p}$.
\\
\ref{th:comcon:dist}
For ${x}^{k+1}$, it holds that
$\frac{1}{\gamma}\Vert x^k - {x}^{k+1}\Vert^{p-2}(x^k - {x}^{k+1})\in \partial \gf({x}^{k+1})$, i.e., 
\[
\dist(0, \partial \gf({x}^{k+1}))\leq \frac{1}{\gamma}\Vert x^k - {x}^{k+1}\Vert^{p-1}\leq \frac{\epsilon^{p-1}}{\gamma},
\]
adjusting our claims.
\end{proof}

Note that the upper bound for $\dist(0, \partial \gf({x}^{k+1}))$ obtained in Theorem~\ref{th:com}~\ref{th:comcon:dist}
provides an estimate of how close ${x}^{k+1}$ is to a Mordukhovich critical point of $\gf$.
 
\begin{figure}
    \centering
    % First row
    \begin{subfigure}{0.45\textwidth}
        \centering
        \includegraphics[width=1.2\textwidth]{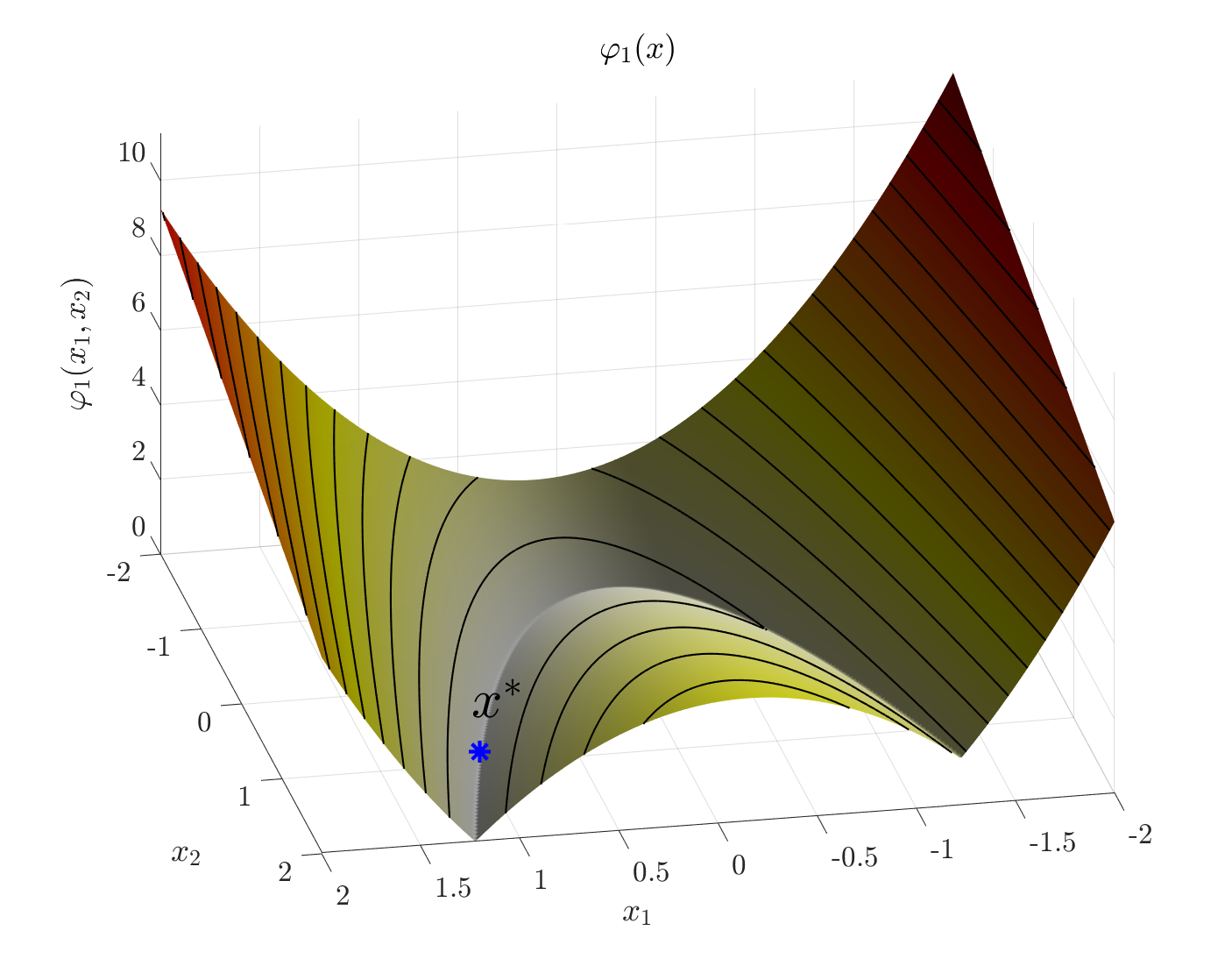}
        \caption{The shape of $\gf_1$ (Example~\ref{ex:nes1})}
    \end{subfigure}
    \qquad~
    \begin{subfigure}{0.45\textwidth}
        \centering
        \includegraphics[width=1.2\textwidth]{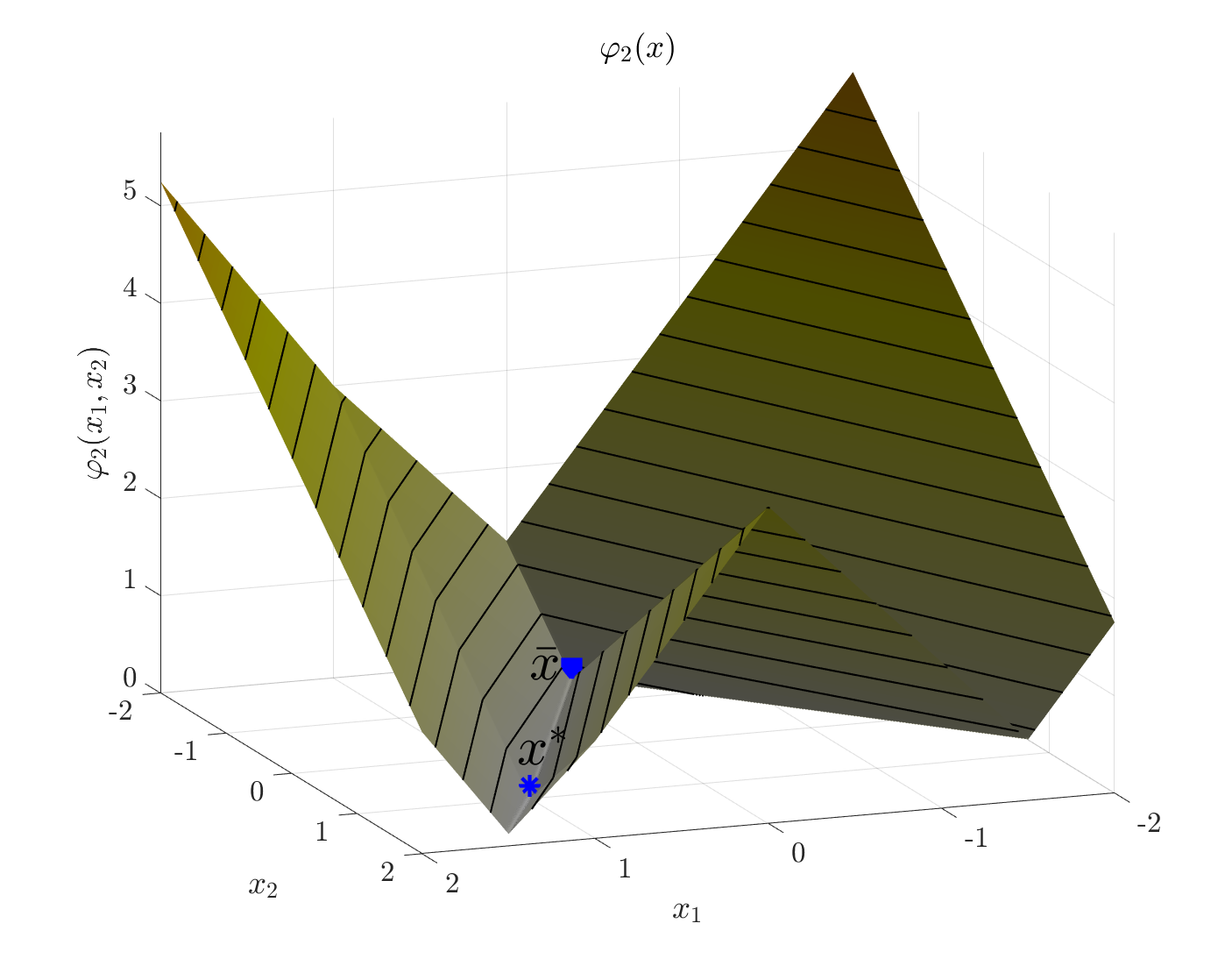}
        \caption{The shape of $\gf_2$ (Example~\ref{ex:nes2})}
    \end{subfigure}
    \vspace{1em}   

    \caption{The shapes of function $\gf_1$ (Example~\ref{ex:nes1}) and $\gf_2$ (Example~\ref{ex:nes2}).}
    \label{fig:3d_functions}
\end{figure}

%%%%%%%%%%%%%%%%%%%%%%%%%%%%%%%%%%%%%%%%%%%%
\subsection{{\bf Minimizing Nesterov-Chebyshev–Rosenbrock functions}}
To demonstrate the promising potential of HiPPA, we study its performance for minimizing  Nesterov-Chebyshev-Rosenbrock's functions, which are nonsmooth and nonconvex given in \cite{GURBUZBALABAN2012Nest}. These examples highlight HiPPA’s ability to locate global minimizers from challenging starting points, outperforming traditional subgradient methods in certain cases. The shape of the function $\gf_1$, introduced in Example~\ref{ex:nes1}, together with its unique minimizer $x^*$, which is also its unique critical point, is shown in Subfigure (a) of Figure~\ref{fig:3d_functions}. Similarly, the shape of the function $\gf_2$, introduced in Example~\ref{ex:nes2}, along with its unique minimizer $x^*$ and another critical point $\ov{x}$, is illustrated in Subfigure (b) of Figure~\ref{fig:3d_functions}.

As HiPPA \eqref{eq:HiPPA} specifies, at the $(k+1)$-th iteration we need to compute an element of
$\prox{\gf}{\gamma}{p}(x^k)$. In the examples below, we approximate this element using the Nelder–Mead simplex algorithm \cite{Audet2017Dfree,Lagarias1998Convergence}, since the problem dimensions are low. We compare our method against the following approaches:

\begin{itemize}
\item \textbf{SG-DSS}: Subgradient method with diminishing step size rule $\frac{\alpha}{\sqrt{k}}$, where $\alpha>0$ and $k$ denotes the iteration \cite{bagirov2014introduction,Davis2018,Rahimi2024};
\item \textbf{NM}: Nelder–Mead simplex algorithm \cite{Audet2017Dfree,Lagarias1998Convergence};
\item \textbf{PBM}: Proximal bundle method with line search \cite{Makla2016Proximal}\footnote{The Fortran implementation is publicly available at \url{https://napsu.karmitsa.fi/proxbundle/}.
}.
\end{itemize}
Let us note that we computed an element of the limiting subdifferential as the subgradient.
All algorithms are terminated after $0.1$ seconds of CPU-time and executed on a laptop equipped with a 12th Gen Intel$\circledR$ Core$^{\text{TM}}$ i7-12800H CPU (1.80 GHz) and 16 GB of RAM, running MATLAB R2025a.

\begin{example}[Nesterov-Chebyshev–Rosenbrock function I]\label{ex:nes1}
Let $\gf_1: \R^2\to \R$ be a function given by
\[
\gf_1(x)= \frac{1}{4}\left(x_1-1\right)^2+\left\vert x_2-2x_1^2 +1\right\vert,
\]
where the unique global minimizer is $x^*=(1,1)$ and $\bs{\rm Mcrit}(\gf_1)=\{x^*\}$.
A challenging starting point for iterative schemes to find the global minimizer of $\gf_1$ is $x^0=(-1,1)$ as reported in \cite{GURBUZBALABAN2012Nest}. Specifically, since $\gf_1$ is not differentiable at $x^0$, optimization algorithms such as the BFGS method, which are designed for smooth optimization problems, cannot be initialized from this point.

We first compare HiPPA with different values of $p\in \{1.5, 2, 2.5, 3, 10, 100\}$ and $\gamma=50$, obtained by tuning, starting from $x^0=(-1,1)$. Subfigure~\ref{fig:hippa:N1}(a) shows the relative errors, indicating that $p=2.5$ works particularly well for this problem. Motivated by this, we then compare SG–DSS, NM, and PBM with HiPPA using $p=2.5$. The corresponding errors are presented in Subfigure~\ref{fig:hippa:N1}(b), which highlight HiPPA’s superior accuracy.

Subfigures~\ref{fig:hippa:N1}(c)– \ref{fig:hippa:N1}(f) further illustrate the trajectories of the algorithms from $x^0$ to $x^*$.
HiPPA follows a notably direct path, efficiently navigating the nonsmooth landscape and reaching $x^*$ within only a few iterations. By contrast, NM and PBM take slightly more circuitous routes but still converge effectively, as shown in Subfigures~\ref{fig:hippa:N1}(d)–\ref{fig:hippa:N1}(e). In fact, Subfigure~\ref{fig:hippa:N1}(b) indicates that NM and PBM achieve essentially the same level of accuracy. Subfigure~\ref{fig:hippa:N1}(f) depicts the trajectory of SG–DSS with a tuned diminishing step size $\alpha=0.98$. Its progress toward $x^*$ is slower and exhibits zigzagging behavior due to the nonsmoothness of $\gf_1$, resulting in less efficient convergence.

Finally, Table~\ref{tab:relerr_inits} reports the relative errors obtained by the four algorithms (NM, HiPPA with $p=2$, SG–DSS, and PBM) under five randomly chosen initializations. HiPPA consistently achieves the smallest relative error across all starting points, often by several orders of magnitude compared to the other methods. Both NM and PBM also perform well, reaching high-accuracy solutions, whereas SG–DSS yields significantly larger errors in all cases. These results demonstrate the robustness of HiPPA with respect to the choice of initialization.
\end{example}
%%%%%%%%%%%%%%%%%%%%%%%%
\begin{table}[h]
\centering
\caption{Relative errors under five initializations for Example~\ref{ex:nes1}.}
\label{tab:relerr_inits}
\begin{tabular}{ccccc}
\toprule
\multicolumn{1}{c}{Init $x^{(0)}$} & \multicolumn{1}{c}{NM}
                                   & \multicolumn{1}{c}{HiPPA ($p=2$)}
                                   & \multicolumn{1}{c}{SG–DSS}
                                   & \multicolumn{1}{c}{PBM} \\
\midrule
$(3.57,\, 2.76)$   & \num{4.83e-07} & \num{4.60e-08} & \num{2.73e-01} &  \num{4.69e-07} \\
$(-1.34,\, 3.03)$    &\num{ 2.73e-07} & \num{6.50e-10} &  \num{1.89e-01} & \num{3.88e-06} \\
$(0.72,\, -0.06)$    & \num{7.72e-07} & \num{4.15e-10} & \num{2.18e-01} & \num{6.13e-06} \\
$(-1.21,\, -1.11)$    & \num{4.23e-07} & \num{1.12e-08} & \num{2.12e-01} &  \num{2.10e-06} \\
$(-1.09,\, 0.03)$    & \num{4.21e-07} & \num{5.41e-11} & \num{1.70e-01} & \num{2.60e-06} \\
\bottomrule
\end{tabular}
\end{table}

%%%%%%%%%%%%%%%%%%%%%%%%%%%%%%%
\begin{figure}
    \centering
    % First row
      \begin{subfigure}{0.42\textwidth}
        \centering
        \includegraphics[width=1.15\textwidth]{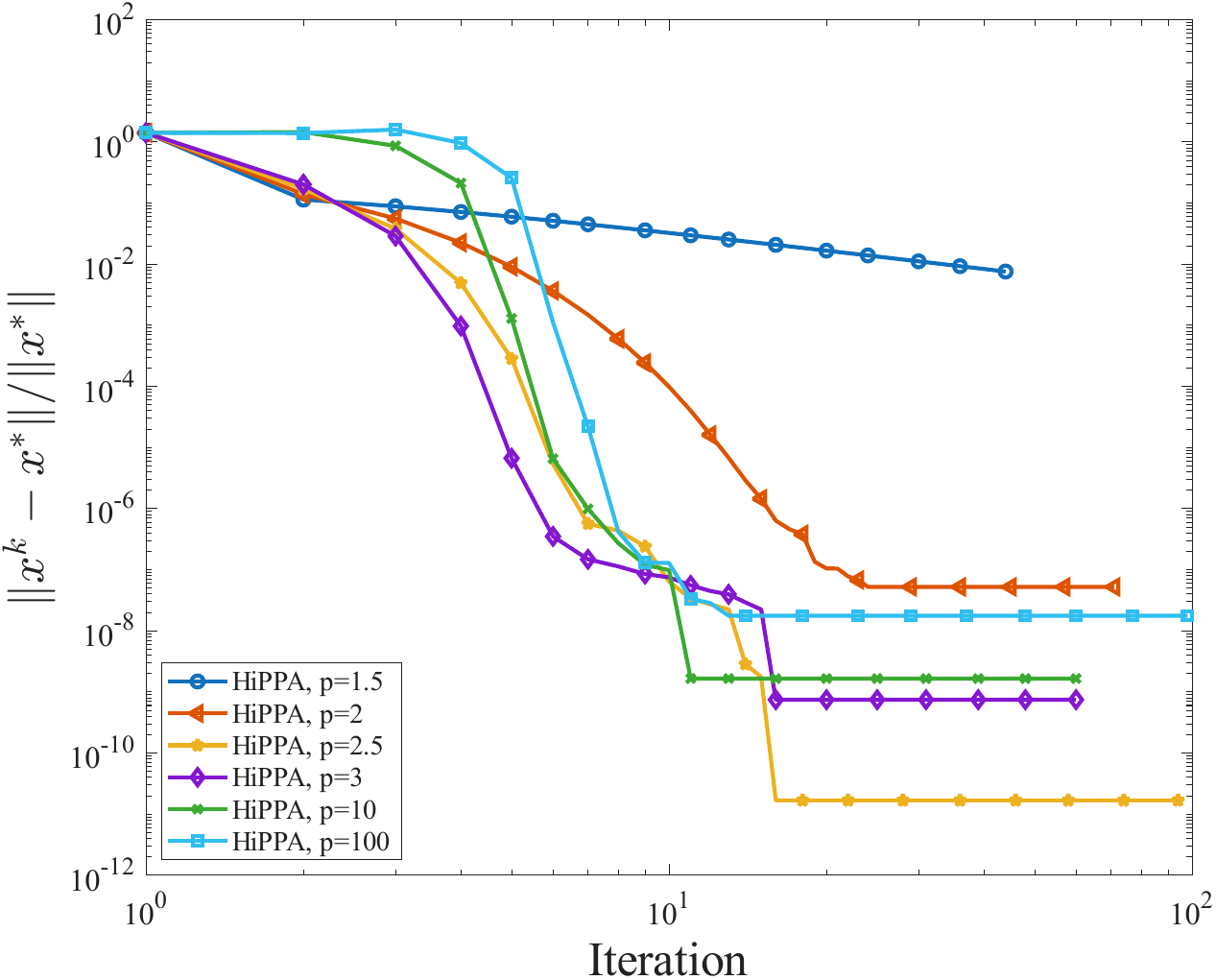}
        \caption{Comparison of relative errors of HiPPA}
    \end{subfigure}
         \qquad\qquad~~
    \begin{subfigure}{0.42\textwidth}
        \centering
        \includegraphics[width=1.15\textwidth]{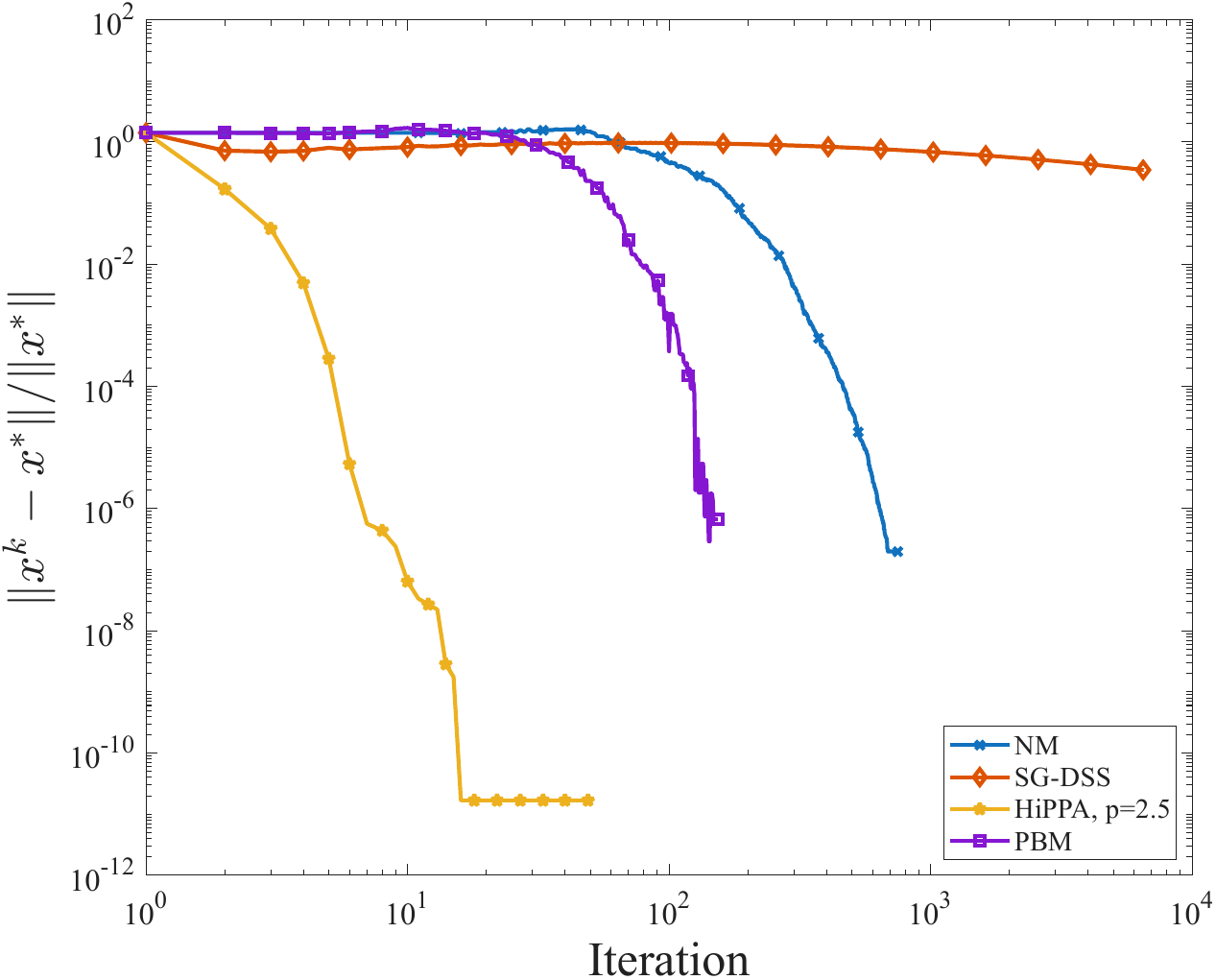}
        \caption{Comparison of relative errors for algorithms}
    \end{subfigure}
    \vspace{1em}
    
    % Second row
    \begin{subfigure}{0.42\textwidth}
        \centering
        \includegraphics[width=1.15\textwidth]{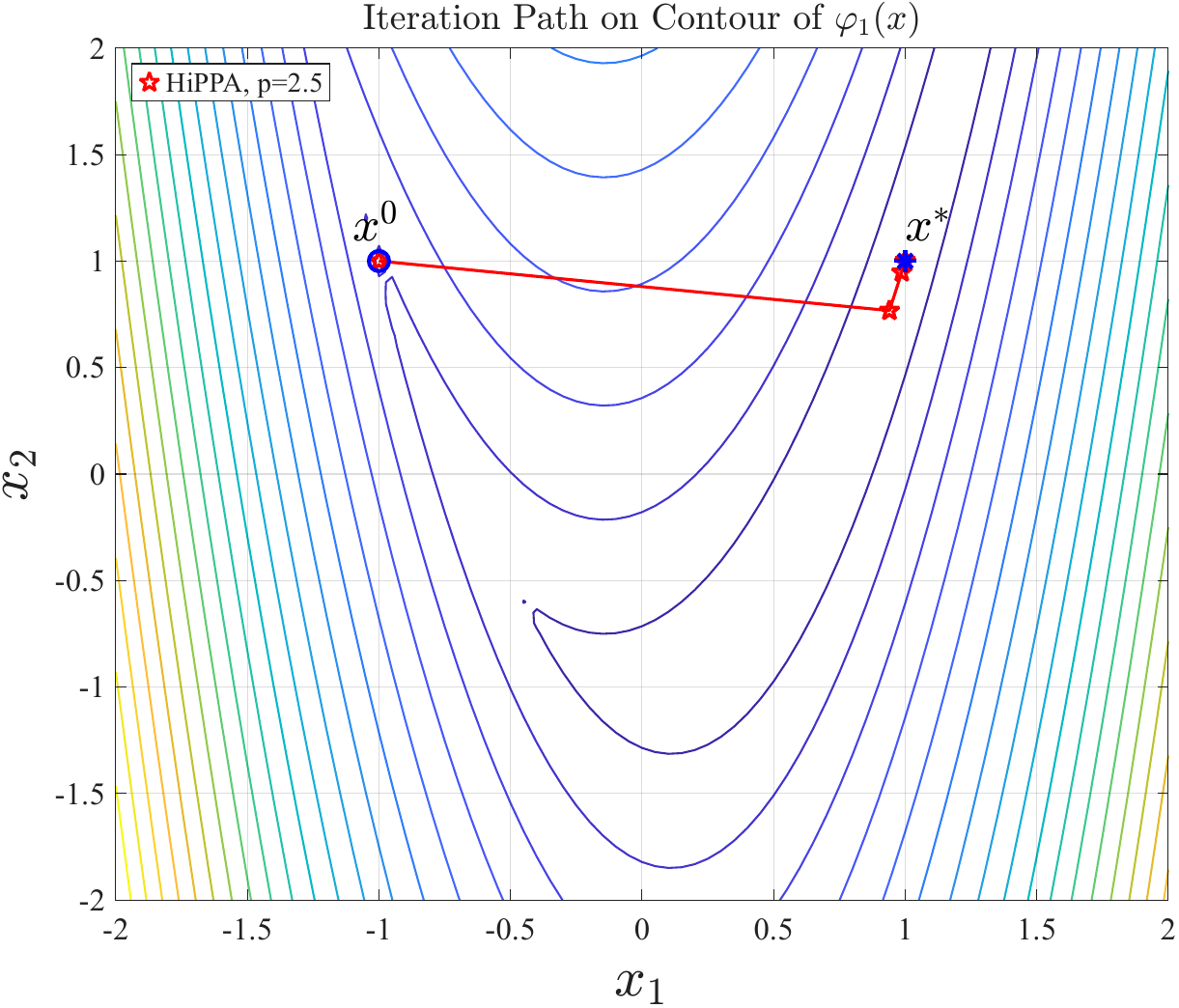}
        \caption{Trajectory of HiPPA, $p=2.5$}
    \end{subfigure}
    \qquad\qquad~~
    \begin{subfigure}{0.42\textwidth}
        \centering
        \includegraphics[width=1.15\textwidth]{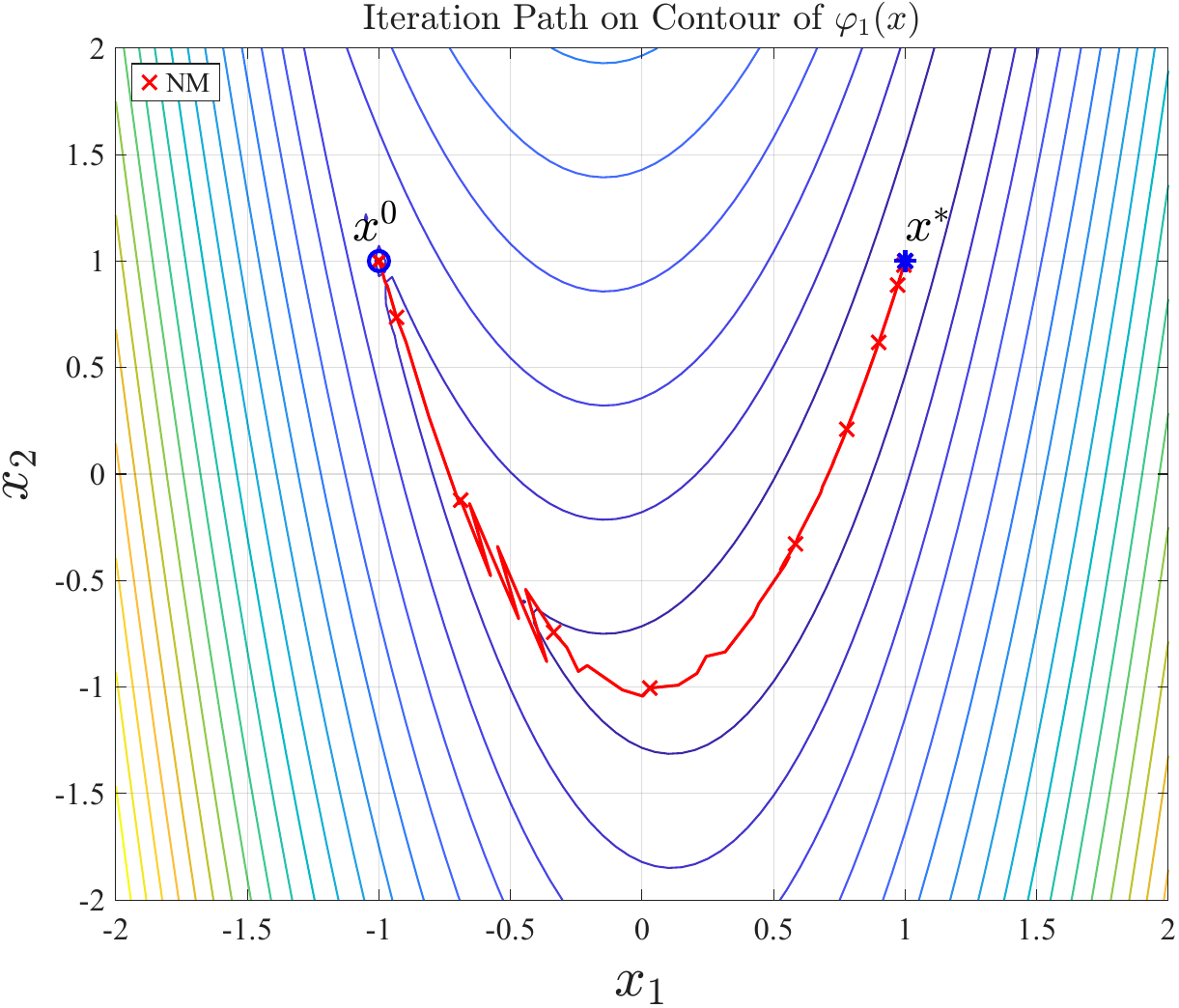}
        \caption{Trajectory of NM}
    \end{subfigure}
    \vspace{1em}
  % th row
    \begin{subfigure}{0.42\textwidth}
        \centering
        \includegraphics[width=1.15\textwidth]{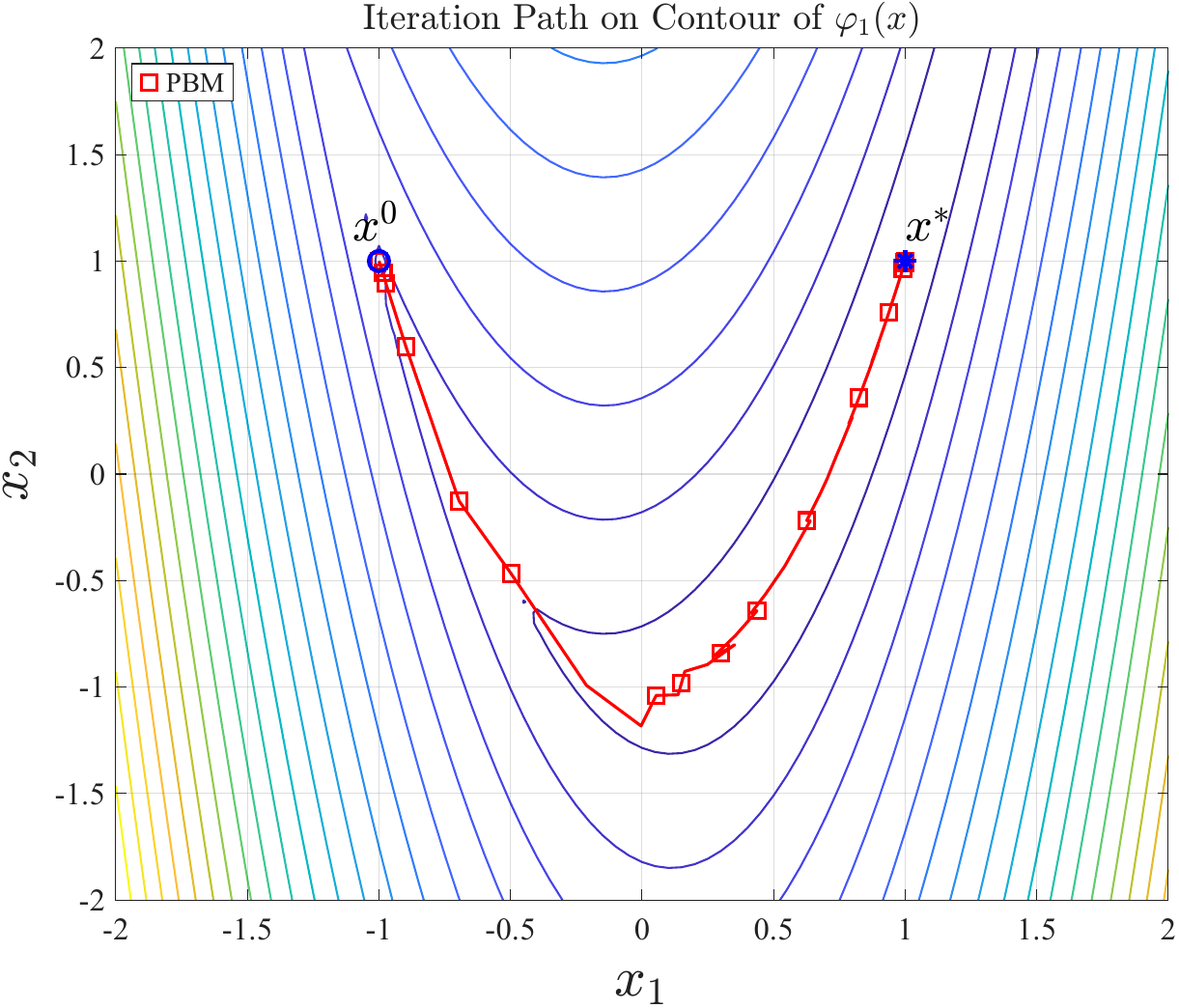}
        \caption{Trajectory of PBM}
    \end{subfigure}
         \qquad\qquad~~
    \begin{subfigure}{0.42\textwidth}
        \centering
        \includegraphics[width=1.15\textwidth]{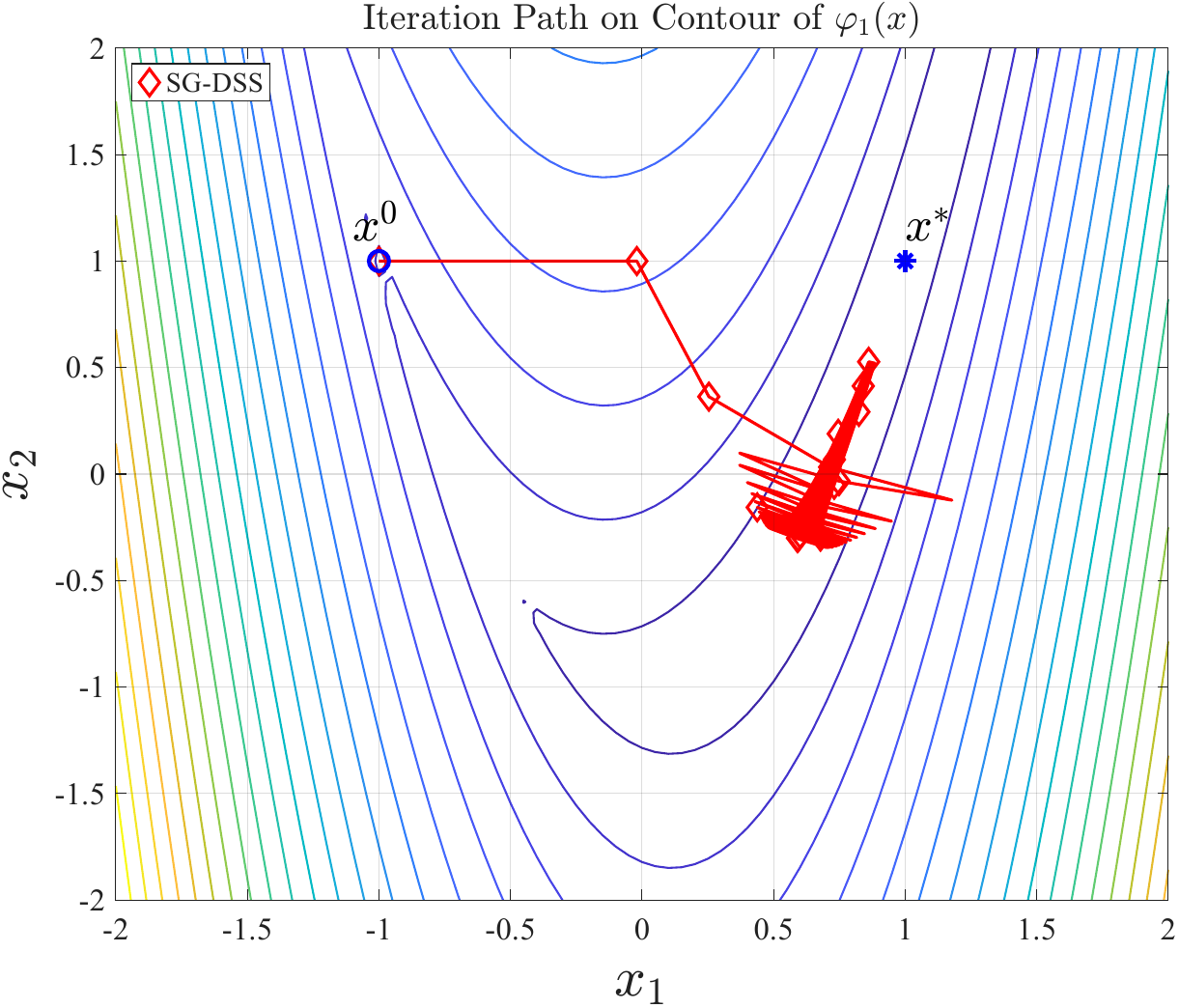}
        \caption{Trajectory of SG-DSS}
    \end{subfigure}
    \vspace{1em}
    \caption{The trajectories and relative errors in Example~\ref{ex:nes1}
            versus iterations for HiPPA ($p=2.5$), SG-DSS, NM, and PBM, starting from $x^0=(-1,1)$. }
    \label{fig:hippa:N1}
\end{figure}
%%%%%%%%%%%%%%%
\begin{example}[Nesterov-Chebyshev–Rosenbrock function II]\label{ex:nes2}
Let $\gf_2: \R^2\to \R$ be a function given by
\[
\gf_2(x)= \frac{1}{4}\left\vert x_1-1\right\vert+\left\vert x_2-2\vert x_1\vert +1\right\vert. 
\]
which attains its unique global minimizer at $x^*=(1,1)$ and $\bs{\rm Mcrit}(\gf_2)=\{x^*\}$.
This function also has another critical point $\ov{x}=(0,-1)$ in the sense of Clarke subdifferential, i.e.,
$0\in \partial^{\circ}\gf_2(\ov{x})$ which is the \textit{convex hull} of
$\partial \gf_2(\ov{x})$. As reported in \cite{GURBUZBALABAN2012Nest}, several algorithms, including BFGS and the gradient sampling algorithm \cite{burke2005robust}, may converge to $\ov{x}$. However, Fact~\ref{prop:relcrit} and Theorem~\ref{th:HiPPA} establish that every cluster point of a sequence generated by HiPPA is a Mordukhovich critical point and, in this example, the unique global minimizer.

In the first part of this example, we compare HiPPA with different values of $p\in \{1.25, 1.5, 1.75, 2, 2.5, 3, 3.5\}$ and $\gamma=9.1$, obtained by tuning, and with NM, SG-DSS (with $\alpha = 0.98$), and PBM, starting from $x^0=(-1,1)$. In the second part, we show that there exist initial points for which a variant of the subgradient method converges to the nonoptimal Clarke critical point $\ov{x}$, whereas HiPPA always converges to the optimal solution $x^*$.

As illustrated in Subfigure~\ref{fig:hippa:N2}(a),
HiPPA performs well for almost all values of $p < 3$. For comparison, we select $p=1.25$ as a representative case.
 Subfigure~\ref{fig:hippa:N2}(b) shows that in this examples, HiPPA, NM, and PBM achieve substantially smaller relative errors than SG-DSS,  with HiPPA being slightly more accurate than NM and PBM. The trajectories shown in Subfigures~\ref{fig:hippa:N2}(c)–\ref{fig:hippa:N2}(f) further confirm that HiPPA, NM, and PBM efficiently approach $x^*$, whereas SG–DSS progresses more slowly due to zigzagging behavior.

Next, we analyze convergence toward the nonoptimal Clarke critical point $\ov{x}=(0,-1)$. To this end, we recall the subgradient method with geometric step sizes (SG–GSS), defined as
\[
    x^{k+1}=x^k - \alpha_k\frac{\zeta_k}{\Vert\zeta_k\Vert},~~\zeta_k\in \partial^{\circ}\gf_2(x^k),
\]
with $\alpha_k=0.98^k$, which obtains by tuning. 

From our experiments with 10,000 randomly chosen initial points, we observed that PBM and SG–DSS (with $\alpha=0.98$) consistently converge to $x^*$. By contrast, convergence of SG–GSS to the optimal solution can be problematic. As illustrated in Figure~\ref{fig:hippa:N2rand}, when initialized at $x^{0}= (-2.48,0.58)$, SG–GSS generates a sequence that converges to the nonoptimal Clarke critical point $\ov{x}$, whereas HiPPA reliably converges to the global minimizer $x^*$.
\end{example}

 \begin{figure}
    \centering
    % First row
      \begin{subfigure}{0.42\textwidth}
        \centering
        \includegraphics[width=1.15\textwidth]{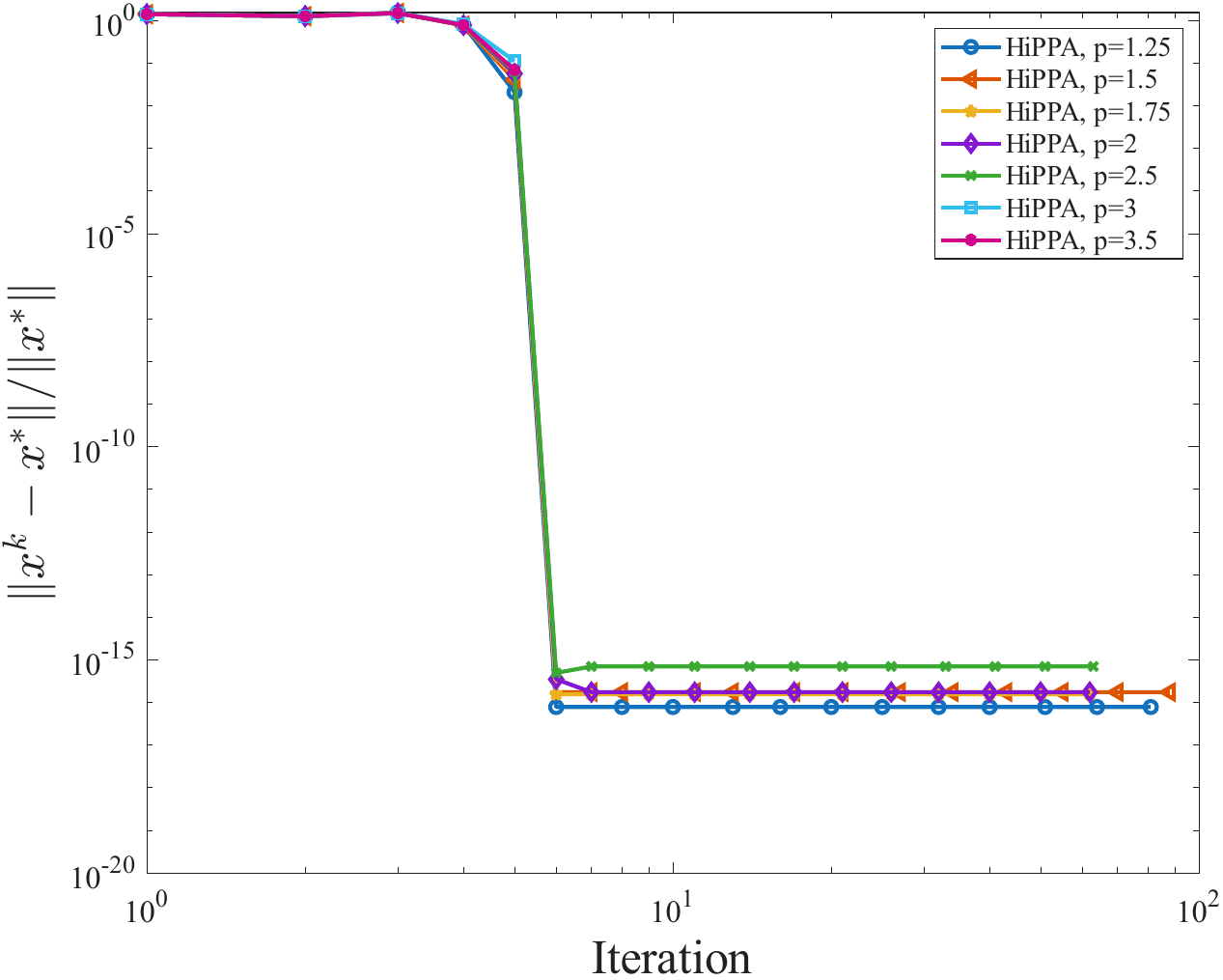}
        \caption{Comparison of relative errors of HiPPA}
    \end{subfigure}
         \qquad\qquad~~
    \begin{subfigure}{0.42\textwidth}
        \centering
        \includegraphics[width=1.15\textwidth]{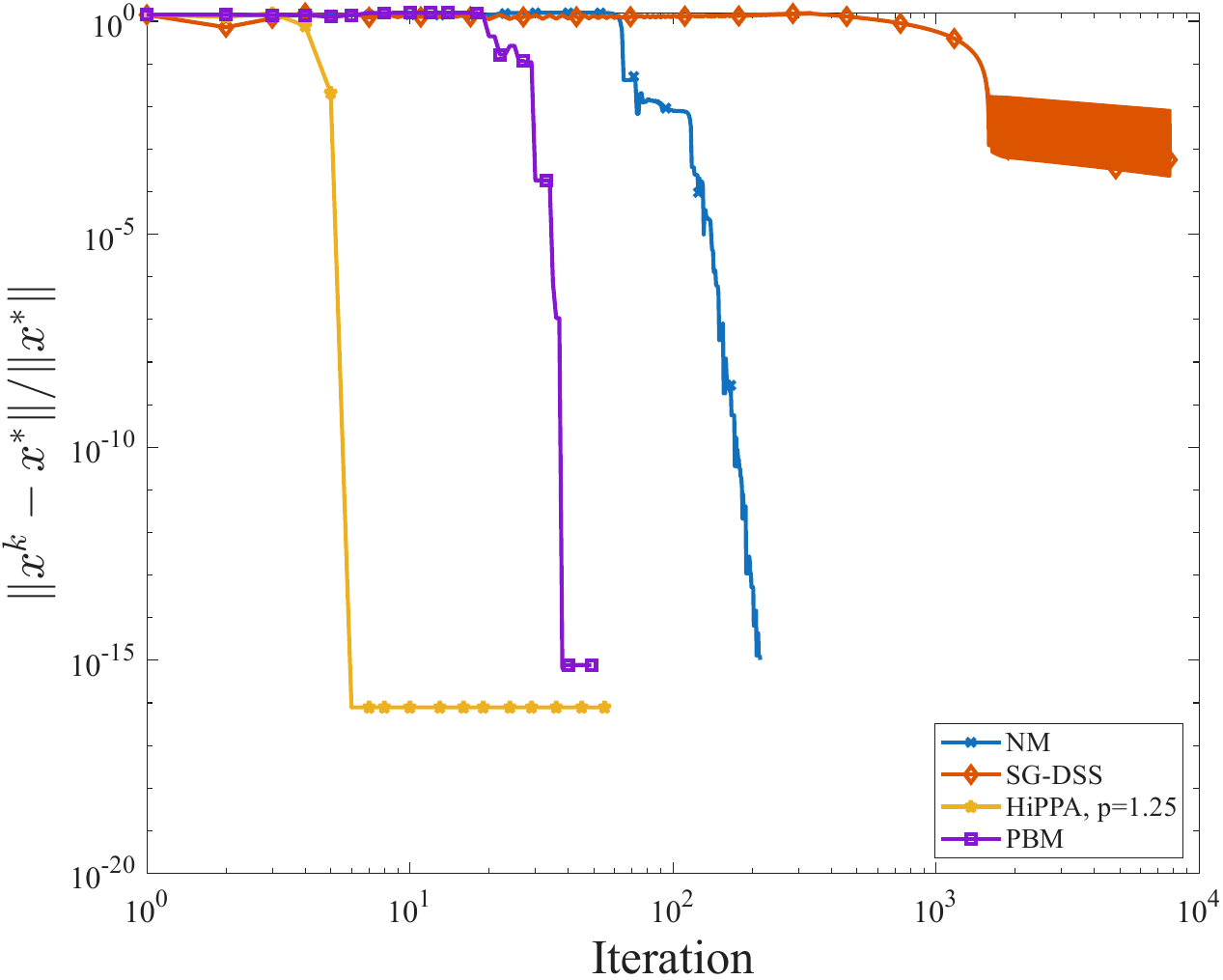}
        \caption{Comparison of relative errors for algorithms}
    \end{subfigure}
    \vspace{1em}
    
    % Second row
    \begin{subfigure}{0.42\textwidth}
        \centering
        \includegraphics[width=1.15\textwidth]{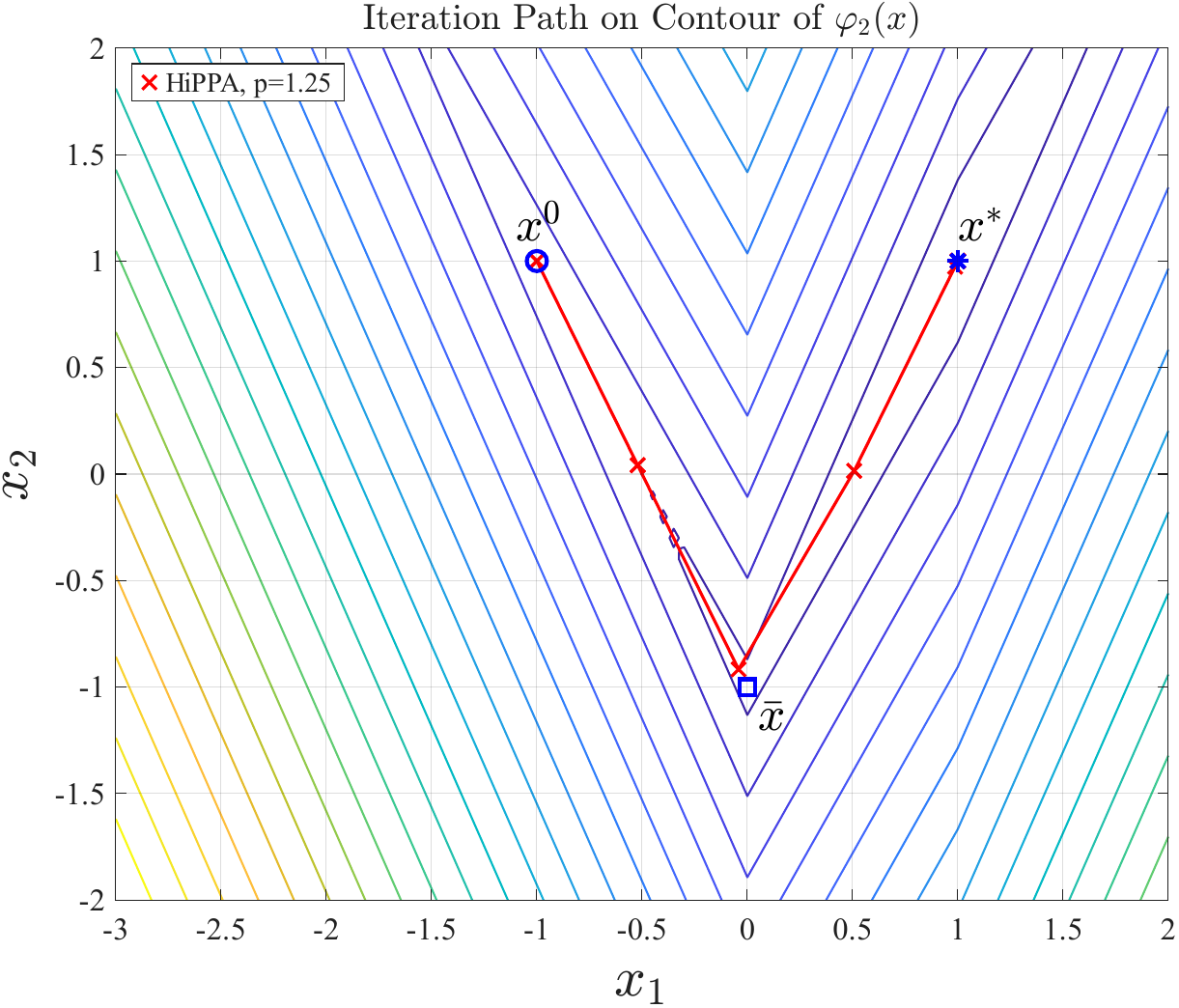}
        \caption{Trajectory of HiPPA, $p=1.25$}
    \end{subfigure}
    \qquad\qquad~~
    \begin{subfigure}{0.42\textwidth}
        \centering
        \includegraphics[width=1.15\textwidth]{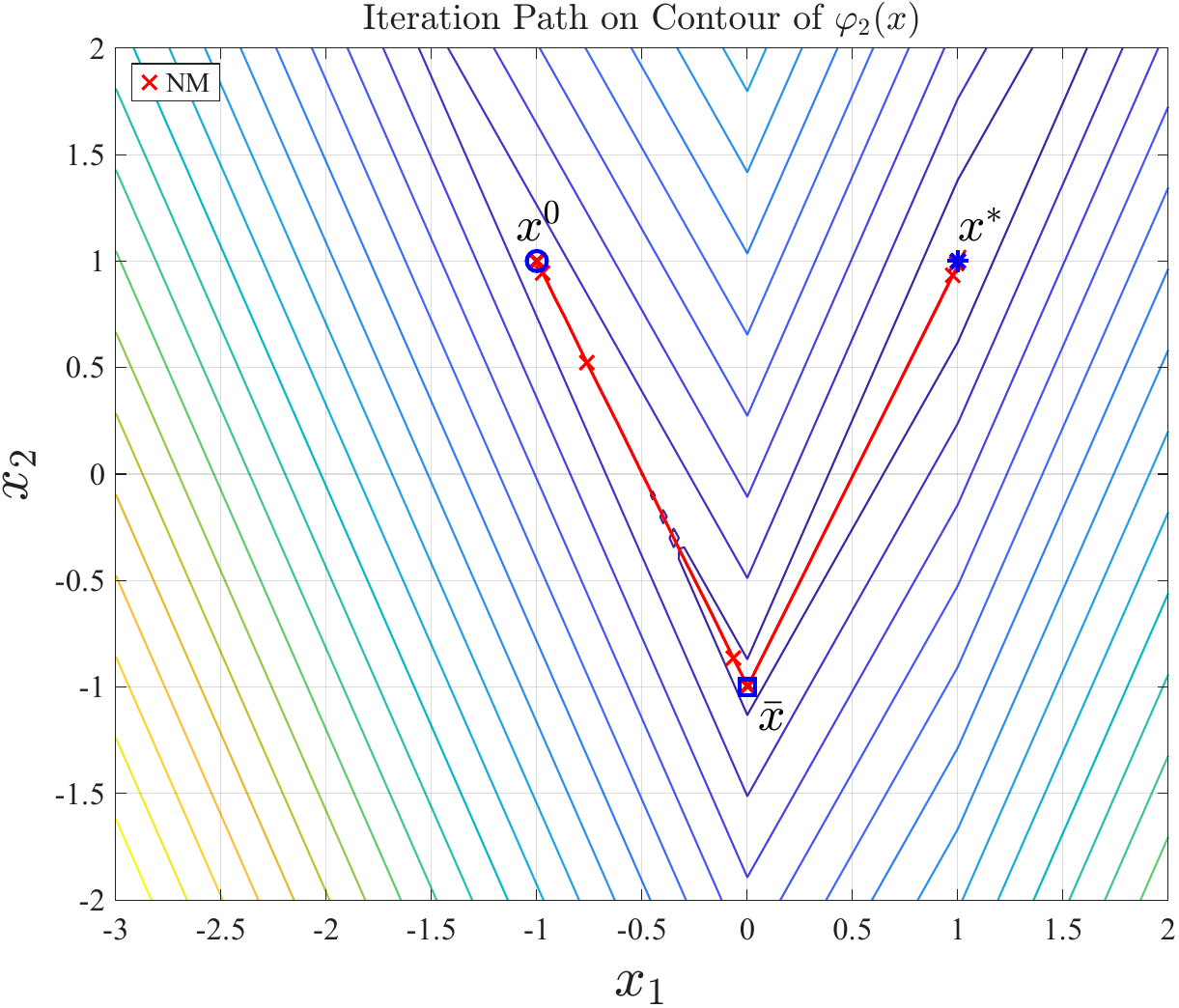}
        \caption{Trajectory of NM}
    \end{subfigure}
    \vspace{1em}
  % th row
    \begin{subfigure}{0.42\textwidth}
        \centering
        \includegraphics[width=1.15\textwidth]{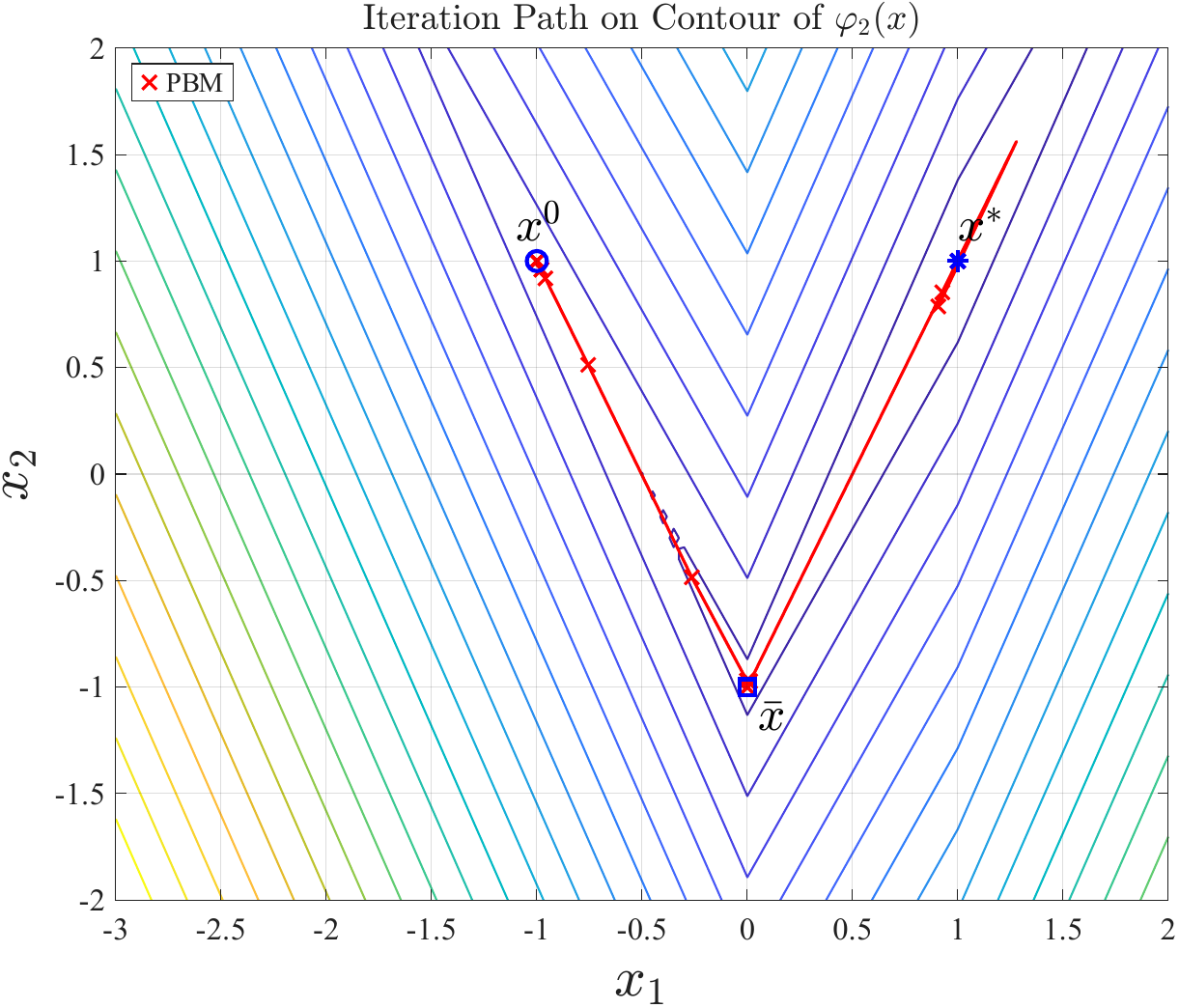}
        \caption{Trajectory of PBM}
    \end{subfigure}
         \qquad\qquad~~
    \begin{subfigure}{0.42\textwidth}
        \centering
        \includegraphics[width=1.15\textwidth]{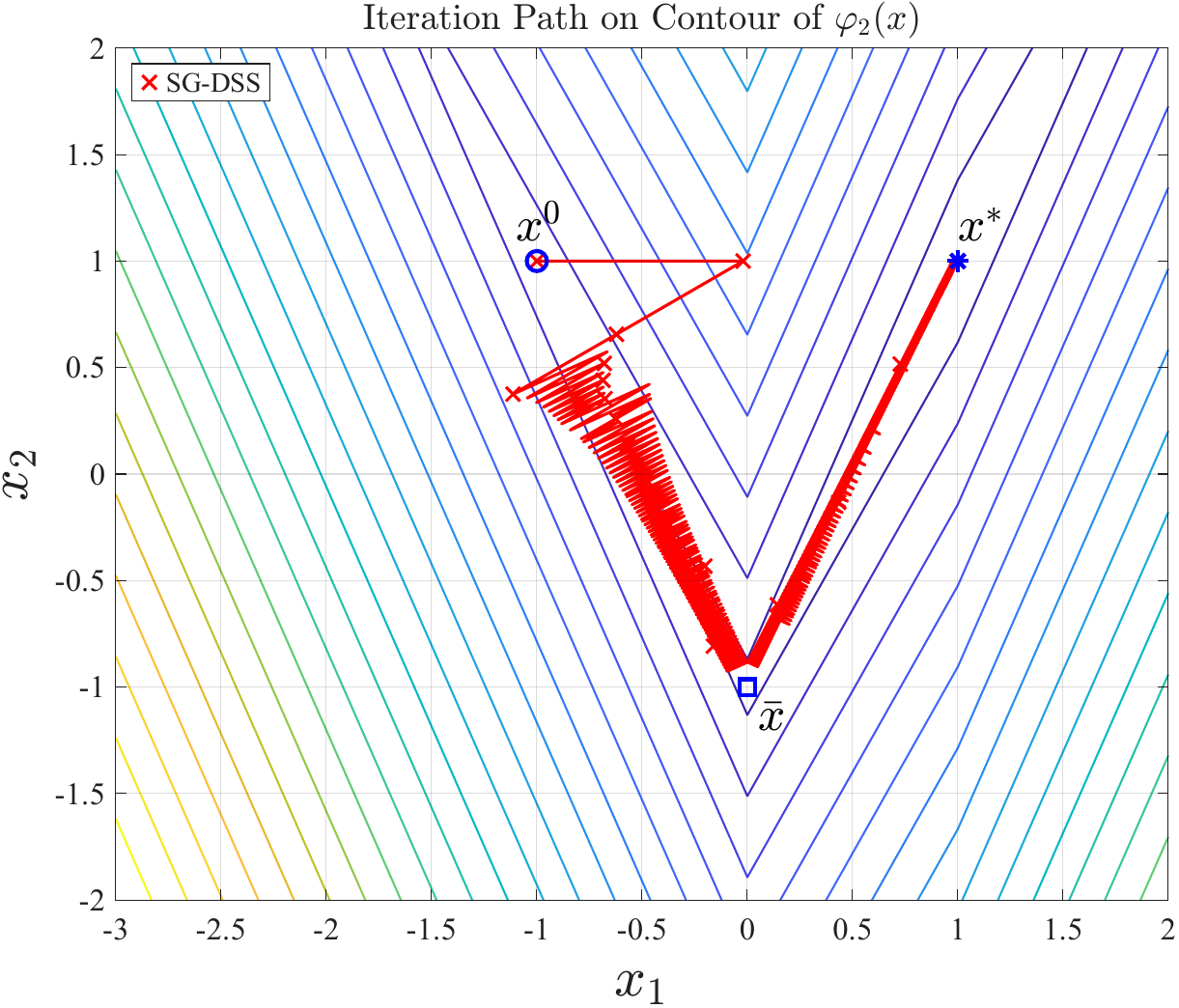}
        \caption{Trajectory of SG-DSS}
    \end{subfigure}
    \vspace{1em}
    \caption{The trajectories and relative errors in Example~\ref{ex:nes2}
            versus iterations for HiPPA ($p=1.25$), SG-DSS, NM, and PBM, starting from $x^0=(-1,1)$. }    \label{fig:hippa:N2}
\end{figure}
 %%%%%%%%%%%%%%%%%%%%%%
 
  \begin{figure}
    \centering
    % First row
      \begin{subfigure}{0.42\textwidth}
        \centering
        \includegraphics[width=1.15\textwidth]{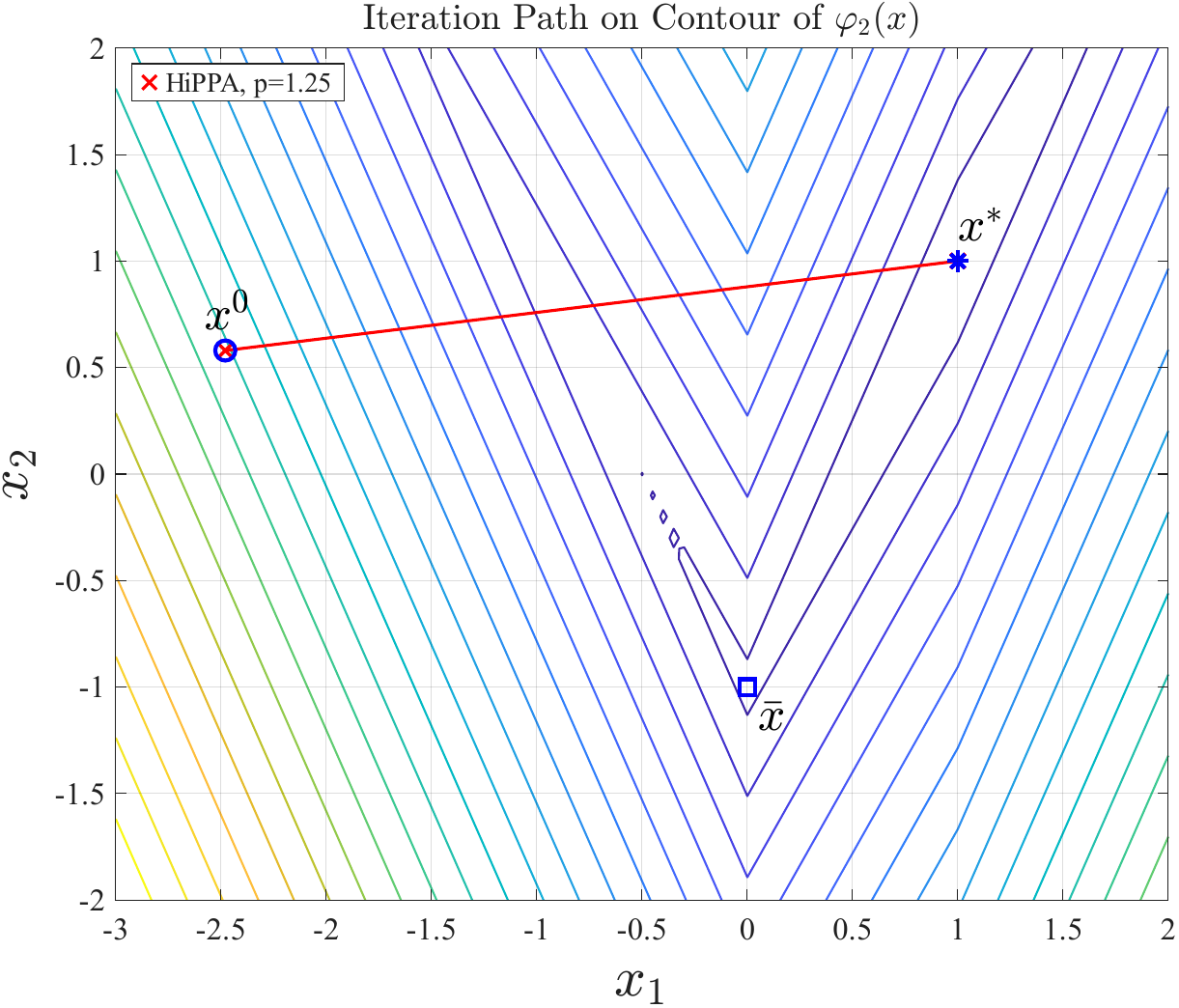}
        \caption{Trajectory of HiPPA, $p=1.25$}
    \end{subfigure}
         \qquad\qquad~~
    \begin{subfigure}{0.42\textwidth}
        \centering
        \includegraphics[width=1.15\textwidth]{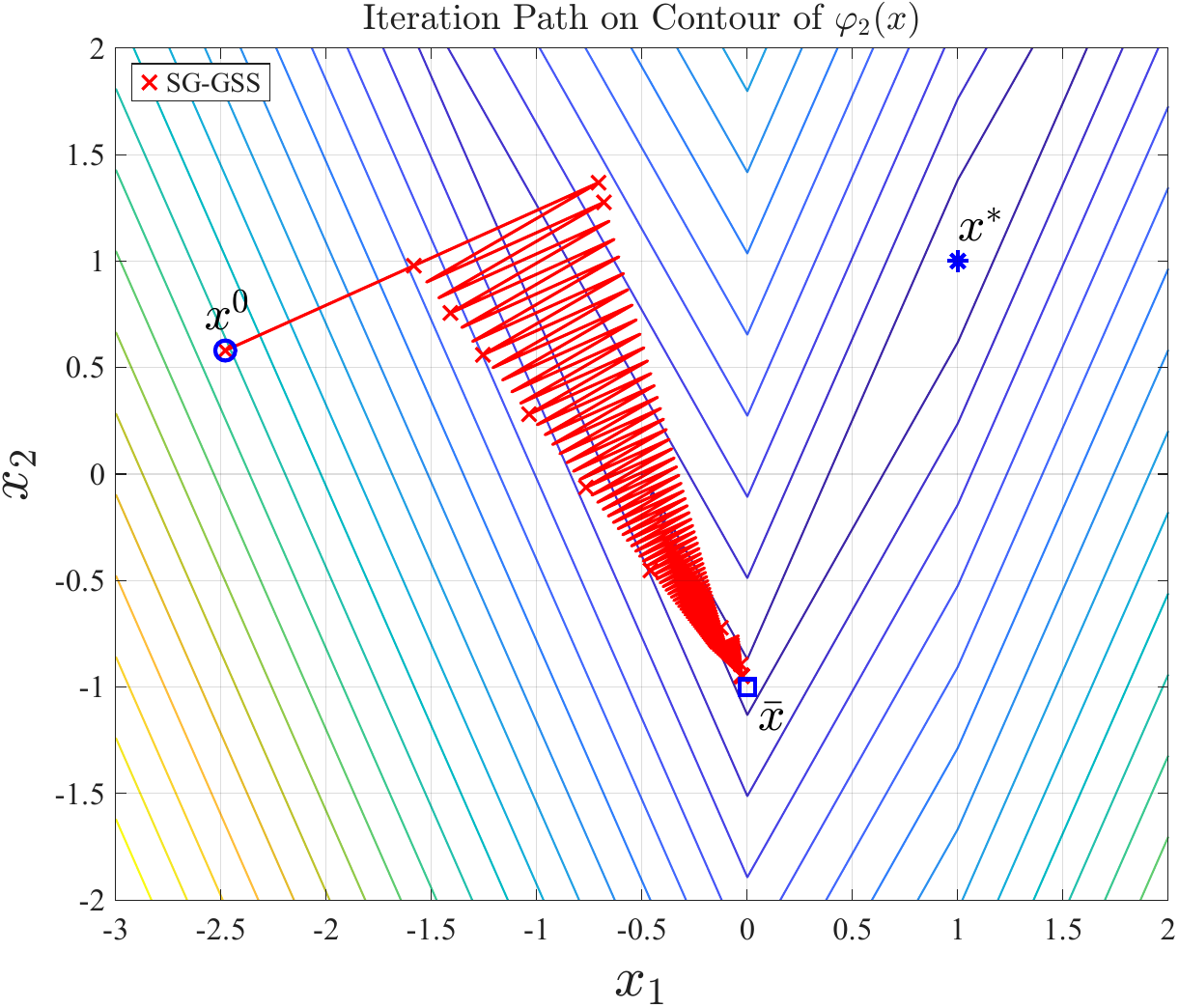}
        \caption{Trajectory of SG-GSS}
    \end{subfigure}
    \vspace{1em}

    \caption{The trajectories in Example~\ref{ex:nes2}
            versus iterations for SG-GSS and HiPPA with $p=1.25$, starting from $x^{0}= (-2.48,0.58)$. }
    \label{fig:hippa:N2rand}
\end{figure}

%%%%%%%%%%%%%%%%%%%%%%%%%%%%%%%%%%%%%%%%%%%%%%%%%%%%%%%%%%%%%%%%%%%%%%%%%%%%
%%%%%%%%%%%%%%%%%%%%%%%%%%%%%%%%%%%%%%%%%%%%%%%%%%%%%%%%%%%%%%%%%%%%%%%%%%%%
\section{Discussion}\label{sec:disc}

This study advances the understanding and application of high-order regularization in nonsmooth and nonconvex optimization through a comprehensive analysis of the high-order Moreau envelope (HOME) and its associated proximal operator (HOPE). We established that HOME exhibits continuous differentiability under $q$-prox-regularity ($q \geq 2$) and $p$-calmness for $p \in (1, 2]$ or $2 \leq p \leq q$, together with weak smoothness with H\"{o}lder continuous gradient under broader conditions. We emphasize that for $q<p$, the differential properties of HOME are still unresolved, as summarized in Subfigure~(a) of Figure~\ref{fig:qpdiff}, which can be a matter of a future work. Overall, the properties detailed in Section~\ref{sec:on diff} generalize the classical results beyond the quadratic case ($p = 2$) and establishes a more flexible theoretical foundation for applying smoothing techniques in nonsmooth and nonconvex optimization; see, e.g., \cite{Kabgani24itsopt,Kabganitechadaptive}.

Building on these insights, we proposed the high-order proximal-point algorithm (HiPPA), demonstrating its subsequential convergence to proximal fixed points, which are also $p$-calm and Mordukhovich critical points, without restrictive assumptions on $\gamma$ (Section~\ref{sec:applications}). Our preliminarily numerical experiments with the Nesterov-Chebyshev-Rosenbrock (NCR) functions further validated HiPPA's efficiency. For the first NCR function, higher $p$ (e.g., $2.5$) outperformed higher values, leveraging broader steps to navigate a milder landscape, while for the second NCR function, lower $p$ (e.g., $1.25$) excelled, precisely escaping sharper features and critical point traps. For both example, HiPPA outperformed the subgradient method SG-DSS considerably.

On the basis of these preliminary results, it is conjectured that: (i) if the critical points are located in a ``steep valley", then it is better to use a bigger regularization parameter $p$; (ii) if the critical points are located in a ``wide valley", then it is better to use a smaller regularization parameter $p$.
These results highlight HiPPA's adaptability, with the choice of $p$ tailoring its performance to the problem's geometrical structure. Future work could explore adaptive $p$ and $\gamma$ selection strategies, potentially enhancing efficiency across diverse optimization methods, and extending these our foundations to broader classes of nonconvex problems.

\section{Statements and Declarations}

\subsection{Funding}
The Research Foundation Flanders (FWO) research project G081222N and UA BOF DocPRO4 projects with ID 46929 and 48996 partially supported the paper's authors.
\subsection{Conflicts of interest/Competing interests}
There are no conflicts of interest or competing interests related to this manuscript.

\subsection{Availability of data and material}
Not applicable.
%
%%The manuscript is self-contained, no additional data and material were used during production.
%
%\subsection{Code availability}
%
%Not applicable.

\subsection{Authors' contributions}
Both authors contributed equally to the study conception and design.

%\begin{acknowledgements}

%\end{acknowledgements}

%%%%%%%%%%%%%%%%%%%%%%%%%%%%%%%%%%%%%%%%%%%%%%%%%%%%%%%%%%%%%%%%%%%%%%%
%%%%%%%%%%%%%%%%%%%%%%%%%%%%%%%%%%%%%%%%%%%%%%%%%%%%%%%%%%%%%%%%%%%%%%%
% \appendix
% \section{Deferred proofs}
% \label{sec:app}

\bibliographystyle{spbasic}
\bibliography{references}

\end{document}